\documentclass[10pt,electronic]{amsart}        
       
\usepackage{baskervald}

\usepackage[swedish, english]{babel}
\usepackage[T1]{fontenc}
\usepackage[latin1]{inputenc}
\usepackage{amsmath}
\usepackage{xcolor}
\usepackage{enumerate}
\usepackage{subcaption}
\usepackage[pdfusetitle]{hyperref}
\usepackage[noabbrev,capitalize]{cleveref}
\hypersetup{
	colorlinks,
	linkcolor={red!50!black},
	citecolor={blue!50!black},
	urlcolor={blue!80!black}
}
\usepackage{comment}

\usepackage{amsthm}                      
\usepackage{amssymb} 
\usepackage{mathrsfs}
\usepackage{stmaryrd}
\usepackage{mathtools}
\usepackage[shortlabels]{enumitem}
\usepackage{tikz-cd}
\usetikzlibrary{babel}
\usepackage{euscript}
\renewcommand{\mathcal}{\EuScript}

\usepackage[hmargin=3.3cm,vmargin=3.5cm]{geometry}

\author{Jeremy Miller}
\author{Peter Patzt}
\author{Dan Petersen}
\author{Oscar Randal-Williams}

\title{Uniform twisted homological stability}

\usepackage{todonotes}

\newcommand\mnote[1]{\marginpar{\tiny #1}}

\newcommand{\Q}{\mathbb Q}
\newcommand{\Z}{\mathbb Z}
\newcommand{\R}{\mathbb R}
\newcommand{\Fq}{\mathbb{F}_q}
\newcommand{\PB}{\mathit{PB}}
\newcommand{\IX}{\mathit{IX}}
\newcommand{\Lk}{\mathrm{Lk}}
\newcommand{\Sp}{\mathrm{Sp}}
\newcommand{\GL}{\mathrm{GL}}
\usepackage{bbm}
\newcommand{\bk}{\mathbbm{k}}
\newcommand{\N}{\mathbb N}
\newcommand\lra{\longrightarrow}
\newcommand{\bL}{\mathbb{L}}
\newcommand\colim{\operatorname*{colim}}
\newcommand\hocolim{\operatorname*{hocolim}}

\newcommand{\offset}{\tau}

\newtheorem{thm}{Theorem}[section]
\newtheorem{cor}[thm]{Corollary}
\newtheorem{lem}[thm]{Lemma}
\newtheorem{prop}[thm]{Proposition}

\theoremstyle{definition}
\newtheorem{defn}[thm]{Definition}
\newtheorem{notation}[thm]{Notation}
\theoremstyle{remark}
\newtheorem{rem}[thm]{Remark}
\newtheorem{example}[thm]{Example}
\numberwithin{equation}{section}

\renewcommand{\mnote}[1]{{\iffalse #1 \fi}}

\begin{document}

\begin{abstract}
    We prove a homological stability theorem for families of discrete groups (e.g.~mapping class groups, automorphism groups of free groups, braid groups) with coefficients in a sequence of irreducible algebraic representations of arithmetic groups. The novelty is that the stable range does not depend on the degree of polynomiality of the coefficient system. Combined with earlier work of Bergstr\"om--Diaconu--Petersen--Westerland this proves the Conrey--Farmer--Keating--Rubinstein--Snaith predictions for all moments of the family of quadratic $L$-functions over function fields, for sufficiently large odd prime powers.
\end{abstract}
 \maketitle

\tableofcontents

	\section{Introduction}
	
	\label{section1}

\subsection{Four instances of the main theorem}

Recent work of \cite{BDPW} investigated connections between moments of families of quadratic $L$-functions over rational function fields, and stable homology of braid groups with twisted coefficients, proving in particular that an improved range of homological stability for the braid groups with these coefficients would imply a certain asymptotic formula for all moments. The goal of this paper is to prove such an improved stable range. In fact, the main result of the paper will be an abstract stability theorem applicable to multiple families of groups and coefficients. But rather than try to state the abstract form of the theorem in the introduction, we will give four important special cases. The hypotheses of the theorem are such that for all families of groups and coefficients to which the theorem applies, existing theorems in the literature already imply homological stability. What is new is the \emph{shape} of the stable range. 

It will be convenient to formulate our homological stability theorems in terms of vanishing of relative homology. Recall that if $H \to G$ is a homomorphism, then vanishing of the relative homology group $H_d(G,H;\Z)$ for all $d \leq n$ means exactly that $H_n(H;\Z)\to H_n(G;\Z)$ is surjective, and $H_{d}(H;\Z)\to H_{d}(G;\Z)$ is an isomorphism for $d<n$. The analogous statement is also true for homology with other coefficients.

\subsubsection{Mapping class groups}

Let $\mathrm{Mod}_g^1$ be the mapping class group of an oriented genus $g$ surface with a boundary component, and consider the natural map $\smash{\mathrm{Mod}_g^1 \to \mathrm{Sp}_{2g}(\Z)}$. A partition $\lambda$ is a descending integer sequence $\lambda_1 \geq \lambda_2 \geq \dots$ which eventually is constantly zero; we write $l(\lambda) := \max\{ i : \lambda_i \neq 0\}$ and $|\lambda| := \sum_i \lambda_i$ for the \emph{length} and \emph{size} of a partition. Irreducible algebraic representations of the symplectic group $\mathrm{Sp}_{2g}$ are indexed by their highest weight, which is a partition $\lambda$ of length $l(\lambda)\leq g$. For a partition $\lambda$, let $V_\lambda(2g)$ be the irreducible rational representation of $\mathrm{Sp}_{2g}$ of highest weight $\lambda$, if $l(\lambda) \leq g$, and set $V_\lambda(2g)=0$ if $l(\lambda)>g$. We have a stabilisation map $\smash{\mathrm{Mod}_{g-1}^1 \to \mathrm{Mod}_g^1}$ given by gluing on a torus with two boundary components. 

Harer \cite{harerstability} proved homological stability for the mapping class groups. Subsequent improvements were given in \cite{ivanovtwisted,boldsen}, where stability was also proven for \emph{polynomial coefficient systems}, such as $V_\lambda$ (which is polynomial of degree $|\lambda|$).	We prove the following stability theorem in this paper:

\begin{thm}\label{thmA} The group $H_d(\mathrm{Mod}_g^1,\mathrm{Mod}_{g-1}^1; V_\lambda(2g),V_\lambda(2g-2))$ vanishes for $d<\tfrac{1}{2}(g-\tfrac 1 2 l(\lambda))$.
\end{thm}

The best stable range known previously for these coefficients \cite[Example 5.22]{gkrw-secondary} is that these groups vanish for $ d < \tfrac{2}{3}(g-\vert\lambda\vert)$. Hence we prove a stable range with a worse slope, but with a much better dependence on $\lambda$. In particular, in the regime $\vert\lambda\vert \gg g$ all previously known bounds gave no stable range at all. 

Note in particular that if $l(\lambda)>g$, then $V_\lambda(2g)=V_\lambda(2g-2)=0$, so the vanishing in \cref{thmA} holds for entirely trivial reasons. We therefore deduce that $H_d(\mathrm{Mod}_g^1,\mathrm{Mod}_{g-1}^1; V_\lambda(2g),V_\lambda(2g-2))=0$ for $d<\tfrac{1}{4}g$, for \emph{any} coefficient system $V_\lambda$, which is now a \emph{uniform} stable range.

	\subsubsection{Automorphism groups of free groups} Let $\mathrm{Aut}(F_n)$ be the automorphism group of a free group on $n$ generators, and consider the natural map $\mathrm{Aut}(F_n)\to\mathrm{GL}_n(\Z)$ given by the action of $\mathrm{Aut}(F_n)$ on its abelianisation. We index irreducible algebraic representations of $\mathrm{GL}_n$ by pairs of partitions $(\lambda,\mu)$ such that $l(\lambda)+l(\mu)\leq n$. Let $W_{\lambda\mu}(n)$ be the irreducible rational representation of $\mathrm{GL}_{n}$ associated to $(\lambda, \mu)$ if $l(\lambda)+l(\mu) \leq n$, and set $W_{\lambda\mu}(n)=0$ otherwise. 
	
	Homological stability of $\mathrm{Aut}(F_n)$ is originally due to Hatcher and Hatcher--Vogtmann \cite{Hatcher95,hatchervogtmann1,hatchervogtmann2}, and can likewise be extended to stability with polynomial coefficients \cite{randalwilliamswahl} such as $W_{\lambda\mu}$, which is polynomial of degree $|\lambda|+|\mu|$. We prove the following.
	
	\begin{thm}\label{thmB}
		The group $H_d(\mathrm{Aut}(F_n),\mathrm{Aut}(F_{n-1});W_{\lambda\mu}(n),W_{\lambda\mu}(n-1))$ vanishes for $d < \tfrac 1 2 (n-\beta)$, where $\beta = \max(\tfrac 1 2 (l(\lambda)+l(\mu)),1+\lambda_2',1+\mu_2')$, where $\lambda'$ and $\mu'$ denote the conjugate partitions of $\lambda$ and $\mu$. 
	\end{thm}

The best range previously in the literature is that these groups vanish for 
$d< \tfrac 1 2 (n-1-\vert\lambda\vert - \vert \mu \vert)$ \cite[Theorem 5.4]{randalwilliamswahl}.\footnote{This can in fact be improved to $d < \tfrac 4 5(n - \vert\lambda\vert - \vert\mu\vert)$. From \cite[Theorem 4.8]{millerpatztpetersen} such a result would follow from a good enough vanishing range of relative homology with constant coefficients. The required stable range with rational coefficients follows from \cite[Proposition 1.2]{hatchervogtmann4} combined with Galatius' theorem \cite{galatius} which tautologically gives surjectivity of the stabilisation map one degree above. Moreover, it is a folk theorem that the slope $\tfrac 4 5$ of \cite{hatchervogtmann4} can be improved to at least $\tfrac 7 8$.}

At first, \cref{thmB} might not seem to provide a uniform stable range --- unlike \cref{thmA} --- since the quantity $n-\beta$ may very well be zero. But in a sense it does. The point is that since every matrix in $\GL_n(\Z)$ has determinant $\pm 1$, we can twist $W_{\lambda\mu}(n)$ by any even power of the determinant, and get an isomorphic representation of $\GL_n(\Z)$. Hence for a given irreducible representation $V$ of $\GL_n$, there are infinitely many pairs of partitions $(\lambda,\mu)$ such that $V \cong W_{\lambda\mu}(n)$ as representations of $\GL_n(\Z)$. Since the stable homologies $H_\ast(\mathrm{Aut}(F_\infty);W_{\lambda\mu})$ for these coefficients are typically not all isomorphic, we should not expect all such $(\lambda,\mu)$ to admit the same stable range. The sense in which \cref{thmB} \emph{is} uniform is that for every $V$ as above there exists some pair $(\lambda,\mu)$ with $W_{\lambda\mu}(n) \cong V$ as representations of $\GL_n(\Z)$, and with $\beta \leq \tfrac 1 2(n+2)$. This makes \cref{thmB} a uniform stability theorem of slope $\tfrac 1 4$. We elaborate on this in Subsection \ref{uniform GL}. 

It is instructive to consider the case that $\lambda = (2^{n})$ and $\mu=(0)$, which gives $W_{\lambda\mu}(n) \cong \det^{\otimes 2}$, which is the trivial representation of $\GL_n(\Z)$, and  $W_{\lambda\mu}(n-1)=0$. Hence in this case already $H_0(\mathrm{Aut}(F_n),\mathrm{Aut}(F_{n-1});W_{\lambda\mu}(n),W_{\lambda\mu}(n-1)) \neq 0$.  This is consistent with \cref{thmB} and that $\beta=n+1$ in this case.
	
	\subsubsection{Handlebody mapping class groups}
	
	Let $\mathrm{HMod}_g^1$ be the mapping class group of a genus $g$ handlebody with a marked disk on its boundary. Homological stability for these groups was proven by Hatcher--Wahl \cite{hatcherwahl}. There is a natural map $\smash{\mathrm{HMod}_g^1 \to \mathrm{Aut}(F_g)}$, since $\mathrm{Aut}(F_g)$ is the homotopy automorphism group of the handlebody. We may thus pull back the coefficients $W_{\lambda\mu}(g)$ considered above to $\mathrm{HMod}_g^1$. We prove the following.
	
	\begin{thm}\label{thmH} The group $H_d(\mathrm{HMod}_g^1,\mathrm{HMod}_{g-1}^1; W_{\lambda\mu}(g),W_{\lambda\mu}(g-1))$ vanishes for $d<\tfrac{1}{3}(g-\beta)$, where $\beta$ is defined as in \cref{thmB}.
		\end{thm}

       The stable range in \cref{thmH} is uniform in the same sense as the one in \cref{thmB}. 
	
	The best previously known  vanishing range is $d < \tfrac{1}{2}(g-1-\vert\lambda\vert -\vert\mu\vert)$ \cite[Theorem 5.31]{randalwilliamswahl}.

\subsubsection{Braid groups and the Burau representation}

Let $\beta_n$ be the braid group on $n$ strands, and let $\beta_{n} \to \mathrm{Sp}_{n-1}(\Z)$ be the \emph{integral Burau representation}. The integral Burau representation of the braid group on an even number of strands takes values in the ``odd symplectic groups'' of Gelfand--Zelevinsky \cite{gelfandzelevinskyodd}, as we elaborate on in the body of the paper. The coefficient systems $V_\lambda$ of the usual symplectic groups can be extended naturally to the odd symplectic groups, too. 

The group $\beta_{2g+1}$ (resp.\ $\beta_{2g+2}$) may be thought of as the \emph{hyperelliptic mapping class group} of a surface of genus $g$ with one (resp.\ two) boundary components. The hyperelliptic mapping class group sits inside the usual mapping class group, and the composite $\beta_{2g+1} \to \mathrm{Mod}_g^1 \to \mathrm{Sp}_{2g}(\Z)$ is precisely the integral Burau representation. Hence the following result may be thought of as a hyperelliptic analogue of \cref{thmA}.

\begin{thm}\label{thmC}
	The group $H_d(\beta_n,\beta_{n-1};V_\lambda(n-1),V_\lambda(n-2))$ vanishes for $d<\frac 1{17} (n-\max(2,l(\lambda))).$ 
\end{thm}

The best range in the existing literature is $d < \tfrac{1}{2}(n-\vert\lambda\vert)$ \cite[Theorem 5.22]{randalwilliamswahl}.\footnote{This can in fact be improved to $d < n-1-\vert\lambda\vert$, as follows. Let $\mathbf R$ be the free $E_2$-algebra on a point. Using the knowledge of the rational homology of the braid groups, one sees that $H_{n,d}(\overline{\mathbf{R}}/\sigma)=0$ for $d<n-1$, with notation as in \cite{GKRW}. An adaptation of \cite[Theorem 19.2]{GKRW} in order to accommodate a vanishing line not through the origin, combined with 
	\cite[Theorem 3.23]{millerpatztpetersen} and \cite[Lemma 19.4]{GKRW}, gives the result.} As in \cref{thmA} the result is interesting when $\vert\lambda\vert$ is large compared to $n$. If $l(\lambda)>\tfrac 1 2 (n+1)$ then $V_\lambda(n)=0$, and hence \cref{thmC}, too, implies a uniform stable range for all partitions $\lambda$ (which we spell out below).

The proofs of Theorems \ref{thmA}, \ref{thmB} and \ref{thmH} all use as input existing results in the literature establishing \emph{high connectivity} of certain simplicial complexes associated to the respective families of groups. By contrast, the simplicial complexes relevant for \cref{thmC} were not previously known to be highly connected, and we show this in Section \ref{connectivity section} of this paper (with a probably-not-optimal range of connectivity). These complexes are associated to the sequence of congruence subgroups of the even/odd symplectic groups, given by the image of the integral Burau representation. High connectivity of the complexes associated to the even/odd symplectic groups themselves was recently established by Sierra--Wahl \cite{sierrawahl}. The idea of studying the system of all even/odd symplectic groups together for homological stability is an algebraic analogue of the set-up of Harr--Vistrup--Wahl \cite{harrvistrupwahl}.
	
	\subsection{Applications to moments of families of \texorpdfstring{$L$}{L}-functions}
	
	\cref{thmC} has important implications in arithmetic statistics, in light of recent work of \cite{BDPW}.

    \begin{prop}
        Let $\theta(n) = \tfrac 1 {34} (n -35)$. If $d \leq \theta(n)$, then $H_d(\beta_n;V_\lambda(n-1))\to H_d(\beta_{n+1};V_\lambda(n))$ is an isomorphism. In other words, $\theta$ is a {uniform stability bound} in the sense of \cite[Definition 11.3.16]{BDPW}.
    \end{prop}

    \begin{proof}
        By \cref{thmC}, the stabilisation $H_d(\beta_n;V_\lambda(n-1))\to H_d(\beta_{n+1};V_\lambda(n))$ is an isomorphism for $d\leq \frac 1{17} (n-\max(2,l(\lambda)))-1$ (and surjective for $d\leq \frac 1{17} (n-\max(2,l(\lambda)))$.) The map is tautologically an isomorphism for $l(\lambda)>\tfrac 1 2 (n+1)$ as both coefficients vanish. Since there is nothing to prove for $n$ small we may assume $\max(2,l(\lambda)) \leq \tfrac 1 2 (n+1)$ and then it is an isomorphism for $d \leq \theta(n)$.
    \end{proof}

    Now let $q$ be an odd prime power. For monic and square-free $d \in \Fq[t]$, let $L(s,\chi_d)$ denote the $L$-function associated to the Galois representation given by the first cohomology group of the affine hyperelliptic curve with equation $y^2=d(x)$. \cref{thmC}, combined with the results of \cite{BDPW}, implies the following:
	
	\begin{thm}\label{thmD}For each $r \geq 1$, there is an explicit polynomial $Q_r$ of degree $r(r+1)/2$ such that 
		$$ q^{-2g-1}\!\!\!\!\sum_{\substack{d \in \Fq[t]\\ \text{ monic, squarefree} \\ \deg(d)=2g+1}} \!\!\!\! L(\tfrac 1 2, \chi_d)^r =  Q_r(2g+1) + \mathcal O(4^{g(r+1)}q^{-(2g-34)/(2\cdot 34)}.$$
	\end{thm}
	

	Indeed, \cite[Theorem 11.3.19]{BDPW} proves such a formula,  but with the second factor in the error term given by $q^{-\theta(2g+1)/2}$, where $\theta$ is a {uniform stability bound} in the above sense. It was conjectured in \cite{BDPW} that a nontrivial uniform stability bound exists. 

	An important class of problems in analytic number theory is to understand the distribution of the central values $L(\frac 1 2, \pi)$, as $\{L(s, \pi)\}_{\pi \in P}$ varies over some naturally occurring family of $L$-functions. Our \cref{thmD} is about the \emph{moments} in the case of the family of quadratic extensions of $\mathbb F_q(t)$. Conrey--Farmer--Keating--Rubinstein--Snaith \cite{cfkrs} have developed a ``recipe'' to predict the asymptotics of moments for large classes of families of $L$-functions. (See also \cite{DGH}.) In the situation of \cref{thmD}, the CFKRS heuristics predict that the left-hand side of \cref{thmD} is asymptotically $Q_r(2g+1)$. See \cite[Conjecture 5]{AK}, which also gives an explicit formula for $Q_r$. The conjecture was in this case known to hold for $r \leq 3$ \cite{andradekeating,florea2}, and for the highest three coefficients of $Q_4$ when $r=4$ \cite{florea3}.  \cref{thmD} shows that for every fixed $r$, the left-hand side is indeed asymptotically $Q_r(2g+1)$, with a power-saving error term, for all sufficiently large (but fixed) $q$. The fact that $q$ may be fixed is important; allowing $q \to \infty$ leads to the case treated by Katz and Sarnak \cite{KaS}. The CFKRS predictions have been proven correct in many cases before, but generally only for the first few moments (although see \cite{saw1,saw2,saw3}). \cref{thmD} represents the first\footnote{The aforementioned papers of Sawin do not treat the essential case where the individual terms are nonnegative.} nontrivial family where the CFKRS asymptotics have been established for \emph{all} moments (with the important caveat that there is no single value of $q$ which works for all $r$).
	

	In a nutshell, the connection between \cref{thmD} and homological stability theorems is as follows. Using the Grothendieck--Lefschetz trace formula, the left-hand side of \cref{thmD} may be written as  the trace of Frobenius on the homology of $\beta_{2g+1}$, with coefficients in $(\wedge V_{g}(\tfrac 1 2) ) ^{\otimes r}$. Here $V_{g}$ denotes the reduced integral Burau representation, $\wedge(-)$ denotes the exterior algebra, and $\tfrac 1 2$ denotes  a half-integer Tate twist. One can show that for $g>l(\lambda)$, the multiplicity of the symplectic irreducible representation $V_\lambda$ in $(\wedge V_g(\tfrac 1 2) ) ^{\otimes r}$ is given by a polynomial $p_{\lambda,r}$ in $(2g+1)$ of degree $r(r+1)/2$. The polynomial $Q_r$ is given by
	$$ Q_r = \sum_\lambda c_\lambda \cdot p_{\lambda,r},$$
	where $c_\lambda$ is the trace of Frobenius on the stable homology of the braid group with coefficients in $V_\lambda$, which was calculated in \cite{BDPW}; one can show that this agrees with the conjectured formula for $Q_r$. Thus, the left-hand side is the trace of Frobenius on the homology of $\beta_n$, and the main term on the right-hand side is the trace of Frobenius on the homology of $\beta_\infty$, suitably regularised. The stable homology is precisely what contributes equally to both left- and right-hand sides, and can be discarded when estimating the difference between the two. If the homologies of $\beta_n$ and $\beta_\infty$ agree up to degree $\approx n/a$ then one gets an error term on the order of $q^{-n/2a}$, since by the Deligne bounds \cite{Del80} the Frobenius eigenvalues on the $k$th homology have absolute value $\leq q^{-k/2}$. Importantly, we need to know this stability result for \emph{all} $\lambda$ contributing nontrivially, which is why it is crucial that the stable range can not depend on $\lambda$. It would also be interesting to improve our value $a=34$.
	
	For other examples of applications of homological stability to questions in arithmetic
	statistics, see \cite{EVW1,EL}.
	
	
	\subsection{Polynomiality}

The standard approach to proving homological stability with twisted coefficients is to prove stability for \emph{polynomial coefficient systems}, a notion going back to work of Dwyer \cite{dwyertwisted}. The coefficient systems occurring in Theorems \ref{thmA}, \ref{thmB}, \ref{thmH}, and \ref{thmC} are all polynomial (of degree $\vert\lambda\vert$ or $\vert\lambda\vert+\vert\mu\vert$, respectively), and the best previously known stable ranges quoted from the literature in connection with Theorems \ref{thmA}--\ref{thmC} are all valid more generally for any polynomial coefficient system. Now for a \emph{general} polynomial coefficient system, easy examples show that the stable range must be allowed to depend on the degree of polynomiality; in the generality of arbitrary polynomial coefficients the traditional bounds are going to be sharp. The stable ranges obtained in Theorems \ref{thmA}--\ref{thmC}, with no dependence on the degree of polynomiality, shows that these specific coefficient systems defined by systems of irreducible representations are somehow rather special in the class of all polynomial coefficient systems.

	Let us make the discussion in the previous paragraph more explicit. 
	We remind the reader that given a sequence of groups \[ \Gamma_0 \to \Gamma_1 \to \Gamma_2 \to \Gamma_3 \to \ldots \] for which we want homological stability, a \emph{coefficient system} consists of a sequence
	\[ V(0)\to V(1) \to V(2) \to V(3) \to \cdots \]
	with each $V(n)$ a $\Gamma_n$-module, with suitably equivariant stabilisation maps. The \emph{shift} of $V$ is the coefficient system defined by $(\Sigma V)(n)=V(n+1)$, and $V$ is inductively defined to be \emph{polynomial of degree} $\leq r$ if $V \to \Sigma V$ is injective with cokernel polynomial of degree $\leq (r-1)$. The inductive definition of polynomiality makes it well suited for inductive arguments, and the standard proof of homological stability with polynomial coefficients goes by induction over the degree of polynomiality. The proof is a version of the \emph{Quillen argument}, which is by far the most common technique for proving homological stability results, and in fact the definition of polynomiality almost appears tailor-made to slot neatly into the Quillen argument. 
	
	The stable ranges obtained from the standard Quillen argument with polynomial coefficients are of the form
    $$ H_d(\Gamma_n,\Gamma_{n-1};V(n),V(n-1))=0 \qquad \text{for } d<A \cdot n - B - r,$$
	where $A$ and $B$ are constants, and $r$ is the degree of polynomiality.\footnote{If the coefficient system is \emph{split} then one can improve this to $A\cdot(n-r)-B$. This is better since, without additional bells and whistles, one always has $A \leq \tfrac 1 2$.}  In particular, {the stable range depends on $r$}, and in a situation where $n$ and $r$ go to infinity at the same rate the stable range may effectively become zero. This is not a bug, or an artifact of a poorly constructed argument --- in general, it is simply not possible to do better: 
	
	\begin{example}\label{example:degree dependence 1}
		Let $\Gamma_n = \mathfrak S_n$, the symmetric group on $n$ letters, and let $V(n) = \operatorname{Ind}_{\mathfrak{S}_{n-r}}^{\mathfrak S_n} \Z$. The family $\{V(n)\}$ is a polynomial coefficient system of degree $r$, and Shapiro's lemma shows that $H_d(\Gamma_n;V(n)) = H_d(\Gamma_{n-r};\Z)$, so that the stable range for the symmetric groups with these coefficients is precisely the stable range for constant coefficients, offset by $r$. \hfill $\blacksquare$
	\end{example}

		\begin{example}\label{example:degree dependence 2}Consider $\Gamma_n = \mathrm{Sp}_{2n}(\Z)$, and let $V(n) = \bk^{2n}$ be the defining representation of $\mathrm{Sp}_{2n}$, for $\bk$ a field of characteristic zero. The family $\{V(n)^{\otimes r}\}$ is a polynomial coefficient system of degree $r$.  The stabilisation maps $$H_0(\Gamma_{n-1};V(n-1)^{\otimes 2s}) \to H_0(\Gamma_{n};V(n)^{\otimes 2s})$$ are isomorphisms for $n > s$, and this is \emph{sharp}; stabilisation is not surjective for $n=s$. This follows by explicit computation. Working instead with cohomology, the first fundamental theorem of invariant theory for the symplectic group says that $H^0(\Gamma_n;V(n)^{\otimes 2s})$ is spanned by all ways of partitioning the set $\{1,\ldots,2s\}$ into $2$-element blocks, each such partition corresponding to a way of inserting the symplectic form. The second fundamental theorem of invariant theory gives all relations between these generators: for $n > s$, there are none, but for $n\leq s$ there are always nontrivial relations. For a review of the first and second fundamental theorems of invariant theory in this context, see e.g. \cite[Section 4]{lehrerzhang}.  \hfill $\blacksquare$\end{example}
    
	Thus even in the most simple and natural examples, the stable range must be allowed to depend on the degree of polynomiality.

	\subsection{Work of Borel}
	One starting point of the present paper is that there are important examples of twisted homological stability theorems in which one has a uniform stable range for certain polynomial coefficient systems of arbitrarily high degree. More specifically, one such uniform homological stability theorem comes out of the results of Borel on the stable cohomology of arithmetic groups \cite{borelstablereal1,borelstablereal}. 
	
	Consider an arithmetic group $\Gamma$ inside a semisimple algebraic group $G$ over $\Q$. Borel  shows firstly that in a stable range, the cohomology $H^*(\Gamma;\R)$ can be computed as the $L^2$-cohomology of the associated locally symmetric space for $\Gamma$. Moreover, the $L^2$-cohomology can in a range be computed in terms of $G$-invariant differential forms on the symmetric space of $G$; in particular, the latter does not depend on the particular choice of $\Gamma$.   If $\{G_n\}$ is a family of classical groups and $\Gamma_n$ is an arbitrary family of arithmetic subgroups, then for $n \to \infty$ the consequence is that $H^*(\Gamma_n;\R)$ stabilises to an answer that can be computed purely representation-theoretically. 
	
	Moreover, Borel also proved twisted homological stability. If in the above setting $V$ is a real irreducible algebraic representation of $G$, then one can similarly calculate $H^*(\Gamma;V)$ in terms of differential forms. When $V$ is nontrivial, the outcome of Borel's argument is simply that $H^*(\Gamma;V)$ vanishes in a stable range. What is important for us is that although Borel in his original paper gave a stable range depending on $V$, it is in fact possible to give a uniform stable range depending only on the group $G$, under mild hypotheses.

	Borel's work on stable cohomology is of a fundamentally different nature than any of the standard machinery for homological stability --- at its core, it is a theorem about automorphic forms, proven by transcendental techniques. Borel's methods only prove homological stability rationally, and analogous integral assertions are completely false. 
	
		\begin{rem}
		    \cref{example:degree dependence 2} does not contradict the uniformity in Borel's results: it is only for systems of \emph{irreducible} representations that we see the uniform homological stability, and the decomposition of $V(n)^{\otimes s}$ into irreducibles depends nontrivially upon $n$, although it stabilises for $n$ sufficiently large with respect to $s$, according to Littlewood's stable branching rule.
		\end{rem}

\subsection{The main theorem}
Thus, in the very specific situation of families of arithmetic groups, we have natural examples of coefficient systems with uniform stable range. One may now instead ask about families of groups $\{\Gamma_n\}$ which come with a natural \emph{map to} a sequence of arithmetic groups $\{Q_n\}$:
\[\begin{tikzcd}\cdots\arrow[r]&\arrow[r]  \arrow[d,two heads]\Gamma_n & \arrow[d,two heads]\arrow[r]\Gamma_{n+1} &\arrow[d,two heads]\arrow[r] \Gamma_{n+2} & \cdots \\
	\cdots\arrow[r]&  Q_n\arrow[r] & \arrow[r]Q_{n+1} & \arrow[r]Q_{n+2} & \cdots\end{tikzcd}\]
The examples of \cref{thmA}, \cref{thmB}, \cref{thmH}, and \cref{thmC} all fit into this pattern:
\begin{enumerate}
	\item $\Gamma_n = \mathrm{Mod}_n^1$ and $Q_n=\mathrm{Sp}_{2n}(\Z)$, 
	\item $\Gamma_n = \mathrm{Aut}(F_n)$ and $Q_n =\mathrm{GL}_n(\Z)$,
 \item $\Gamma_n = \mathrm{HMod}_n^1$ and $Q_n = \mathrm{GL}_n(\Z)$,
	\item $\Gamma_n = \beta_n$ and $Q_n \subset \mathrm{Sp}_{n-1}(\Z)$ the image of the  integral Burau representation.
\end{enumerate}

In each of these examples, homological stability of the family $\{\Gamma_n\}$ is known via the Quillen argument. Thus, they satisfy homological stability \emph{integrally}, with \emph{arbitrary} polynomial coefficient systems, but with a stable range depending on the degree of polynomiality. But one may instead take real (or rational) coefficients, and consider the \emph{very specific} systems of polynomial coefficients obtained by pulling back a sequence of irreducible algebraic representations from the family of groups $\{Q_n\}$. Borel tells us that for this particular type of coefficients, the homology of $Q_n$ stabilises with a uniform stable range, independent of the polynomial degree. It is natural to ask whether the same is true for the homology of $\Gamma_n$ with the same coefficients. 

Now, there is in general no reason why a stable range for $H_\ast(Q_n;V(n))$ should propagate to a stable range for $H_\ast(\Gamma_n;V(n))$, but our main theorem gives conditions under which it is possible. If $V$ is a coefficient system for the family $\{Q_n\}$, then we prove a stable range for  $H_\ast(\Gamma_n;V(n))$, which is bounded in terms of three quantities:
\begin{enumerate}[(a)]
	\item \label{condA} The connectivity of the \emph{destabilisation} or \emph{splitting complexes} of the family of groups $\{\Gamma_n\}$.
	\item \label{condB} The connectivity of the \emph{destabilisation} or \emph{splitting complexes} of the family of groups $\{Q_n\}$.
	\item \label{condC} The stable range of all \emph{shifts} $H_\ast(Q_n, (\Sigma^a V)(n)) = H_\ast(Q_n,V(n+a))$, for all $a \geq 0$. 
\end{enumerate}
Estimating the connectivities in \ref{condA} and \ref{condB} is what one would usually do to prove homological stability for the respective families of groups $\{\Gamma_n\}$ and $\{Q_n\}$. Condition \ref{condC} seems awkward to verify in general. What makes it possible in our situation is that we are working with families of classical groups, and the embeddings $Q_n \to Q_{n+a}$ have well-understood \emph{branching rules}, which allow us to explicitly decompose \smash{$\operatorname{Res}^{Q_{n+a}}_{Q_n} V(n+a)$} into irreducibles; even better, the branching rule is very nearly independent of $n$, which allows us to express each shift $\Sigma^a V$ in terms of (truncated versions of) other coefficient systems of the same form as $V$.

The truncation mentioned in the preceding sentence is what gives rise to the dependence on the lengths of the partitions in \cref{thmA}, \cref{thmB}, \cref{thmH}, and \cref{thmC}. The dependence on the quantity $\max(\lambda_2',\mu_2')$ in \cref{thmB} and \cref{thmH} arises rather from the fact that the coefficient systems $W_{\lambda\mu}$ do not have a uniform stable range already for the groups $\GL_n(\Z)$ --- stability fails whenever it happens that $W_{\lambda\mu}(n)$ is a nontrivial even tensor power of the determinant. This does not contradict Borel's work, since $\GL_n$ is not semisimple, and Borel's theorem only applies to $\mathrm{SL}_n(\Z)$.

		\subsection{Sharpness}
    We do not expect any of \cref{thmA}, \cref{thmB}, \cref{thmH}, or \cref{thmC} to be optimal.   
An easy remark is that if $V$ is a coefficient system such that $V(n)=0$ for $n < N$, and $H_d(\Gamma_n,\Gamma_{n-1};V(n),V(n-1))=0$ for $d < f(n)$, then $H_d(\Gamma_\infty;V)=0$ for $d<f(N)$.
This gives an upper bound for what stable range could hold for a coefficient system $V_\lambda$ or $W_{\lambda\mu}$, given the knowledge of the stable homology with these twisted coefficients. 

In \cref{thmA}, \cref{thmB}, and \cref{thmC}, the stable homology is known, which gives explicit upper bounds. The groups $H_\ast(\Gamma_\infty^1;V_\lambda)$ are known by \cite{looijengastable,kawazumi,randalwilliamstwisted}, building on the result for constant coefficients (the Mumford conjecture) of Madsen--Weiss \cite{madsenweiss}. The groups $H_\ast(\mathrm{Aut}(F_\infty);W_{\lambda\mu})$ were determined by Lindell \cite{lindell}, building similarly on Galatius' theorem \cite{galatius}. The groups $H_\ast(\beta_\infty;V_\lambda)$ were computed in \cite{BDPW}.

These considerations show that the best possible slope of stability in \cref{thmC} that could hold uniformly for all partitions is at most $\tfrac 1 4$, by considering the homology of \smash{$V_{(k+1,2^{k},1^{k-1})}$} in degree $k$, see \cite[Remark 7.0.15]{BDPW}. Similarly in \cref{thmA} the best possible uniform slope is at most $\tfrac 1 3$, by considering the homology of \smash{$V_{(1^{3k})}$} in degree $k$, and the best possible slope in \cref{thmB} is also at most $\tfrac 1 3$, by considering the homology of \smash{$W_{(1^{2k}),(1^{k})}$} in degree $k$. Our expectation is that all these three upper bounds are sharp.

 \subsection{Relation to Torelli groups}
 Another motivation for proving uniform twisted homological stability theorems is the work of Hain \cite{haininfinitesimal} giving quadratic presentations of the Malcev completion of the Torelli subgroup of the mapping class group for $g \geq 6$ (later improved to $g \geq 4$ \cite{haing3}). A crucial input in Hain's proof is knowing (or bounding) $H^d(\mathrm{Mod}_g^1;V_\lambda(2g))$ for $d=1,2$ and \emph{all} partitions $\lambda$. When $d=1$ and $g \geq 3$, these cohomology groups were completely determined by Johnson \cite{JohnsonHomomorphism}, and when $d=2$ Kabanov \cite{kabanov} has proven purity of their mixed Hodge structure when $g \geq 6$, using a delicate algebro-geometric argument; this purity suffices to carry out the argument. Now \cref{thmA} implies that we can completely calculate $H^d(\mathrm{Mod}_g^1;V_\lambda(2g))$ in any degree $d$, for $g$ sufficiently large, which in particular would allow a simplified version of Hain's proof to carry through (but only for $g \geq 12$).
 
 It is a well-known open problem to prove an analogue of Hain's results for the Torelli subgroups of $\mathrm{Aut}(F_n)$, and one may hope that \cref{thmB} will be useful in obtaining such an analogue. \cref{thmB} by itself cannot be immediately applied to obtain a quadratic presentation of the Torelli Lie algebra of $\mathrm{Aut}(F_n)$ --- another key ingredient in Hain's work is to argue that the Lie algebra in question is isomorphic to its associated graded for the lower central series, using Hodge theory, and the analogous statement is open in the $\mathrm{Aut}(F_n)$ case.

 Hain's work may be seen as an early precursor to the general theory of \emph{representation stability} \cite{churchfarbrepresentationstability} and our \cref{thmA} may be considered as supporting evidence for uniform representation stability (as in \cite[Definition 2.6]{churchfarbrepresentationstability}) of the homology of the Torelli groups. 
 Uniform representation stability for the homology of the Torelli groups is known in degree $1$ since the work of Johnson \cite{JohnsonAb} and was recently proven in degree $2$ by Minahan and Putman \cite{minahanputman}. If uniform representation stability of the homology of the Torelli groups were known in general, then a uniform stable range for the homology \smash{$H_\ast(\mathrm{Mod}_g^1;V_\lambda(2g))$} would be an immediate consequence. Similarly, Theorems \ref{thmB}, \ref{thmH}, and \ref{thmC} may be considered as weak evidence for uniform representation stability of the homology of the respective Torelli subgroups of $\mathrm{Aut}(F_n)$, $\mathrm{HMod}_g^1$, and $\beta_n$. 

\subsection{A further example}\label{cf} In this paper, we have focused on the situation where $\{Q_n\}$ is a family of arithmetic groups, but our abstract stability theorem can be applied in other situations, too. One natural example is the braid groups mapping to the symmetric groups, i.e.\ $\Gamma_n=\beta_n$ and $Q_n=\mathfrak S_n$. Coefficient systems for the symmetric groups are $\mathit{FI}$-\emph{modules} \cite{churchellenbergfarb}. If $\lambda = (\lambda_1,\dots,\lambda_k)$ is a partition of $N$, define
 $S_\lambda(n)$ to be zero if $n < N+\lambda_1$, and $S_\lambda(n)$ to be the irreducible representation of $\mathfrak S_n$ (Specht module) associated with the partition $(n-N,\lambda_1,\lambda_2,\dots)$ if $n \geq N+\lambda_1$. Each $S_\lambda$ naturally forms an $\mathit{FI}$-module. Since the paper is already long enough, we do not give a detailed verification of the following, but let us briefly record that in this case our theorem shows that\footnote{The complex of injective words $X_n$ is a wedge of $(n-1)$-spheres \cite{farmerposets}, so Axiom \ref{it:verifyingIIIviadestab2} of \cref{prop:verifyingIIIviadestab} is satisfied with $\nu=1$, $\xi''=0$. For Axiom \ref{it:verifyingIIIviadestab1} we use \cite[Theorem 2.48]{damiolini} or \cref{degenerate remark}. \cite[Proposition 8.6]{HMW} checks that $S_\lambda$ satisfies Axiom \ref{axiomIIbis} of \cref{cor:GenCor2} with $\theta = \infty$, $\tau=0$, and $\beta=\lambda_1$. Using \cref{lem:LowDegRModHomology}\ref{it:LowDegRModHomology3} and \cref{cor:GenCor2}, the result follows.} 
 $$ H_d(\beta_n,\beta_{n-1};S_\lambda(n),S_\lambda(n-1))=0 \qquad \text{for} \qquad d < \tfrac 1 2 (n-\lambda_1).$$
 Since $S_\lambda(n)=0$ for $\lambda_1>\tfrac 1 2 n$, we in particular obtain uniform twisted homological stability with slope~$\tfrac 1 4$.
 Let $\mathrm P\beta_n$ denote the pure braid groups, which are the  ``Torelli'' groups in this setting. Since the multiplicity of $S_\lambda(n)$ in $H_d(\mathrm{P}\beta_n;\mathbb Q)$ is the dimension of 
 \[ H_0(\mathfrak S_n; H_d(\mathrm P\beta_n;\mathbb Q) \otimes S_\lambda(n)) \cong H_d(\beta_n; S_\lambda(n)), \]
 uniform twisted homological stability of the braid groups is equivalent to uniform representation stability for the homology of the pure braid groups.   
 This result was originally proven by Church--Farb \cite[Theorem 4.1]{churchfarbrepresentationstability}. 
 The optimal slope is known to be $\tfrac 1 3$ \cite{hershreiner}.

\subsection{Acknowledgements} We are grateful to Ismael Sierra and Nathalie Wahl for useful discussions around the material in \cref{connectivity section}. The braided monoidal groupoid $\mathsf T$ in \cref{connectivity section} was first written down by them (serendipitously, for unrelated reasons), and their proof of high connectivity for $\mathsf T$ in \cite{sierrawahl} predates our proof of high connectivity for $\mathsf Q$. We also thank Erik Lindell, Victor Wang, and Jennifer Wilson  for helpful discussions concerning representation theory and homological stability. We thank the anonymous referees for their careful reading and thoughtful feedback.

JM was supported by a Simons Foundation Travel Support for Mathematicians grant and
NSF grant DMS-2202943 and DMS-2504473.

PP was supported by a Simons collaboration grant and NSF grant DMS-2405310.

DP was supported by a Wallenberg Scholar Fellowship.

ORW was supported by the ERC under the European Union's Horizon 2020 research and innovation programme (grant agreement No.\ 756444) and by the Danish National Research Foundation
through the Copenhagen Centre for Geometry and Topology (DNRF151).

\section{A general theorem for propagating homological stability}
\label{section2}

Our goal in this section is to formulate and prove a general theorem for propagating homological stability properties for the family of groups $\{Q_n\}$ to the family of groups $\{\Gamma_n\}$, taking as axioms certain desirable features of both families of groups and of a coefficient system $V$ for the family $\{Q_n\}$.
%
In Section \ref{sec:Verifying} we will then verify these axioms in the cases described in Theorems \ref{thmA}, \ref{thmB}, \ref{thmH}, and \ref{thmC}. In those examples the family $\{Q_n\}$ consists of arithmetic groups, and satisfies uniform homological stability by (analogues of) the work of Borel: the general result will imply that the family $\{\Gamma_n\}$ does too. But the theorem can in principle be used in other ways as well (cf.\ Subsection \ref{cf}). We will express this result in the language developed in \cite{GKRW}, and rely on that work as well as \cite{RWclassical} for some technical support.

\subsection{Formulation}\label{subsec:formulation}

Our basic datum will be a morphism
$$p : (\mathsf{G}, \oplus, 0) \lra (\mathsf{Q}, \oplus, 0)$$
of braided strict monoidal groupoids, both of which have monoid of objects $\N$, with $p$ the identity map on objects. We abbreviate
$$\Gamma_n := \mathrm{Aut}_\mathsf{G}(n) \quad\quad\quad Q_n := \mathrm{Aut}_\mathsf{Q}(n).$$
We assume that:
\begin{enumerate}[(i)]
\item\label{it:ass:1}  The induced maps $p_n : \Gamma_n \to Q_n$ are surjective, with kernels denoted $K_n$.
\item\label{it:ass:2} $\Gamma_0$ is trivial, and hence $Q_0$ and $K_0$ are trivial too.
\item\label{it:ass:3} The maps $\Gamma_a \times \Gamma_b \to \Gamma_{a + b}$ and $Q_a \times Q_b \to Q_{a + b}$ are injective.
\end{enumerate}

\begin{rem}
Let us comment on how essential these assumptions are, making use of the notation we will introduce over the next few pages. Assumption \ref{it:ass:1} allows us to construct the canonical map $\overline{\mathbf{R}} \to \underline{\bk}_\mathsf{Q}$ by Postnikov 0-truncation in Section \ref{sec:derivedIndec}, and is also used (in the guise of Lemma \ref{lem:LowDegRModHomology} \ref{it:LowDegRModHomology2}) in a crucial way in the argument.\mnote{orw: revised to say this} Assumption \ref{it:ass:2} ensures that the $E_2$-algebras $\underline{\bk}_\mathsf{Q}$ and $\underline{\bk}_\mathsf{G}$ admit preferred augmentations, and hence that $\overline{\mathbf{R}}$ does too: this is necessary to define the derived $\overline{\mathbf{R}}$-module indecomposables in Section \ref{sec:derivedIndec} and hence to even formulate assumption \ref{axiomIII} of Theorem \ref{thm:GenThm}.

Assumption \ref{it:ass:3} is not necessary for Theorem \ref{thm:GenThm}, but is used in formulating Propositions \ref{prop:verifyingIIIviadestab} and \ref{prop:VerifyingIII}, which allow us to verify assumption \ref{axiomIII} of Theorem \ref{thm:GenThm} in terms of simplicial complexes.

Finally, following \cite{krannichtopologicalmoduli, RWclassical} and looking forward to Proposition \ref{prop:verifyingIIIviadestab}, it seems likely that a variant of our results can be obtained for $p : \mathsf{G} \to \mathsf{Q}$ a morphism of groupoids which are modules over the free braided monoidal groupoid on one generator. We leave this to the interested reader.
\end{rem}

\subsubsection{$\mathsf{G}$- and $\mathsf{Q}$-graded simplicial modules}

Fix a field $\bk$,\mnote{orw: removed ``of characteristic zero'', not used in this section} and let $\mathsf{sMod}_\bk$ denote the category of simplicial $\bk$-modules. We encourage the reader to mentally replace this by the category $\mathsf{Ch}_\bk$ of non-negatively graded chain complexes, if they are more comfortable in that context: the difference is purely technical, and allows us to directly quote results from \cite{GKRW}, but at a conceptual level there is no difference. We will work in the categories $\mathsf{sMod}_\bk^\mathsf{G}$ and $\mathsf{sMod}_\bk^\mathsf{Q}$ of functors from $\mathsf{G}$ or $\mathsf{Q}$ to $\mathsf{sMod}_\bk$. These are endowed with the projective model structure \cite[Section 7.3.1]{GKRW}, induced by the standard model structure on simplicial modules \cite[Section 7.2.3]{GKRW}. For an object $X$ of one of these categories, we define bigraded homology groups as the simplicial homotopy groups:
$$H_{n,d}(X) := \pi_d(X(n)).$$
This is the same as the homology of the associated chain complex of normalised chains. Equivalently, first let $S^d := \Delta^d_\bullet/\partial \Delta^d_\bullet \in \mathsf{sSet}_*$ denote the pointed simplicial $d$-sphere, and  $S^d_\bk := \bk[S^d]/\bk[\mathrm{pt}] \in \mathsf{sMod}_\bk$ denote its reduced $\bk$-linearisation: the homotopy groups of $S^d_\bk$ are the reduced $\bk$-homology of $S^d$. Then let $S^{n,d}_\bk := n_*S^{d}_\bk \in \mathsf{sMod}_\bk^\mathsf{G}$ denote the left Kan extension of $S^{d}_\bk$ along the inclusion $\{n\} \to \mathsf{G}$. (Explicitly, $S^{n,d}_\bk$ is supported on the object $n$, and evaluates here to $ \bk[\Gamma_n] \otimes_\bk S^d_\bk$ with its evident $\Gamma_n$-action.) This is cofibrant in $\mathsf{sMod}_\bk^\mathsf{G}$, as $S^d_\bk$ is cofibrant in $\mathsf{sMod}_\bk$ by definition of its model structure \cite[Section 7.2.3]{GKRW} and the left Kan extension $n_*(-)$ is a left Quillen functor \cite[Example 7.11]{GKRW}. All $X \in \mathsf{sMod}_\bk^\mathsf{G}$ are fibrant, by definition of the projective model structure and the fact that all objects of $\mathsf{sMod}_\bk$ are fibrant. We therefore have $H_{n,d}(X) \cong [S^{n,d}_\bk, X]_{\mathsf{sMod}_\bk^\mathsf{G}}$, the set of morphisms from $S^{n,d}_\bk$ to $X$ in the homotopy category of $\mathsf{sMod}_\bk^\mathsf{G}$, and a morphism in this category is a weak equivalence precisely when it induces an isomorphism on all bigraded homology groups. Similarly for $\mathsf{sMod}_\bk^\mathsf{Q}$.

The symmetric monoidal structure $\otimes_\bk$ on $\mathsf{sMod}_\bk$ and the braided monoidal structures on $\mathsf{G}$ and $\mathsf{Q}$ endow $\mathsf{sMod}_\bk^\mathsf{G}$ and $\mathsf{sMod}_\bk^\mathsf{Q}$ with braided monoidal structures by Day convolution, which we denote by $\otimes$. 
The unit for this monoidal structure is the functor which is $\bk$ at $0 \in \N$ and 0 otherwise: in other words it is $S^{0,0}_\bk$. Explicitly, in $\mathsf{sMod}_\bk^\mathsf{G}$ for example, it is given by the formula
$$(X \otimes Y)(n) := \colim_{\oplus / n} X(-) \otimes_\bk Y(-) \cong \bigoplus_{a+b = n} \mathrm{Ind}_{\Gamma_a \times \Gamma_b}^{\Gamma_n} X(a) \otimes_\bk Y(b),$$
where $\oplus / n$ denotes the overcategory of the functor $\oplus : \mathsf{G} \times \mathsf{G} \to \mathsf{G}$ at the object $n \in \mathsf{G}$.

A fundamental construction will be the left Kan extension $p_* : \mathsf{sMod}_\bk^\mathsf{G} \to \mathsf{sMod}_\bk^\mathsf{Q}$, given by 
$$p_*(X)(n) := \colim_{p/n} X(-) \cong  \bk[Q_n] \otimes_{\bk[\Gamma_n]} X(n) \cong  \bk \otimes_{\bk[K_n]} X(n).$$
 It is strong monoidal \cite[Lemma 2.13]{GKRW} with respect to the Day convolution monoidal structures described above. It is left adjoint to restriction $p^*$, and admits a left derived functor $\bL p_*$. In grading $n$ this is given by forming the homotopy $K_n$-orbits, giving a spectral sequence
 \begin{equation}\label{eq:HOSS}
     E^2_{n,s,t} = H_s(K_n ; H_{n,t}(X)) \Longrightarrow H_{n,s+t}(\bL p_*(X)).
 \end{equation}
 Identifying $X(n)$ with a chain complex by the Dold--Kan correspondence, this is known as the hyperhomology spectral sequence (denoted $^{I\!I}\!E^r_{*,*}$ in \cite[6.1.15]{Weibel}).

\subsubsection{Coefficient systems} \label{def:coeffient systems}In each case, the constant functor with value $\bk$ defines a commutative monoid object, denoted $\underline{\bk}_\mathsf{G}$ and  $\underline{\bk}_\mathsf{Q}$ respectively. Following \cite[Section 19.1]{GKRW}, a \emph{coefficient system} for the family of groups $\{Q_n\}$ is simply a left $\underline{\bk}_\mathsf{Q}$-module. Spelled out, this consists of a functor $V : \mathsf{Q} \to \mathsf{sMod}_\bk$, i.e.\ $\bk[Q_n]$-modules $V(n)$, equipped with $Q_a \times Q_b$-equivariant maps
$$\phi_{a,b} : \bk \otimes V(b) \lra V(a + b)$$
which are suitably associative. One does not need commutativity to discuss monoids or left modules over them, so this notion depends only on the monoidal category $(\mathsf{Q}, \oplus, 0)$ and not on its braiding.

Restricted to functors taking values in the subcategory $\mathsf{Mod}_\bk \subset \mathsf{sMod}_\bk$ of $\bk$-modules 
this recovers several well-known notions: if $\mathsf{Q}$ is the groupoid of finite sets and bijections then a coefficient system is equivalent to an $\mathit{FI}$-module of Church--Ellenberg--Farb \cite[Definition 2.1.1]{churchellenbergfarb}; if $\mathsf{Q}$ is the groupoid of finite-rank free modules over a ring and isomorphisms between them then a coefficient system is equivalent to a $\mathit{VIC}$-module of Putman--Sam \cite[Section 1.3]{PutmanSam}; more generally, a coefficient system in this sense is equivalent to one in the sense of Randal-Williams--Wahl \cite[Section 4.1]{randalwilliamswahl}, namely a functor from the enveloping homogeneous category $U\mathsf{Q}$ to $\bk$-modules.

This structure in particular provides a map $\phi_{1,n-1} : V(n-1) \to V(n)$ equivariant with respect to $1 \oplus - : Q_{n-1} \to Q_n$, using which we can form the map
$$H_*(Q_{n-1} ; V(n-1)) \lra H_*(Q_n ; V(n)).$$
In fact this map already exists at the level of chain complexes $C_*(Q_{n-1} ; V(n-1)) \to C_*(Q_n ; V(n))$, and we write $H_*(Q_n, Q_{n-1} ; V(n), V(n-1))$ for the homology of the mapping cone of this chain map. 
Similarly, a coefficient system for the groups $\{\Gamma_n\}$ is a left $\underline{\bk}_\mathsf{G}$-module. As $p^*$ is lax monoidal and $p^* \underline{\bk}_\mathsf{Q} = \underline{\bk}_\mathsf{G}$, any $\underline{\bk}_\mathsf{Q}$-module $V$ gives a $\underline{\bk}_\mathsf{G}$-module $p^*V$. We often abuse notation by writing $V$ for the latter too.

Given a coefficient system $V$, its \emph{shift} is the functor $\Sigma V$ given by $V(- \oplus 1)$. This again has the structure of a coefficient system, with the structure maps $(\Sigma\phi)_{a,b} = \phi_{a,b+1}$.

\subsubsection{Derived indecomposables}\label{sec:derivedIndec}
The final ingredient for formulating our theorem is more technical. The commutative algebra object $\underline{\bk}_\mathsf{G}$ in the braided monoidal category $\mathsf{sMod}_\bk^\mathsf{G}$ is not usually cofibrant (because the $\bk[\Gamma_n]$-module $\bk$ is not usually projective). It cannot usually be replaced by a cofibrant object still having the structure of a \emph{commutative} algebra, but it can be replaced by a cofibrant object having the weaker structure of a \emph{unital $E_2$-algebra} (cf.~\cite[Section 12]{GKRW}). We will use this concept in our proofs, but for stating the theorem we can work with a more elementary and explicit notion. (The notation is chosen to fit with \cite{GKRW}.)

Define an object $\overline{{T}} \in \mathsf{sMod}_\bk^\mathsf{G}$ by
$$\overline{{T}}(n) := \bk[E_\bullet \Gamma_n],$$
where $E_\bullet \Gamma_n$ is the simplicial set given by the two-sided bar construction $B_\bullet(\Gamma_n, \Gamma_n, *)$ of $\Gamma_n$ acting on itself and on a point: this is contractible, and has a free left $\Gamma_n$-action so $\overline{{T}}$ is cofibrant. The monoidal structure on $\mathsf{G}$ gives homomorphisms $\Gamma_n \times  \Gamma_m \to  \Gamma_{n+m}$. The functors $\mathsf{Groups} \overset{E_\bullet(-)}\to \mathsf{sSet} \overset{\bk[-]}\to \mathsf{sMod}_\bk$ are both strong symmetric monoidal (when $\mathsf{Groups}$ and $\mathsf{sSet}$ are both given the cartesian monoidal structure), giving maps
$$\overline{{T}}(n) \otimes_\bk \overline{{T}}(m) \lra \overline{{T}}(n+m)$$
which are equivariant for $\Gamma_n \times  \Gamma_m \to  \Gamma_{n+m}$: these assemble to a morphism $\mu: \overline{{T}} \otimes \overline{{T}} \to \overline{{T}}$ in $\mathsf{sMod}_\bk^\mathsf{G}$. The identity element $e \in \Gamma_0$ gives a morphism $\iota : S_\bk^{0,0} \to \overline{{T}}$, and the strictness of the monoidal structure makes $\overline{\mathbf{T}} := (\overline{{T}}, \mu, \iota)$ into a unital and associative algebra object in $\mathsf{sMod}_\bk^\mathsf{G}$. As each $E_\bullet \Gamma_n$ is contractible, there is an equivalence $\overline{\mathbf{T}} \overset{\sim}\to \underline{\bk}_\mathsf{G}$ of associative algebras, exhibiting $\overline{{T}}$ as a cofibrant replacement of $\underline{\bk}_\mathsf{G}$ in $\mathsf{sMod}_\bk^\mathsf{G}$. 

The left Kan extension $p_*$ is a left Quillen functor, so the object
$$\overline{\mathbf{R}} := p_*(\overline{\mathbf{T}})$$
is cofibrant in $\mathsf{sMod}_\bk^\mathsf{Q}$ and is a model for the derived Kan extension $\bL p_* (\underline{\bk}_\mathsf{G})$. As $p_*$ is also lax (in fact even strong) monoidal this object is an associative algebra in $\mathsf{sMod}_\bk^\mathsf{Q}$. Using assumption \ref{it:ass:1} it has
$$H_{n,d}(\overline{\mathbf{R}}) = \pi_d\left(\colim_{p/n} \bk[E_\bullet \Gamma_n] \right) \cong \pi_d\left(\hocolim_{p/n} \bk \right) \cong H_d(K_n ;\bk),$$
given by the homology of the kernels $K_n$ of the surjective homomorphisms $p_n : \Gamma_n \to Q_n$. As these homology groups in degree 0 are all $\bk$, there is a map of associative algebras $\overline{\mathbf{R}} \to \underline{\bk}_\mathsf{Q}$ induced by the lax symmetric monoidality of Postnikov truncation. This makes $\underline{\bk}_\mathsf{Q}$ into a left module over $\overline{\mathbf{R}}$. There is also an augmentation $\epsilon : \overline{\mathbf{R}} \to S^{0,0}_\bk$, so we may form the relative tensor product 
$$S^{0,0}_\bk \otimes_{\overline{\mathbf{R}}} (-) : \overline{\mathbf{R}}\text{-}\mathsf{mod} \lra \mathsf{sMod}_\bk^\mathsf{Q}.$$
In \cite[Section 9.4.2]{GKRW} this functor is denoted $Q^{\overline{\mathbf{R}}}(-)$, and it is explained there that it is a left Quillen functor and so has a left derived functor $Q_\bL^{\overline{\mathbf{R}}}(-)$. Following that reference, we will write
$$H_{n,d}^{\overline{\mathbf{R}}}(-) := H_{n,d}(Q_\bL^{\overline{\mathbf{R}}}(-)).$$
We will need to understand something about $H_{n,d}^{\overline{\mathbf{R}}}(\underline{\bk}_\mathsf{Q})$ to apply our methods.

\subsubsection{Statement of the theorem}

Our generic theorem for propagating twisted homological stability then takes the following form.

	\begin{thm}\label{thm:GenThm}
		In the setting just described, suppose that $\nu \in (0,\infty)$ is such that 
		\begin{enumerate}[(I)]
			\item \label{axiomIII} $H_{m,k}^{\overline{\mathbf{R}}}(\underline{\bk}_\mathsf{Q})=0$ 
			for $k < \nu \cdot m + 1$ and $m \geq 1$.
		\end{enumerate}
		Now fix $d,n \in \N$ and a coefficient system $V$, and suppose that
		\begin{enumerate}[(I)]
			\setcounter{enumi}{1}
			\item \label{axiomII} $H_{d-b}(Q_{n-a}, Q_{n-1-a} ; \Sigma^a V(n-a), \Sigma^a V(n-1-a))=0$ for all $a,b\in \N$ with $b \geq  \nu \cdot a$.
		\end{enumerate}
		Then  $H_d(\Gamma_n,\Gamma_{n-1};V(n),V(n-1))=0$.
	\end{thm}
	
	In our applications it will be convenient to check the following (less general) variants of Axiom \ref{axiomII}.

\begin{cor}\label{cor:GenCor}
	In the setting just described, suppose Axiom \ref{axiomIII} holds with slope $\nu$, that $V$ is a coefficient system, and that $\tau \geq 0$ is such that
	\begin{enumerate}[(I')]
		\setcounter{enumi}{1}
		\item \label{axiomIIprime} $H_{d}(Q_{n}, Q_{n-1} ; \Sigma^a V(n), \Sigma^a V(n-1))=0$ whenever $d < \nu \cdot (n-\offset)$ and $a \geq 0$.
	\end{enumerate}
	Then $H_d(\Gamma_n,\Gamma_{n-1};V(n),V(n-1))=0$ whenever $d < \nu \cdot 	(n-\tau)$.
\end{cor}
\begin{proof}
	We apply \cref{thm:GenThm}. Suppose $d$ and $n$ satisfy $d < \nu \cdot (n-\offset)$,  and $a, b \in \mathbb{N}$ satisfy $b \geq \nu \cdot a$.
	Then  $d-b < \nu\cdot(n-a-\tau)$. Hence if
	Axiom \ref{axiomIIprime} holds, then
	\begin{equation*}
		H_{d-b}(Q_{n-a}, Q_{n-1-a} ; \Sigma^a V(n-a), \Sigma^a V(n-1-a))=0,
	\end{equation*}
	which verifies Axiom \ref{axiomII}.
\end{proof}

\begin{cor}\label{cor:GenCor2}
	In the setting just described, suppose Axiom \ref{axiomIII} holds with slope $\nu$, that $V$ is a coefficient system, and that $\theta > 0$, $\beta \geq 0$ and $\tau \geq 0$ are such that
	\begin{enumerate}[(I'')]
		\setcounter{enumi}{1}
		\item \label{axiomIIbis} $H_{d}(Q_{n}, Q_{n-1} ; \Sigma^a V(n), \Sigma^a V(n-1))=0$ whenever $d < \theta \cdot (n-\offset)$, and $a \geq 0$, and $n > \beta$.
	\end{enumerate}
	Then $H_d(\Gamma_n,\Gamma_{n-1};V(n),V(n-1))=0$ for $d < \min(\nu,\theta) \cdot 	(n-\max(\tau,\beta))$.
\end{cor}
\begin{proof}
	We apply \cref{cor:GenCor}. Since Axiom \ref{axiomIII} is satisfied with slope $\nu$, it is certainly satisfied with slope $\min(\nu,\theta)$. We need to check that Axiom \ref{axiomIIbis} implies that 
	$$H_{d}(Q_{n}, Q_{n-1} ; \Sigma^a V(n), \Sigma^a V(n-1))=0 \text{ whenever } d < \min(\nu,\theta) \cdot (n-\max(\tau,\beta)).$$
	If $n-\max(\tau,\beta) \leq 0$ then the conclusion is vacuously true. If $n-\max(\tau,\beta) > 0$ then we certainly have $n > \beta$, and also $\min(\nu,\theta) \cdot (n-\max(\tau,\beta)) < \theta \cdot (n-\tau)$. 
\end{proof}

In fact, the proofs show that it suffices to verify axioms \ref{axiomIIprime} or \ref{axiomIIbis} only for $a$ satisfying $0 \leq a \leq d/\nu$.

\begin{rem}\label{rem:EmptyGroup}
In interpreting Axioms \ref{axiomII}, \ref{axiomIIprime}, \ref{axiomIIbis}, and their conclusions, we consider $Q_i$ and $\Gamma_i$ for $i < 0$ to be the ``empty group'': the classifying space of the empty group is the empty space, all of whose homology groups are zero. This explains in particular why $\tau$ must be nonnegative in Axioms \ref{axiomIIprime} and \ref{axiomIIbis}.
\end{rem}

\begin{rem}\label{rem:do not need coefficient systems}A coefficient system consists of a large amount of data: an infinite family of modules $V(m)$, and an infinite family of maps between them, satisfying an infinite list of equations. In the situation of \cref{thm:GenThm}, careful inspection of Axiom \ref{axiomII} shows that very little of this data can play any role in the proof of the theorem. Indeed, the module $V(m)$ only appears in the hypotheses of the theorem for $m \in \{n,n-1\}$, since $\Sigma^a V(n-a)=\operatorname{Res}^{Q_n}_{Q_{n-a}} V(n)$, and $\Sigma^a V(n-1-a)=\operatorname{Res}^{Q_{n-1}}_{Q_{n-1-a}} V(n-1)$. In fact,  there is no loss of generality in assuming that $V(m)=0$ for $m \notin \{n,n-1\}$, and hence that the coefficient system $V$ is of the form described in the following \cref{lem:SimpleCoeffSys}, whose proof we omit. 
\end{rem}
\begin{lem}\label{lem:SimpleCoeffSys}%
Let $n$ be fixed, $A$ be a $\bk[Q_n]$-module, $B$ be a $\bk[Q_{n-1}]$-module, and $\rho : \bk \otimes B \to \mathrm{Res}^{Q_n}_{Q_1 \times Q_{n-1}} A$ be a $Q_1 \times Q_{n-1}$-equivariant map. Then there is a unique coefficient system $V$ for the family $\{Q_m\}$ having
$$V(m) = \begin{cases}
    A & m=n\\
    B &m=n-1\\
    0 & \text{otherwise}
\end{cases}$$
and structure map $\phi_{1,n-1} : \bk \otimes V(n-1) \to V(n)$ given by $\rho$.\qed
\end{lem}

\subsection{Methods for verifying Axiom \ref{axiomIII}}
In practice, it can be difficult to directly verify Axiom \ref{axiomIII} (unless $\mathsf{G}$ happens to be the free braided monoidal groupoid on one generator --- see \cref{degenerate remark}). We therefore provide two companion results, \cref{prop:verifyingIIIviadestab} and \cref{prop:VerifyingIII}, which say that it suffices to estimate the connectivities of certain absolute invariants of the braided monoidal groupoids $\mathsf{G}$ and $\mathsf{Q}$. Before giving them, we first give a simple lemma which describes $H_{n,d}^{\overline{\mathbf{R}}}(\underline{\bk}_\mathsf{Q})$ in low degrees. This will be useful in verifying Axiom \ref{axiomIII}, and often allow for improvements to the slope $\nu$.

\begin{lem}\label{lem:LowDegRModHomology}\mbox{}
\begin{enumerate}[(i)]
    \item\label{it:LowDegRModHomology1} $H_{n,0}^{\overline{\mathbf{R}}}(\underline{\bk}_\mathsf{Q})$ is just $\bk$ supported at $n=0$. Similarly $H_{0,d}^{\overline{\mathbf{R}}}(\underline{\bk}_\mathsf{Q})$ is just $\bk$ supported at $d=0$.
    \item\label{it:LowDegRModHomology2} $H_{n,1}^{\overline{\mathbf{R}}}(\underline{\bk}_\mathsf{Q})=0$ for all $n$.
    \item\label{it:LowDegRModHomology3} If $\mu \in (0,\infty)$ and $\xi \in (-\infty,1]$ are such that $H_{n,d}^{\overline{\mathbf{R}}}(\underline{\bk}_\mathsf{Q})=0$ for $d < \mu \cdot n + \xi$ and $n \geq 1$, then Axiom \ref{axiomIII} is satisfied with $\nu = \tfrac{\mu}{2-\xi}$.
    \item\label{it:LowDegRModHomology4} Suppose that $N \geq 0$ is such that for each $n > N$ the image of the stabilisation map $H_1(K_{N};\bk) \to H_1(K_n;\bk)$ generates the target as a $\bk[Q_n]$-module. Then $H_{n,2}^{\overline{\mathbf{R}}}(\underline{\bk}_\mathsf{Q})=0$ for all $n > N$. 

    \item \label{it:LowDegRModHomology5}If both \ref{it:LowDegRModHomology3} and \ref{it:LowDegRModHomology4} are satisfied, then Axiom \ref{axiomIII} is satisfied with $\nu = \min(\tfrac{1}{N}, \tfrac{2\mu}{3-\xi})$.
\end{enumerate}    
\end{lem}
\begin{proof}
The map $H_{n,d}(\overline{\mathbf{R}}) = H_d(K_n;\bk) \to H_{n,d}(\underline{\bk}_\mathsf{Q})$ is an isomorphism when $d=0$ and an epimorphism when $d=1$ (as $H_{*,1}(\underline{\bk}_\mathsf{Q})=0$), so that the relative homology satisfies $H_{n,d}(\underline{\bk}_\mathsf{Q}, \overline{\mathbf{R}})=0$ for $d \leq 1$. 
As $H_{*,0}(\overline{\mathbf{R}}) = \underline{\bk}_\mathsf{Q}$, it follows from \cite[Corollary 11.14 (ii)]{GKRW} that the Hurewicz map $\bk \otimes_{\underline{\bk}_\mathsf{Q}} H_{*,*}(\underline{\bk}_\mathsf{Q}, \overline{\mathbf{R}}) \to H_{*,*}^{\overline{\mathbf{R}}}(\underline{\bk}_\mathsf{Q}, \overline{\mathbf{R}})$ is an isomorphism in bidegrees $(n,d)$ with $d \leq 2$, so that $H_{n,d}^{\overline{\mathbf{R}}}(\underline{\bk}_\mathsf{Q}, \overline{\mathbf{R}})=0$ for $d\leq 1$. Using $H_{0,*}(\overline{\mathbf{R}}) \cong H_*(K_0;\bk)$, and assumption \ref{it:ass:2} that $K_0$ is trivial, one sees that $H_{n,d}^{\overline{\mathbf{R}}}(\underline{\bk}_\mathsf{Q}, \overline{\mathbf{R}})$ vanishes for $n=0$ too. The long exact sequence for a pair then shows that
\begin{align*}
    H_{*,0}^{\overline{\mathbf{R}}}(\underline{\bk}_\mathsf{Q}) &=H_{*,0}^{\overline{\mathbf{R}}}(\overline{\mathbf{R}}) =\bk \text{ supported in grading 0},\\
        H_{*,1}^{\overline{\mathbf{R}}}(\underline{\bk}_\mathsf{Q}) &= H_{*,1}^{\overline{\mathbf{R}}}(\overline{\mathbf{R}})=0,\\
    H_{0,*}^{\overline{\mathbf{R}}}(\underline{\bk}_\mathsf{Q}) &=H_{0,*}^{\overline{\mathbf{R}}}(\overline{\mathbf{R}})= \bk \text{ supported in degree 0}.
\end{align*}
This proves \ref{it:LowDegRModHomology1} and \ref{it:LowDegRModHomology2}.

For \ref{it:LowDegRModHomology3}, observe first that $H_{n,d}^{\overline{\mathbf{R}}}(\underline{\bk}_\mathsf{Q})=0$ for $d < \tfrac{\mu}{2-\xi} \cdot n + 1$ and $n \geq 1$ whenever $d \leq 1$, by \ref{it:LowDegRModHomology1} and \ref{it:LowDegRModHomology2}. But $\tfrac{\mu}{2-\xi} \cdot n + 1 \leq \mu \cdot n + \xi$ whenever the left-hand side is $\geq 2$ (so that $n \geq \tfrac{2-\xi}{\mu}$).

For \ref{it:LowDegRModHomology4}, for each $n$ let us write $\tau_{\geq 1}\overline{\mathbf{R}}(n) \to \overline{\mathbf{R}}(n)$ for the 0-connected cover (i.e.\ the fibre of Postnikov 0-truncation). By adjunction we obtain a map $\bigoplus_{n \leq N} n_*(\tau_{\geq 1}\overline{\mathbf{R}}(n)) \to \overline{\mathbf{R}}$ which is an isomorphism on $H_{n,1}(-)$ for all $n \leq N$, and we extend it to a $\overline{\mathbf{R}}$-module map $\overline{\mathbf{R}}\otimes \left(\bigoplus_{n \leq N} n_*(\tau_{\geq 1}\overline{\mathbf{R}}(n))\right) \to \overline{\mathbf{R}}$
with cofibre an  $\overline{\mathbf{R}}$-module $\mathbf{C}$. 
The long exact sequence on homology groups shows that $\underline{\bk}_\mathsf{Q} = H_{*,0}(\overline{\mathbf{R}}) = H_{*,0}(\mathbf{C})$, and has a portion
$$ \bigoplus_{\substack{a+b = n \\ b \leq N}} \mathrm{Ind}_{Q_a \times Q_b}^{Q_n}(\bk \boxtimes H_1(K_b;\bk)) \lra H_1(K_n;\bk) \lra H_{n,1}(\mathbf{C}) \lra 0.$$
By the assumption the left-hand map is surjective, so $H_{*,1}(\mathbf{C})=0$. It follows that $\underline{\bk}_\mathsf{Q}$ can be constructed as an $\overline{\mathbf{R}}$-module from $\mathbf{C}$ by attaching cells of homological degree $\geq 3$, and therefore
$$H_{*,2}^{\overline{\mathbf{R}}}(\mathbf{C}) \lra H_{*,2}^{\overline{\mathbf{R}}}(\underline{\bk}_\mathsf{Q})$$
is surjective. On the other hand, the long exact sequence on $\overline{\mathbf{R}}$-module homology for the cofibre sequence defining $\mathbf{C}$ shows that $H_{n,2}^{\overline{\mathbf{R}}}(\mathbf{C})=0$ for $n>N$.

For \ref{it:LowDegRModHomology5}, observe first that $H_{n,d}^{\overline{\mathbf{R}}}(\underline{\bk}_\mathsf{Q})=0$ for $d < \min(\tfrac{1}{N}, \tfrac{2\mu}{3-\xi}) \cdot n + 1$ and $n \geq 1$ whenever $d \leq 2$, by \ref{it:LowDegRModHomology1}, \ref{it:LowDegRModHomology2} and \ref{it:LowDegRModHomology4}. But $\min(\tfrac{1}{N}, \tfrac{2\mu}{3-\xi}) \cdot n + 1 \leq \tfrac{2\mu}{3-\xi} \cdot n + 1 \leq  \mu \cdot n + \xi$ whenever the left-hand side is $\geq 3$ (so that in particular $n \geq \tfrac{3-\xi}{\mu}$.)
\end{proof}

The following \cref{prop:verifyingIIIviadestab} and \cref{prop:VerifyingIII} should be thought of as being philosophically very similar. The former takes as input an estimate of the connectivities of the \emph{destabilisation complexes} of $\mathsf{G}$ and $\mathsf{Q}$, in the sense of \cite[Definition 2.1]{randalwilliamswahl}, and the latter works instead with their \emph{$E_1$-splitting complexes} in the sense of \cite[Definition 17.9]{GKRW}. An estimate of the connectivity of either of these families of complexes will imply an estimate of the connectivity of the other, under mild hypotheses \cite[Proposition 7.1]{RWclassical}, but having both 
\cref{prop:verifyingIIIviadestab} and \cref{prop:VerifyingIII} will allow for better stable ranges  in practice. \cref{prop:verifyingIIIviadestab} allows for a connectivity estimate of any slope, and a small offset, whereas \cref{prop:VerifyingIII} only treats the case that the $E_1$-splitting complexes satisfy what \cite{GKRW} calls the \emph{standard connectivity estimate}; it seems that the splitting complexes are in any case most useful when the standard connectivity estimate holds.  

In order to define the space of destabilisations, we use assumption \ref{it:ass:3} which implies that the stabilisation maps $\Gamma_{n-p-1} \to \Gamma_n$ are injective. Then the semisimplicial set $W_\bullet(\mathsf{G}; n)$ has $p$-simplices
$$W_p(\mathsf{G}; n) := \colim_{(- \oplus (p+1))/n} \mathsf{G}(-, n) \cong \frac{\Gamma_n}{\Gamma_{n-p-1}}$$
and face maps induced by certain maps $- \oplus (p+1) \overset{\sim}\to - \oplus 1 \oplus p$ formed using the braided monoidal structure of $\mathsf{G}$, see \cite[Definition 2.1]{randalwilliamswahl} for details. 

\begin{prop}\label{prop:verifyingIIIviadestab}
Suppose that $\nu \in (0,\infty)$, $\xi' \in [-1, \infty)$, and $\xi'' \in (-\infty,\infty)$ satisfy
\begin{enumerate}[(a)]
    \item\label{it:verifyingIIIviadestab1} $\widetilde{H}_{d}(W_\bullet(\mathsf{G};n);\bk)=0$ for $d <  \nu \cdot n + \xi'-1$ and $n \geq 1$,
    \item\label{it:verifyingIIIviadestab2} $\widetilde{H}_{d}(W_\bullet(\mathsf{Q};n);\bk)=0$ for $d <  \nu \cdot n + \xi''-1$ and $n \geq 1$.
\end{enumerate}
Then $H_{n,d}^{\overline{\mathbf{R}}}(\underline{\bk}_\mathsf{Q})=0$
for $d < \nu \cdot n + \xi$ and $n \geq 1$, where $\xi := \min(\xi'', 1+\xi')$.
\end{prop}

\begin{rem}
The strategy of proof of Proposition \ref{prop:verifyingIIIviadestab} will be identical to that of \cite[Theorem 6.2 (ii)]{RWclassical}, which concerns a similar result in the setting $\mathsf{G}=\mathsf{Q}$ but with $\underline{\bk}_\mathsf{Q}$ replaced by a general $\underline{\bk}_\mathsf{Q}$-module (such a result was originally proved as \cite[Theorem 5.7]{PatztCentralStability}). Technicalities arising from working in $\mathsf{sMod}_\bk^\mathsf{Q}$ and with $E_2$-algebras obfuscate the argument a little, so let us first describe the essential idea in a simplified situation.

Let $S \to R$ be a map of differential graded $\bk$-algebras equipped with an additional $\mathbb{N}$-grading (i.e.\ of associative algebra objects in $\mathsf{Ch}(\bk)^\mathbb{N}$), $M$ be a left $R$-module, and $\epsilon : R \to \bk$ be an augmentation such that both $S \to R \to \bk$ are homology isomorphisms in gradings $\leq 0$. We assume that 
\begin{enumerate}[(A)]
\item \label{itema} $H_{n,d}(\bk \otimes_S^\bL R)=0$ for $d < \nu \cdot n + \xi'$ and $n \geq 1$,
\item \label{itemb} $H_{n,d}(\bk \otimes_S^\bL M)=0$ for $d < \nu \cdot n + \xi''$ and $n \geq 1$.
\end{enumerate} 
We wish to deduce that $H_{n,d}(\bk \otimes^\bL_R M)=0$ for $d < \nu \cdot n + \xi$ and $n \geq 1$. That is, we wish to estimate the $R$-module cells needed to construct the $R$-module $M$ by knowing the $S$-module cells needed to construct the $S$-module $M$ and the $S$-module cells needed to construct $R$.

To do so, we consider $\bk \otimes^\bL_S M \simeq (\bk \otimes^\bL_S R) \otimes^\bL_R M$ and filter the right $R$-module $\bk \otimes^\bL_S R$ by minus its grading. The associated graded of this filtration may be identified with $\bk \otimes^\bL_S R$, but the right $R$-module structure has been trivialised: it is now induced by $\epsilon : R \to \bk$. This filtered object yields a spectral sequence of signature
$$E^1_{n,p,q} = \bigoplus_{a+b=p+q} H_{-p, a}(\bk \otimes^\bL_S R) \otimes H_{n+p,b}(\bk \otimes^\bL_R M) \Longrightarrow H_{n,p+q}(\bk \otimes^\bL_S M).$$
The assumption that $S \to R \to \bk$ are both isomorphisms in grading 0 means that $H_{0,*}(\bk \otimes^\bL_S R)=\bk$ and so $H_{n,d}(\bk \otimes^\bL_R M) \cong E^1_{n,0,d}$. This assumption, together with assumption \ref{itema} and an induction, shows that the targets of the differentials $d^r : E^r_{n,0,d} \to E^r_{n,-r,d+r-1}$ are trivial for $d < \min(1+\nu \cdot n + \xi' + \xi, 1 + \nu \cdot n + \xi')$ and $n \geq 1$. Assumption \ref{itemb} shows that $E^\infty_{n,0,d}=0$ for $d < \nu \cdot n + \xi''$ and $n \geq 1$. Putting these together gives that $E^1_{n,0,d}=0$ for $d < \nu \cdot n + \xi$ and $n \geq 1$, as long as $\xi' \geq -1$ and $\xi \leq \min(\xi'', 1+\xi')$.


The argument we will give below is essentially this one, though in the slightly more complicated category $\mathsf{sMod}_\bk^\mathsf{Q}$ and with $R$ corresponding to $\overline{\mathbf{R}}$, $S$ corresponding to the free $E_2$-algebra on $S^{1,0}_\bk$, and $M$ corresponding to $\underline{\bk}_\mathsf{Q}$. Then $H_{n,d}(\bk \otimes^\bL_R M)$ corresponds to $H_{n,d}^{\overline{\mathbf{R}}}(\underline{\bk}_\mathsf{Q})$. It is preceded by some translation between the assumption \ref{it:verifyingIIIviadestab1} and \ref{it:verifyingIIIviadestab2} in the statement of Proposition \ref{prop:verifyingIIIviadestab} to the analogue of assumptions \ref{itema} and \ref{itemb} in this sketch. The results of \cite{RWclassical} allow us to relate destabilisation complexes to derived indecomposables (derived tensor products in this remark).
\end{rem}


\begin{proof}[Proof of Proposition \ref{prop:verifyingIIIviadestab}]

We change our perspective on $\overline{\mathbf{R}}$ slightly, by giving a more conceptual construction of it. The commutative algebra $\underline{\bk}_\mathsf{G}$ in the braided monoidal category $\mathsf{sMod}_\bk^\mathsf{G}$ is in particular a unital $E_2$-algebra object in this category. Using assumption \ref{it:ass:2} again it has a unique augmentation $\epsilon : \underline{\bk}_\mathsf{G} \to S_\bk^{0,0}$, and we write $(\underline{\bk}_\mathsf{G})_{>0}$ for its kernel, which is a nonunital $E_2$-algebra. The category $\mathsf{Alg}_{E_2}(\mathsf{sMod}_\bk^\mathsf{G})$ admits a projective model structure transferred from the standard model structure on $\mathsf{sMod}_\bk^\mathsf{G}$, as discussed in \cite[Section 9.2.1]{GKRW}, using that all objects of $\mathsf{sMod}_\bk$ are fibrant \cite[Section 7.2.3]{GKRW} so those of $\mathsf{sMod}_\bk^\mathsf{G}$ are too. By \cite[Lemma 9.5]{GKRW} cofibrant objects in $\mathsf{Alg}_{E_2}(\mathsf{sMod}_\bk^\mathsf{G})$ are also cofibrant in $\mathsf{sMod}_\bk^\mathsf{G}$. The non-unital $E_2$-algebra $(\underline{\bk}_\mathsf{G})_{>0}$ admits a cofibrant replacement $\mathbf{T} \overset{\sim}\to (\underline{\bk}_\mathsf{G})_{>0} \in \mathsf{Alg}_{E_2}(\mathsf{sMod}_\bk^\mathsf{G})$, which is also cofibrant in $\mathsf{sMod}_\bk^\mathsf{G}$. This in turn can be 
unitalised-and-strictified to an associative monoid $\overline{\mathbf{T}}$ (see \cite[Section 12.2.1]{GKRW}), which is cofibrant in $\mathsf{sMod}_\bk^\mathsf{G}$ by \cite[Lemma 12.7 (i)]{GKRW} and is equivalent to $\underline{\bk}_\mathsf{G}$ as an associative algebra. Using \cite[Lemma 12.8]{GKRW}, this $\overline{\mathbf{T}}$ is equivalent as an associative algebra to the explicit associative algebra of the same name we defined earlier. In particular, the associative algebra $\overline{\mathbf{R}}$ we defined earlier is equivalent as an  associative algebra (or $E_1$-algebra) to the unitalisation-and-strictification of the non-unital $E_2$-algebra
$$\mathbf{R} := p_*(\mathbf{T}) \in \mathsf{Alg}_{E_2}(\mathsf{sMod}_\bk^\mathsf{Q}).$$

Following \cite[Section 5]{RWclassical}, let us write $\widetilde{\mathbf{S}}_\mathsf{Q} := \overline{\mathbf{E}_2(S^{1,0}_\bk)}$ for the unitalisation-and-strictification of the free $E_2$-algebra on the object $S^{1,0}_\bk \in \mathsf{sMod}_\bk^{\mathsf{Q}}$. A choice of map $S^{1,0}_\bk \to \mathbf{R}$ representing the canonical generator of $H_{1,0}(\mathbf{R}) = H_0(K_1;\bk)$ freely extends to an $E_2$-map $\mathbf{E}_2(S^{1,0}_\bk) \to \mathbf{R}$, and applying $\overline{(-)}$ to this gives a map of associative algebras $\widetilde{\mathbf{S}}_\mathsf{Q} \to \overline{\mathbf{R}}$. Postnikov truncation also gives a map of associative algebras $\widetilde{\mathbf{S}}_\mathsf{Q} \to \underline{\bk}_\mathsf{Q}$.

We now translate the assumptions in the proposition into statements about derived indecomposables. Working momentarily in the category $\mathsf{sMod}_\bk^{\mathsf{G}}$, we can form the analogous associative algebra $\widetilde{\mathbf{S}}_\mathsf{G} := \overline{\mathbf{E}_2(S^{1,0}_\bk)}$ in this category, and $\underline{\bk}_\mathsf{G}$ is a left module for it. Then \cite[Theorem 5.1]{RWclassical} shows that $H_{n,d}^{\widetilde{\mathbf{S}}_\mathsf{G}}(\underline{\bk}_\mathsf{G}) \cong \widetilde{H}_{d-1}( W_\bullet(\mathsf{G};n);\bk)$, which by assumption \ref{it:verifyingIIIviadestab1} vanishes for $d-1 < \nu \cdot n+ \xi'-1$ and $n \geq 1$. As $Q^{\widetilde{\mathbf{S}}_\mathsf{Q}}_\bL(\overline{\mathbf{R}}) \simeq \bL p_*(Q^{\widetilde{\mathbf{S}}_\mathsf{G}}_\bL(\underline{\bk}_\mathsf{G}))$, the spectral sequence \eqref{eq:HOSS} then implies that $H_{n,d}^{\widetilde{\mathbf{S}}_\mathsf{Q}}(\overline{\mathbf{R}})=0$ for $d-1 < \nu \cdot n + \xi'-1$ and $n \geq 1$ as well. Similarly, $H_{n,d}^{\widetilde{\mathbf{S}}_\mathsf{Q}}(\underline{\bk}_\mathsf{Q}) \cong \widetilde{H}_{d-1}( W_\bullet(\mathsf{Q};n);\bk)$, which vanishes for $d-1 < \nu \cdot n + \xi''-1$ and $n \geq 1$ by assumption \ref{it:verifyingIIIviadestab2}.

Let $(\underline{\bk}_\mathsf{Q})^c \overset{\sim}\to \underline{\bk}_\mathsf{Q}$ be a cofibrant replacement as a left $\overline{\mathbf{R}}$-module, which is then in particular cofibrant in $\mathsf{sMod}_\bk^{\mathsf{Q}}$. Let $B(-,-,-)$ denote the two-sided bar construction.
Interchanging geometric realisations and using 
$B(\overline{\mathbf{R}}, \overline{\mathbf{R}}, (\underline{\bk}_\mathsf{Q})^c) \overset{\sim}\to (\underline{\bk}_\mathsf{Q})^c$ gives equivalences
$$B(B(S^{0,0}_\bk, \widetilde{\mathbf{S}}_\mathsf{Q}, \overline{\mathbf{R}}), \overline{\mathbf{R}}, (\underline{\bk}_\mathsf{Q})^c) \simeq B(S^{0,0}_\bk, \widetilde{\mathbf{S}}_\mathsf{Q}, B(\overline{\mathbf{R}},\overline{\mathbf{R}}, (\underline{\bk}_\mathsf{Q})^c)) \simeq B(S^{0,0}_\bk, \widetilde{\mathbf{S}}_\mathsf{Q}, (\underline{\bk}_\mathsf{Q})^c).$$ 
Furthermore, using that $(\underline{\bk}_\mathsf{Q})^c$ is cofibrant in $\mathsf{sMod}_\bk^{\mathsf{Q}}$, \cite[Corollary 9.19]{GKRW} together with the homotopy invariance of $Q_\bL^{\widetilde{\mathbf{S}}_\mathsf{Q}}(-)$ gives equivalences
\begin{equation}\label{eq:BarToIndec}
B(S^{0,0}_\bk, \widetilde{\mathbf{S}}_\mathsf{Q}, (\underline{\bk}_\mathsf{Q})^c) \simeq Q_\bL^{\widetilde{\mathbf{S}}_\mathsf{Q}}((\underline{\bk}_\mathsf{Q})^c) \simeq Q_\bL^{\widetilde{\mathbf{S}}_\mathsf{Q}}(\underline{\bk}_\mathsf{Q}).
\end{equation}
We filter the right $\overline{\mathbf{R}}$-module $B(S^{0,0}_\bk, \widetilde{\mathbf{S}}_\mathsf{Q}, \overline{\mathbf{R}})$ by its $\N$-grading, by setting
$$F_i B(S^{0,0}_\bk, \widetilde{\mathbf{S}}_\mathsf{Q}, \overline{\mathbf{R}})(n) := \begin{cases}
B(S^{0,0}_\bk, \widetilde{\mathbf{S}}_\mathsf{Q}, \overline{\mathbf{R}})(n) & n \geq -i\\
0 & n < -i,
    \end{cases}$$
which are right $\overline{\mathbf{R}}$-submodules.
The associated graded of this filtration may be canonically identified with $B(S^{0,0}_\bk, \widetilde{\mathbf{S}}_\mathsf{Q}, \overline{\mathbf{R}})$ as an object of $\mathsf{sMod}_\bk^{\mathsf{Q}}$, but its $\overline{\mathbf{R}}$-module structure is now trivial, i.e.\ factors through $\epsilon : \overline{\mathbf{R}} \to S^{0,0}_\bk$. It induces a filtration of $B(B(S^{0,0}_\bk, \widetilde{\mathbf{S}}_\mathsf{Q}, \overline{\mathbf{R}}), \overline{\mathbf{R}}, (\underline{\bk}_\mathsf{Q})^c)$, with associated graded
$$\operatorname{Gr}^F B(B(S^{0,0}_\bk, \widetilde{\mathbf{S}}_\mathsf{Q}, \overline{\mathbf{R}}), \overline{\mathbf{R}}, (\underline{\bk}_\mathsf{Q})^c) \simeq B(S^{0,0}_\bk, \widetilde{\mathbf{S}}_\mathsf{Q}, \overline{\mathbf{R}}) \otimes B(S^{0,0}_\bk,\overline{\mathbf{R}}, (\underline{\bk}_\mathsf{Q})^c).$$
The grading of the right-hand side is by \emph{minus} the $\N$-grading of the first factor. The analogue of \eqref{eq:BarToIndec} identifies the right-hand factor with $Q_\bL^{\overline{\mathbf{R}}}(\underline{\bk}_\mathsf{Q})$, and \cite[Corollary 9.19]{GKRW} identifies the left-hand factor with $Q_\bL^{\widetilde{\mathbf{S}}_\mathsf{Q}}(\overline{\mathbf{R}})$. This filtered object therefore gives a spectral sequence of signature
$$E^1_{n,p,q} = \bigoplus_{a+b=p+q} \mathrm{Ind}_{Q_{-p} \times Q_{n+p}}^{Q_n} H_{-p, a}^{\widetilde{\mathbf{S}}_\mathsf{Q}}(\overline{\mathbf{R}}) \otimes H_{n+p,b}^{\overline{\mathbf{R}}}(\underline{\bk}_\mathsf{Q}) \Longrightarrow H_{n,p+q}^{\widetilde{\mathbf{S}}_\mathsf{Q}}(\underline{\bk}_\mathsf{Q}),$$
with $p \leq 0$ and differentials $d^r : E^r_{n,p,q} \to E^r_{n,p-r, q+r-1}$. For each $n$ the filtration $F_\bullet B(S^{0,0}_\bk, \widetilde{\mathbf{S}}_\mathsf{Q}, \overline{\mathbf{R}})(n)$ is finite, so the same goes for the induced filtration of $B(B(S^{0,0}_\bk, \widetilde{\mathbf{S}}_\mathsf{Q}, \overline{\mathbf{R}}), \overline{\mathbf{R}}, (\underline{\bk}_\mathsf{Q})^c)(n)$. It follows that this spectral sequence converges strongly.

We now prove the statement in the proposition by induction on $d$. 
By \cref{lem:LowDegRModHomology} \ref{it:LowDegRModHomology1}
it holds for $d=0$. Suppose then that it holds for all $d < D$. Similar to the proof of \cref{lem:LowDegRModHomology} \ref{it:LowDegRModHomology1} we have that $H_{0,*}^{\widetilde{\mathbf{S}}_\mathsf{Q}}(\overline{\mathbf{R}})$ is $\bk$ supported in degree zero, so $E^1_{n, 0, D} \cong H_{n,D}^{\overline{\mathbf{R}}}(\underline{\bk}_\mathsf{Q})$. By our discussion above and the inductive assumption, for $1 \leq r < n$ the target of the differential
$$d^r : E^r_{n,0,D} \lra E^r_{n,-r, D+r-1} = \bigoplus_{a+b=D-1} \mathrm{Ind}_{Q_{r} \times Q_{n-r}}^{Q_n} H_{r, a}^{\widetilde{\mathbf{S}}_\mathsf{Q}}(\overline{\mathbf{R}}) \otimes H_{n-r,b}^{\overline{\mathbf{R}}}(\underline{\bk}_\mathsf{Q})$$
vanishes as long as $D-1 < (\nu \cdot r + \xi') + (\nu \cdot (n-r) +\xi)$, so as long as $D < \nu \cdot n + \xi$ using that $\xi' \geq -1$. Similarly, the target of $d^n : E^n_{n,0,D} \to E^n_{n,-n, D+n-1}$ is identified with $H_{n,D-1}^{\widetilde{\mathbf{S}}_\mathsf{Q}}(\overline{\mathbf{R}})$ by \cref{lem:LowDegRModHomology} \ref{it:LowDegRModHomology1}, so vanishes as long as $D-1 < \nu \cdot n + \xi'$ and hence in particular as long as $D < \nu \cdot n + \xi$ using that $\xi \leq 1+\xi'$. Thus in this range $E^1_{n, 0, D} = E^\infty_{n, 0, D}$. But $E^\infty_{n, 0, D}$ is a subquotient of $H_{n,D}^{\widetilde{\mathbf{S}}_\mathsf{Q}}(\underline{\bk}_\mathsf{Q})$, which by our earlier discussion vanishes as long as $D < \nu \cdot n+ \xi''$. Using that $\xi \leq \xi''$ it follows that $E^1_{n, 0, D} \cong H_{n,D}^{\overline{\mathbf{R}}}(\underline{\bk}_\mathsf{Q})$ vanishes for $D < \nu \cdot n + \xi$. 
%
\end{proof}

In order to define the $E_1$-splitting complex, we use assumption \ref{it:ass:3}. Then the semisimplicial set $S^{E_1}_\bullet(\mathsf{G};n)$ has $p$-simplices
$$S^{E_1}_p(\mathsf{G};n) := \colim_{n_0, \ldots, n_{p+1} \in \N_{>0}^{p+2}} \mathsf{G}(n_0 \oplus \cdots \oplus n_{p+1}, n) \cong \coprod_{\substack{n_0 + \cdots + n_{p+1} = n\\n_i > 0}} \frac{\Gamma_n}{\Gamma_{n_0} \times \cdots \times \Gamma_{n_{p+1}}},$$
and face maps induced by the monoidal structure. 

\begin{prop}\label{prop:VerifyingIII}
Suppose that 
\begin{enumerate}[(a)]
    \item\label{it:VerifyingIII1} $\widetilde{H}_{d}(S^{E_1}_\bullet(\mathsf{G};n);\bk)=0$ for $d <   n - 2$,
    \item\label{it:VerifyingIII2} $\widetilde{H}_{d}(S^{E_1}_\bullet(\mathsf{Q};n);\bk)=0$ for $d <   n - 2$.
\end{enumerate}
Then $H_{n,d}^{\overline{\mathbf{R}}}(\underline{\bk}_\mathsf{Q})=0$
for $d < \tfrac{2}{3}  \cdot n$.
\end{prop}
\begin{proof}
We work in the category $\mathsf{sMod}_\bk^\mathsf{Q}$, and continue to use the nonunital $E_2$-algebra $\mathbf{R}$ in this category which we described in the proof of Proposition \ref{prop:verifyingIIIviadestab}. We will apply \cite[Theorem A.1]{RWclassical} to the map of nonunital $E_2$-algebras $\mathbf{R} \to \mathbf{S} \overset{\sim}\to (\underline{\bk}_\mathsf{Q})_{>0}$ given by a relative CW-approximation of the Postnikov truncation map. 

We first establish a vanishing line for $H_{*,*}^{E_2}(\mathbf{R})$. Applied to $\bL p_* Q_\bL^{E_1}((\underline{\bk}_\mathsf{G})_{>0}) \simeq Q_\bL^{E_1}(\mathbf{R})$, the spectral sequence \eqref{eq:HOSS} takes the form
$$E^2_{n,s,t} = H_s(K_n ; H_{n,t}^{E_1}((\underline{\bk}_\mathsf{G})_{>0})) \Longrightarrow H_{n,s+t}^{E_1}(\mathbf{R}).$$
Combining Proposition 17.4 and Lemma 17.10 of \cite{GKRW} gives $H_{n,t}^{E_1}((\underline{\bk}_\mathsf{G})_{>0}) \cong \widetilde{H}_{t-1}(S^{E_1}_\bullet(\mathsf{G};n);\bk)$ for $n>0$,  so under assumption \ref{it:VerifyingIII1} we have $E^2_{s,t}=0$ for $t <   n -1$, and so $H_{n,d}^{E_1}(\mathbf{R})=0$ for $d < n - 1$. Using \cite[Theorem 14.4]{GKRW} it follows that $H_{n,d}^{E_2}(\mathbf{R})=0$ for $d < n -1$ too.

Similarly, using assumption \ref{it:VerifyingIII2} it follows that $H_{n,d}^{E_2}(\mathbf{S}) \cong H_{n,d}^{E_2}((\underline{\bk}_\mathsf{Q})_{>0})=0$ for $d < n-1$. The long exact sequence on $E_2$-homology for the map $\mathbf{R} \to \mathbf{S}$ has the form
$$H_{n,d}^{E_2}(\mathbf{R}) \lra H_{n,d}^{E_2}(\mathbf{S}) \lra H_{n,d}^{E_2}(\mathbf{S}, \mathbf{R}) \overset{\partial} \lra H_{n,d-1}^{E_2}(\mathbf{R}) \lra H_{n,d-1}^{E_2}(\mathbf{S}),$$
so $H_{n,d}^{E_2}(\mathbf{S}, \mathbf{R})=0$ for $d < n - 1$. To apply \cite[Theorem A.1]{RWclassical} we need a vanishing line of the form $d < \alpha \cdot n$ for these relative $E_2$-homology groups, which we will obtain by considering the homology of $\mathbf{R}$ and $\mathbf{S}$ in low degrees. 

The map $\mathbf{R} \to \mathbf{S}$ is an isomorphism on 0-th homology and an epimorphism on 1-st homology (as $H_{*,1}(\mathbf{S})=0$), so it follows from the Hurewicz theorem \cite[Corollary 11.12]{GKRW} that $H_{n,d}^{E_2}(\mathbf{S}, \mathbf{R})=0$ for $d \leq 1$ and all $n$. As these relative homology groups also vanish for $d < n-1$ as explained above, it follows that they vanish for $d < \tfrac{2}{3} \cdot n$.

This puts us in a situation to apply \cite[Theorem A.1]{RWclassical}, which provides a map
$$H^{\overline{\mathbf{R}}}_{n,d}(\overline{\mathbf{S}}, \overline{\mathbf{R}}) \lra H_{n,d}^{E_2}(\mathbf{S}, \mathbf{R})$$
which is an isomorphism in degrees $d < \tfrac{2}{3} \cdot n$: thus the domain also vanishes in this range. 
The same vanishing line then holds for $H^{\overline{\mathbf{R}}}_{n,d}(\overline{\mathbf{S}}) \cong H^{\overline{\mathbf{R}}}_{n,d}(\underline{\bk}_\mathsf{Q})$, as required.
\end{proof}

\subsection{Proof of Theorem \ref{thm:GenThm}}

\subsubsection{The pointwise monoidal structure}\label{pointwise}
We will make use of an additional monoidal structure on the categories $\mathsf{sMod}_\bk^\mathsf{G}$ and $\mathsf{sMod}_\bk^\mathsf{Q}$ which has not been exploited in \cite{GKRW}, namely the pointwise
tensor product
$$(A \boxtimes B)(n) := A(n) \otimes_\bk B(n).$$
The units for these monoidal structures are the constant functors $\underline{\bk}_\mathsf{G}$ and  $\underline{\bk}_\mathsf{Q}$. The functor $- \boxtimes -$ enjoys a projection formula $p_*(A) \boxtimes B \cong p_*(A \boxtimes p^*(B))$, for $A \in \mathsf{sMod}_\bk^\mathsf{G}$ and $B \in \mathsf{sMod}_\bk^\mathsf{Q}$. When evaluated at $n$, this corresponds to the identity  $(\bk[Q_n] \otimes_{\bk[\Gamma_n]} A(n)) \otimes_\bk B(n) \cong \bk[Q_n] \otimes_{\bk[\Gamma_n]}(A(n) \otimes_\bk \mathrm{Res}^{Q_n}_{\Gamma_n} B(n))$ of $\bk[Q_n]$-modules, i.e. the projection formula for induction and restriction along $p_n : \Gamma_n \to Q_n$. The interaction between the pointwise and the Day convolution monoidal structures goes under the name of a ``duoidal structure'' (see e.g. \cite[Corollary 52]{LopezFrancoVasilakopoulou}), and consists of a natural distributivity law
$$(A \boxtimes B) \otimes (C \boxtimes D) \lra (A \otimes C) \boxtimes (B \otimes D)$$
satisfying appropriate unitality and associativity properties (see e.g. \cite[Definition 6.1.1]{AguiarMahajan} for the full definition, where ``2-monoidal'' is used for ``duoidal''). This is merely a morphism, \emph{not} an isomorphism.

The homotopical behaviour of $- \boxtimes -$ is excellent, sharing the same kinds of properties as $- \otimes_\bk -$ on $\mathsf{sMod}_\bk^{G_n}$. 
Firstly, if $A$ is cofibrant, then $A \boxtimes B$ is cofibrant whatever $B$ is, because the tensor product over $\bk$ of a projective $\bk[G_n]$-module with any $\bk[G_n]$-module is a projective $\bk[G_n]$-module. Secondly, $- \boxtimes B$ preserves weak equivalences and homotopy cofibre sequences, whatever $B$ is (as $- \otimes_\bk B(n)$ does on $\mathsf{sMod}_\bk^{G_n}$), because both properties are detected after forgetting the $G_n$-action.

\subsubsection{Homology with twisted coefficients}
Considering $\N$ as a braided monoidal category with only identity morphisms, there are morphisms
$$r : \mathsf{G} \overset{p}\lra \mathsf{Q} \overset{q}\lra \N$$
of braided monoidal groupoids, where $q$ is defined to be the identity on objects. We will consider objects of $\mathsf{sMod}_\bk^\N$: these again have bigraded homology groups, which are again corepresented by the bigraded sphere objects $S^{n,d}_\bk$. There is again the left Kan extension $q_* : \mathsf{sMod}_\bk^\mathsf{Q} \to \mathsf{sMod}_\bk^\N$, which is strong monoidal with respect to Day convolution, and admits a left derived functor $\bL q_*$. If $W$ is a $\underline{\bk}_\mathsf{Q}$-module, then the derived Kan extension $\bL q_*(W)$ satisfies
$$H_{n,d}(\bL q_*(W)) \cong \pi_d\left(\hocolim_{Q_n} W(n)\right) \cong H_d(Q_n ; W(n)).$$
To see the latter, we may choose to compute the homotopy colimit using the cofibrant replacement $\bk[E_\bullet Q_n] \otimes W(n) \overset{\sim}\to W(n)$ where $E_\bullet Q_n$ is the standard contractible simplicial set with free $Q_n$-action, so that $\hocolim_{Q_n} W(n) \simeq \colim_{Q_n} \bk[E_\bullet Q_n] \otimes W(n) = \bk[E_\bullet Q_n] \otimes_{\bk[Q_n]} W(n)$; under the Dold--Kan correspondence this is the usual chain complex calculating group homology of $Q_n$ with coefficients in the $Q_n$-module $W(n)$. Similarly,
$$H_{n,d}(\bL r_*(p^* W)) \cong \pi_d\left( \hocolim_{\Gamma_n} p^*W(n) \right) \cong H_d(\Gamma_n ;  W(n)).$$
We can manipulate the latter object as
\begin{align*}
    \bL r_*(p^* W) &\simeq \bL q_* \bL p_*(\underline{\bk}_\mathsf{G} \boxtimes p^* W) && \text{as $\underline{\bk}_\mathsf{G}$ is the unit for $\boxtimes$}\\
    &\simeq \bL q_* \bL p_*(\overline{\mathbf{T}} \boxtimes p^* W) && \text{by homotopy invariance of $\boxtimes$}\\
    &\simeq \bL q_* p_*(\overline{\mathbf{T}} \boxtimes p^* W) && \text{as $\overline{\mathbf{T}}$ is cofibrant in $\mathsf{sMod}_\bk^\mathsf{G}$ and hence $\overline{\mathbf{T}} \boxtimes p^* W$ is too}\\
    &\simeq \bL q_* (\overline{\mathbf{R}} \boxtimes W) && \text{by the projection formula}.
\end{align*}

\subsubsection{The stability map}

The object $\overline{\mathbf{R}} \boxtimes W \in \mathsf{sMod}_\bk^\mathsf{Q}$ obtains the structure of a left $\overline{\mathbf{R}}$-module by writing $\overline{\mathbf{R}} = \overline{\mathbf{R}} \boxtimes \underline{\bk}_\mathsf{Q}$ and using the distributivity law
$$(\overline{\mathbf{R}} \boxtimes \underline{\bk}_\mathsf{Q}) \otimes (\overline{\mathbf{R}} \boxtimes W) \lra (\overline{\mathbf{R}} \otimes \overline{\mathbf{R}}) \boxtimes (\underline{\bk}_\mathsf{Q} \otimes W)$$
followed by the multiplication $\overline{\mathbf{R}} \otimes \overline{\mathbf{R}} \to \overline{\mathbf{R}}$ and the structure map $\underline{\bk}_\mathsf{Q} \otimes W \to W$. This may be verified to define a left $\overline{\mathbf{R}}$-module structure by developing a large commutative diagram using the associativity properties \cite[Definition 6.1.1]{AguiarMahajan} of the distributivity law. More generally, by the analogous formula $\mathbf{M} \boxtimes W$ obtains a left $\overline{\mathbf{R}}$-module structure whenever $\mathbf{M}$ is itself a left $\overline{\mathbf{R}}$-module.

Under the isomorphisms $\bk = H_0(K_1 ;\bk) \cong H_{1,0}(\overline{\mathbf{R}}) \cong [S^{1,0}_\bk, \overline{\mathbf{R}}]_{\mathsf{sMod}_\bk^\mathsf{Q}}$ the canonical generator is represented by a map $\sigma : S^{1,0}_\bk \to \overline{\mathbf{R}}$. Using the left $\overline{\mathbf{R}}$-module structure described above, we can form 
\begin{equation}\label{eq:StabMapInQ}
\sigma \cdot - : S^{1,0}_\bk \otimes (\overline{\mathbf{R}} \boxtimes W) \stackrel{\sigma \otimes id}{\lra}\overline{\mathbf{R}} \otimes (\overline{\mathbf{R}} \boxtimes W) \overset{- \cdot -}\lra (\overline{\mathbf{R}} \boxtimes W).
\end{equation}
Unravelling definitions shows that on $H_{n,d}(-)$ this has the form 
\[\sigma \cdot - : H_d(K_{n-1} ;  W(n-1)) \lra H_d(K_{n} ; W(n)),\] 
where the $K_i$ act trivially on the $W(i)$, as these are pulled back from the $Q_i$. Recall that the functor $q_*(-) : \mathsf{sMod}^\mathsf{Q}_\bk \to \mathsf{sMod}^{\N}_\bk$, and so $\bL q_*(-)$, is strong monoidal with respect to Day convolution, and observe that $\bL q_*(S^{1,0}_\bk) \simeq S^{1,0}_\bk$. (The latter may seem surprising at first, but recall that the objects $S^{1,0}_\bk$ have different meanings in the two categories: in $\mathsf{sMod}^\mathsf{Q}_\bk$ it is the object supported at $1 \in \mathsf{Q}$ with value $\bk[Q_1]$, whereas in $\mathsf{sMod}^\N_\bk$ it is the object supported at $1 \in \N$ with value $\bk$, and it is indeed true that $\hocolim_{Q_1} \bk[Q_1] \simeq \bk$.) On applying $\bL q_*(-)$ to the map \eqref{eq:StabMapInQ} we therefore obtain a map $\sigma \cdot - : S^{1,0}_\bk \otimes \bL q_*(\overline{\mathbf{R}} \boxtimes W) \to \bL q_*(\overline{\mathbf{R}} \boxtimes W)$, and again unravelling definitions shows that on $H_{n,d}(-)$ it has the form
$$\sigma \cdot - : H_d(\Gamma_{n-1} ; W(n-1)) \lra H_d(\Gamma_{n} ; W(n)).$$
This is the kind of map to which the conclusion of Theorem \ref{thm:GenThm} refers, so to prove that theorem we must establish a vanishing line for the bigraded homology groups of the homotopy cofibre $\bL q_*(\overline{\mathbf{R}} \boxtimes W)/\sigma$, for certain $W$'s. As $\bL q_*(-)$ is a derived left adjoint, writing $(\overline{\mathbf{R}} \boxtimes W)/\sigma$ for the homotopy cofibre of \eqref{eq:StabMapInQ}, this is equivalent to the object
$\bL q_*((\overline{\mathbf{R}} \boxtimes W)/\sigma) \in \mathsf{sMod}_\bk^\N$.

The construction $\bL q_*((- \boxtimes W)/\sigma)$ can more generally be applied to any  $\overline{\mathbf{R}}$-module.

\begin{lem}\label{lem:RelTwistedHomologyPresCofSeq}
The functor $\bL q_*((- \boxtimes W)/\sigma)$ sends cofibre sequences of $\overline{\mathbf{R}}$-modules to homotopy cofibre sequences in $\mathsf{sMod}_\bk^\N$.
\end{lem}
\begin{proof}
We begin by observing that the construction $(- \boxtimes W)/\sigma$ sends cofibre sequences of $\overline{\mathbf{R}}$-modules to homotopy cofibre sequences in $\mathsf{sMod}^\mathsf{Q}_\bk$. To see this, suppose that $\mathbf{A} \to \mathbf{B} \to \mathbf{C}$ is a cofibre sequence of $\overline{\mathbf{R}}$-modules and $W \in \mathsf{sMod}_\bk^\mathsf{Q}$, and consider the diagram
\begin{equation*}
\begin{tikzcd}
S^{1,0}_\bk \otimes (\mathbf{A} \boxtimes W) \rar \dar & S^{1,0}_\bk \otimes(\mathbf{B} \boxtimes W) \rar \dar & S^{1,0}_\bk \otimes(\mathbf{C} \boxtimes W) \dar\\
(\mathbf{A} \boxtimes W) \dar \rar & (\mathbf{B} \boxtimes W) \rar \dar & (\mathbf{C} \boxtimes W) \dar\\
(\mathbf{A} \boxtimes W)/\sigma \rar & (\mathbf{B} \boxtimes W)/\sigma \rar & (\mathbf{C} \boxtimes W)/\sigma.
\end{tikzcd}
\end{equation*}
Its columns are by definition homotopy cofibre sequences in $\mathsf{sMod}_\bk^\mathsf{Q}$, as are its top two rows as $- \boxtimes W$ and $S^{1,0}_\bk \otimes -$ both preserve homotopy cofibre sequences. Thus the bottom row is a homotopy cofibre sequence too. Being a left derived functor, $\bL q_*(-)$ preserves homotopy cofibre sequences.
\end{proof}

\subsubsection{The filtration}

Axiom \ref{axiomIII} says that $H_{n,d}^{\overline{\mathbf{R}}}(\underline{\bk}_\mathsf{Q})=0$ for $d < \nu \cdot n + 1$ and $n \geq 1$. Combined with \cref{lem:LowDegRModHomology} \ref{it:LowDegRModHomology2}
it follows that $H_{n,d}^{\overline{\mathbf{R}}}(\underline{\bk}_\mathsf{Q}, \overline{\mathbf{R}})$ vanishes for $d < \nu \cdot n + 1$ or for $n=0$. 
We now apply \cite[Theorem 11.21]{GKRW}, with $\mathsf{S} :=\mathsf{sMod}_\bk$, $\mathsf{G}:=\mathsf{Q}$, and $\mathcal{O} := \overline{\mathbf{R}}$ considered as an operad supported in arity 1 (so we are in case (ii), and an $\mathcal{O}$-algebra is precisely an $\overline{\mathbf{R}}$-module). By that theorem the $\overline{\mathbf{R}}$-module map $\overline{\mathbf{R}} \to \underline{\bk}_\mathsf{Q}$ admits
a relative CW $\overline{\mathbf{R}}$-module approximation $\overline{\mathbf{R}} \to \mathbf{M} \overset{\sim}\to \underline{\bk}_\mathsf{Q}$ which only has relative $(n,d)$-cells with $d \geq \nu \cdot n + 1$ and $n > 0$ (so in fact $d \geq 2$). 
The skeletal filtration of the relative CW $\overline{\mathbf{R}}$-module $\overline{\mathbf{R}} \to \mathbf{M}$ induces a skeletal filtration of the CW $\overline{\mathbf{R}}$-module $\mathbf{M}/\overline{\mathbf{R}}$, having associated graded
$$\mathrm{gr}(\mathbf{M}/\overline{\mathbf{R}}) \simeq \bigoplus_{\alpha} S^{n_\alpha, d_\alpha}_\bk \otimes \overline{\mathbf{R}}$$
with $d_\alpha \geq \nu\cdot n_\alpha + 1$ and $d_\alpha \geq 2$. 

\subsubsection{The proof}

We first establish the following, from which Theorem \ref{thm:GenThm} will follow.

\begin{prop}\label{prop:MainPreInduction}
Suppose Axiom \ref{axiomIII} holds. Fix $d,n \in \N$ and a coefficient system $V$, and suppose that
\begin{enumerate}[(a)]
    \item\label{it:MainPreInductioni} $H_d(Q_n,Q_{n-1};V(n),V(n-1))=0$, and
    \item\label{it:MainPreInductionii} $H_{d-b}(\Gamma_{n-a},\Gamma_{n-1-a}; \Sigma^{a}V(n-a),\Sigma^{a}V(n-1-a))=0$ for all $a, b \in \mathbb{N}$ such that $b \geq 1$ and $b \geq \nu \cdot a$.
\end{enumerate}
Then $H_d(\Gamma_n,\Gamma_{n-1};V(n),V(n-1))=0$. 
\end{prop}

\begin{proof}
	The skeletal filtration of the CW $\overline{\mathbf{R}}$-module $\mathbf{M}/\overline{\mathbf{R}}$ described above induces a natural filtration of  $\bL q_*(((\mathbf{M}/\overline{\mathbf{R}}) \boxtimes V)/\sigma)$ which, by Lemma \ref{lem:RelTwistedHomologyPresCofSeq}, yields a spectral sequence of signature
	$$E^1_{n,s,t} = \bigoplus_{\alpha \text{ s.t. } d_\alpha = s} H_{n, s+t}(\bL q_*(((S^{n_\alpha, d_\alpha}_\bk \otimes \overline{\mathbf{R}}) \boxtimes V)/\sigma)) \Longrightarrow H_{n, s+t}(\bL q_*(((\mathbf{M}/\overline{\mathbf{R}}) \boxtimes  V)/\sigma)).$$
	Unravelling definitions, we see in particular that 
	\begin{gather*}
		H_{n,d+1}(\bL q_*(((S^{n_\alpha, d_\alpha}_\bk \otimes \overline{\mathbf{R}})) \boxtimes V)/\sigma)) \cong H_{d-d_\alpha+1}(\Gamma_{n-n_\alpha}, \Gamma_{n-1-n_\alpha} ;  V(n), V(n-1)) \hspace{2cm}\\
		\hspace{4cm} = H_{d-d_\alpha+1}(\Gamma_{n-n_\alpha}, \Gamma_{n-1-n_\alpha} ; \Sigma^{n_\alpha}V(n-n_\alpha), \Sigma^{n_\alpha}V(n-1-n_\alpha)).
	\end{gather*}
    As $d_\alpha \geq 2$ and $d_\alpha \geq \nu \cdot n_\alpha+1$, it follows from assumption \ref{it:MainPreInductionii} with $b:=d_\alpha-1$ and $a:=n_\alpha$ that these groups all vanish, and hence that $H_{n, d+1}(\bL q_*(((\mathbf{M}/\overline{\mathbf{R}}) \boxtimes V)/\sigma))=0$ by running the spectral sequence.
	
	We now consider the portion of the long exact sequence
	$$H_{n,d+1}(\bL q_*(((\mathbf{M}/\overline{\mathbf{R}}) \boxtimes V)/\sigma)) \overset{\partial}\lra H_{n,d}(\bL q_*((\overline{\mathbf{R}} \boxtimes  V)/\sigma)) \lra H_{n,d}(\bL q_*((\mathbf{M} \boxtimes V)/\sigma))$$
	obtained by applying $\bL q_*((- \boxtimes V)/\sigma)$ to the cofibration sequence $\overline{\mathbf{R}} \to \mathbf{M} \to \mathbf{M}/\overline{\mathbf{R}}$ and using Lemma \ref{lem:RelTwistedHomologyPresCofSeq}. The equivalence $\mathbf{M} \overset{\sim}\to \underline{\bk}_\mathsf{Q}$, and the fact that $-\boxtimes-$ preserves weak equivalences in both variables and $\underline{\bk}_\mathsf{Q}$ is the unit for it, identifies the right-hand term with
	$$H_d(Q_n, Q_{n-1} ; V(n), V(n-1)),$$
	which vanishes by assumption \ref{it:MainPreInductioni}. Thus the middle term vanishes, too.
\end{proof}

With these preparations completed, we now embark on the proof of Theorem \ref{thm:GenThm} proper. We proceed by induction on $d$. For the base case we may take $d=-1$, where the conclusion is vacuously true. 
	
For the inductive step, we assume Axiom \ref{axiomIII} and Axiom \ref{axiomII} for the data $(d,n,V)$, and we suppose that Theorem \ref{thm:GenThm} holds for all data $(d', n', V')$ with $d' < d$. We wish to apply \cref{prop:MainPreInduction}. Note that $H_d(Q_n,Q_{n-1};V(n),V(n-1))=0$, by taking $a=b=0$ in Axiom \ref{axiomII}, which verifies assumption \ref{it:MainPreInductioni}. Assumption \ref{it:MainPreInductionii} is that
\begin{equation}\label{eq:Assii}
H_{d-b}(\Gamma_{n-a},\Gamma_{n-1-a}; \Sigma^{a}V(n-a),\Sigma^{a}V(n-1-a))=0  
\end{equation}
for all $a, b \in \mathbb{N}$ such that $b \geq 1$ and $b \geq \nu \cdot a$. As $b \geq 1$ we have $d-b < d$ so by inductive assumption Theorem \ref{thm:GenThm} may be applied with the data $(d-b, n-a, \Sigma^a V)$. That theorem requires Axiom \ref{axiomIII}, which we are assuming, and Axiom \ref{axiomII} for the data $(d-b, n-a, \Sigma^a V)$ says that
$$H_{d-b-s}(Q_{n-a-t},Q_{n-1-a-t}; \Sigma^{a+t}V(n-a-t),\Sigma^{a+t}V(n-1-a-t))=0$$
for all $t, s \in \mathbb{N}$ such that $s \geq \nu \cdot t$. But as $b+s \geq \nu \cdot (a+t)$ this follows from Axiom \ref{axiomII} for the data $(d, n, V)$, which we have assumed. So Theorem \ref{thm:GenThm} applies to prove the vanishing in \eqref{eq:Assii}, providing assumption \ref{it:MainPreInductionii} of \cref{prop:MainPreInduction}, so the latter applies to show that $H_d(\Gamma_n,\Gamma_{n-1};V(n),V(n-1))=0$ as required.

\section{The destabilisation complexes for the image of the Burau representation}\label{connectivity section}

The goal in this section is to prove high connectivity of  the complexes $W_\bullet(\mathsf Q;n)$ used in the proof of \cref{thmC}, i.e.~the destabilisation complexes associated to the family of groups given by the image of the integral Burau representation. The reader who is more interested in Theorems \ref{thmA}, \ref{thmB}, and \ref{thmH} may skip ahead to \cref{sec:Verifying}. The proof is an adaptation of an argument of Mirzaii--van der Kallen \cite{MvdK} in the case of the usual symplectic groups, taking as input also a result of Charney \cite{charneycongruence} showing high connectivity of partial basis complexes with a congruence condition.

	\subsection{Even and odd symplectic groups}
The images of the integral Burau representation form a family of congruence subgroups of the \emph{even and odd} symplectic groups, as introduced by Gelfand--Zelevinsky \cite{gelfandzelevinskyodd}. They were introduced to ``interpolate'' between the usual symplectic groups --- in many situations, it may seem that there ought to be some such family of ``missing'' groups.  

\begin{defn}
    The \emph{odd symplectic group} $\mathrm{Sp}_{2n-1}$ is the  subgroup of $\mathrm{Sp}_{2n}$ stabilising a unimodular vector.
\end{defn}

 \begin{rem}The even and odd symplectic groups may be considered as algebraic groups over $\Z$, but we will really only be interested in their integer points. \end{rem}

The first goal of this section is to make the system of all even and odd symplectic groups into a braided monoidal groupoid. The monoidal structure is somewhat subtle. In particular, it does \emph{not} restrict to the usual block-sum of matrices when restricted to the classical even symplectic groups. In fact, the even symplectic groups do not even form a subgroupoid, as the indexing in the groupoid is ``off by one'' compared to the classical one --- the groupoid is of the form $ \coprod_{n \geq 0} \mathrm{Sp}_{n-1}(\Z) $, so to speak. 
Moreover, unlike the block-sum, the braiding is \emph{not} a symmetry; it is genuinely 
braided monoidal.

	\subsubsection{A monoidal groupoid of even and odd symplectic groups}\label{sec:evenandoddsp}
	Let $\mathsf M$ be the following groupoid. Objects are triples $(M,\langle-,-\rangle,\phi)$ where $M$ is an abelian group, $\langle-,-\rangle$ is a skew-symmetric bilinear form on $M$, and $\phi : M \to \Z$ is linear. Morphisms are isomorphisms of this data. We generally denote an object of $\mathsf M$ by $M$, leaving the remaining data implicit in the notation.
	
	The category $\mathsf M$ admits a monoidal structure given by 
	$$ (M,\langle-,-\rangle_M,\phi_M) \oplus (M',\langle-,-\rangle_{M'},\phi_{M'}) = (M \oplus {M'}, \langle-,-\rangle_{M\oplus {M'}},\phi_{M \oplus {M'}})$$
	where for $x,y \in M$ and $x',y' \in M'$ we have
	$$\langle x+x',y+y'\rangle_{M \oplus M'} = \langle x,y\rangle_M + \langle x',y'\rangle_{M'} + \phi_M(x)\phi_{M'}(y') - \phi_M(y)\phi_{M'}(x'),$$
	 and $\phi_{M \oplus {M'}}(x+x')=\phi_M(x)+\phi_{M'}(x')$. The sum of two morphisms is the block-sum.

	 Let $\Z$ denote the object of $\mathsf M$ given by the triple $(\Z,0,\mathrm{id})$. 
	 Let $\mathsf T$ denote the full subgroupoid spanned by the objects $\Z^{\oplus n}$ for $n \geq 0$. Write $T_n = \mathrm{Aut}_{\mathsf T}(\Z^{\oplus n})$.  We denote the standard basis for $\Z^{\oplus n}$ by $e_1,\ldots,e_{n}$. Explicitly, $\Z^{\oplus n}$ has the bilinear form
	 $$\langle e_i,e_j\rangle = \begin{cases}
	 	1 & i<j \\ 0 & i=j \\ -1 & i>j
	 \end{cases},$$
 and $\phi(e_i) = 1$ for all $i$. The monoidal structure furnishes maps $T_n \times T_m \to T_{n+m}$, given by block-sum of matrices. The following is also proven in \cite[Lemma 2.1]{bloomquist} and \cite[Proposition 2.17]{sierrawahl}.

	 \begin{thm}\label{Qn is symplectic group}
	 	There are isomorphisms $T_n \cong \Sp_{n-1}(\Z)$. 
	 \end{thm}
	
	\begin{proof}Let $M=\Z^{\oplus n}$, and define a new basis for $M$ by
 \begin{align*}
     e_1' &= e_1 &
     e_2' &= e_2 &\\
     e_3' &= e_1-e_2+e_3 &
     e_4' &= e_1-e_2+e_4 &\\
     e_5' &= e_1-e_2+e_3-e_4+e_5 &
     e_6' &= e_1-e_2+e_3-e_4+e_6, 
 \end{align*}
and so on.

If $n$ is even, then one checks that $e_1',\dots,e_n'$ is a standard symplectic basis for $M$ with respect to $\langle-,-\rangle$. Via the symplectic form, we can identify $\phi \in M^\vee$ with a unimodular vector in $M$, and so $T_n$ is identified with the stabiliser  of this unimodular vector inside $\mathrm{Sp}(M)$.

		If $n$ is odd, this basis shows instead that $\langle-,-\rangle$ has a rank one kernel $K$, spanned by the unimodular vector $e_n'$, and that the form descends to a symplectic form on $M/K$. But the map $\phi$ splits the short exact sequence $0\to K\to M \to M/K\to 0$, since $\phi(e_n')=1$. This gives a canonical identification $M/K \cong \ker(\phi)$ and picks out a distinguished generator of $K$.  Thus $T_n \cong \mathrm{Sp}(M/K)$.
  %
	%
	\end{proof}

\begin{cor}\label{distinguished vector}
	The element $v_n := e_1 - e_2 + \ldots + (-1)^{n-1}e_{n} \in \Z^{\oplus n}$  is fixed by every element of $T_n$. 
\end{cor}
\begin{proof}
If $n$ is even, then in the preceding proof we identified $\phi$ with a unimodular vector in $\Z^{\oplus n}$ using the symplectic form, and one checks that this vector is precisely $v_n$. Explicitly, this means that $v_n$ is the unique vector such that $\langle v_n,x\rangle = \phi(x)$ for all $x \in \Z^{\oplus n}$.

If $n$ is odd, then in the preceding proof we argued that $K$ has a distinguished generator $e_n'=v_n$. Explicitly, $v_n$ is the unique vector in $\Z^{\oplus n}$ satisfying $\phi(v_n)=1$ and $\langle v_n,-\rangle$ is identically zero.
\end{proof}

\begin{defn}\label{def:defining}
    By the \emph{defining representation} of $T_n$ or $\Sp_{n-1}(\Z)$ 
    we mean the representation $\ker(\phi)$.
\end{defn}

\begin{rem}The defining representation may be considered as a homomorphism $\mathrm{Sp}_n \to \mathrm{GL}_n$. This homomorphism is injective for even $n$, but it is \emph{not} injective when $n$ is odd. Therefore a reader may object to \cref{def:defining}, saying that the action of $\mathrm{Sp}_n$ on $\Z^{\oplus (n+1)}$ ought to be considered its defining representation. But a redeeming quality of \cref{def:defining} is that the defining representation of $\Sp_{2g}(\Z)$ is the usual representation of rank $2g$. 
\end{rem}

\begin{rem}\label{rem:proctor}Another definition of odd symplectic groups, different from Gelfand--Zelevinsky's, was independently proposed by Proctor \cite{proctor}. We denote his groups by $\smash{\mathrm{Sp}^P_{2n-1}}$, which he defines as the subgroup of $\mathrm{GL}_{2n-1}$ stabilising a skew-symmetric form of maximal rank. Equivalently, $\smash{\mathrm{Sp}^P_{n}}$ is the image of the defining representation $\mathrm{Sp}_n \to \mathrm{GL}_n$. 
In \S\ref{subsec:symplectic} we will describe the representations of odd symplectic groups which will be relevant for us. All these representations factor through the natural surjections $\mathrm{Sp}_{n} \to \smash{\mathrm{Sp}_{n}^P}$, and we may if we prefer think about them in terms of Proctor's groups instead.  Nevertheless, it will be important that we work with the Gelfand--Zelevinsky groups throughout. One reason is that Proctor's groups do not even admit suitable stabilisations: the embedding of $\mathrm{Sp}_{2n-2}$ into $\mathrm{Sp}_{2n}$ does not factor through the evident embedding $\mathrm{Sp}_{2n-2} \hookrightarrow \smash{\mathrm{Sp}_{2n-1}^P}$, but it obviously factors through $\mathrm{Sp}_{2n-1}$.\end{rem}

\subsubsection{The Burau representation}\label{sec:BurauRep} The Burau representations are a well-studied sequence of representations of the braid groups, whose definition we now recall. In general the representation has coefficients in the Laurent polynomial ring $\Z[t,t^{-1}]$. We will only be interested in its specialization to $t=-1$, which is called the \emph{integral} Burau representation.

\begin{defn}
    The (unreduced) \emph{integral Burau representation} of the braid group $\beta_n$ is the representation which maps the standard generator $\sigma_i$ of $\beta_n$ (where $i=1,\dots,n-1$) to the block-matrix
    \[ \begin{pmatrix}\mathbf{1}_{i-1} & 0 & 0 \\ 0 & B & 0 \\ 0 & 0 & \mathbf 1_{n-i-1}\end{pmatrix}\]
    where $B=\left(\begin{smallmatrix} 2 & 1 \\-1 & 0 \end{smallmatrix}\right)$. 
\end{defn}
We could replace $B$ by its transpose in this definition, and we would get an isomorphic representation of the braid group. Indeed, transposing $B$ amounts to conjugating by the matrix $\mathrm{diag}(1,-1,1,-1,\dots)$. For this reason there are differing conventions in the literature for exactly how the Burau matrices are defined.  

\begin{defn}
    The \emph{reduced integral Burau representation} is the rank $n-1$ subrepresentation of the integral Burau representation, given by $\beta_n$ acting on $\{(x_1,\dots,x_n) \in \Z^n: x_1+\dots+x_n=0\}$. 
\end{defn}

There is a ``reflection'' automorphism of the braid groups $\rho : \beta_n \to \beta_n$ given geometrically by interchanging under- and over-crossings, or algebraically by sending the standard generators $\sigma_i$ to their inverses. These automorphisms are natural with respect to stabilisation of braids. Pulling back a representation $W$ of $\beta_n$ along $\rho$ gives a new representation $\rho^*W$, the \emph{reflection} of $W$. This will not usually be isomorphic to $W$. However, there is an isomorphism
\begin{equation}\label{eq:ReflIso}
    (\rho, \mathrm{Id})_* : H_*(\beta_n ; \rho^*W) \overset{\sim}\lra H_*(\beta_n ; W),
\end{equation}
so the homologies of $\beta_n$ with coefficients in these two representations are isomorphic. It will be convenient for us to work with the reflection of the Burau representation rather than the Burau representation itself, as this will give more conveniently described destabilisation complexes (\cref{description of destab}), but by \eqref{eq:ReflIso} it makes no difference at the level of homology or of homological stability.



\subsubsection{A braiding on the monoidal groupoid of symplectic groups}\label{subsubsec:braid}
Using the distinguished vector $v_n$ from \cref{distinguished vector}, we can write down a braiding on the monoidal category $(\mathsf T, \oplus, 0)$. It is defined by the isomorphisms $c_{n,m} : \Z^{\oplus n} \oplus \Z^{\oplus m} \to \Z^{\oplus m} \oplus \Z^{\oplus n}$ given by
 \begin{equation}\label{eq:braiding2}
  c_{n,m}(x,y) = ((-1)^{n} y,x+(-1)^{n+1} 2\phi(y)v_n).
  \end{equation}
 for $x \in \Z^{\oplus n}$ and $y \in \Z^{\oplus m}$. It is routine to verify that these isomorphisms preserve $\phi$ and $\langle-,-\rangle$, and satisfy the hexagon identities; we will not write the verification out. The inverse of this braiding appears in \cite[Proposition 2.9]{sierrawahl}.
 
 It is natural to ask how one would come up with this formula. It can be motivated in at least two ways:
 
 \emph{First motivation.} The homomorphism $\beta_n \to \mathrm{Aut}(\Z^{\oplus n})$ induced by \eqref{eq:braiding2} agrees with the reflection of the integral Burau representation, as one sees by computing the images of the standard generators of the braid group. (In particular, $c_{n,m}$ is just the image of the corresponding braid group element under the integral Burau representation.) Thus, the braiding \eqref{eq:braiding2} makes the union of all reflected integral Burau representations into a braided monoidal functor, informally speaking. 
     
      \emph{Second motivation.} Harr--Vistrup--Wahl \cite{harrvistrupwahl} have considered a monoidal groupoid $\mathbf M_2$ of bidecorated surfaces, meaning triples $(S,I,J)$ of a surface with two marked intervals on its boundary. The monoidal structure $\#$ is given by gluing surfaces along half-intervals (see \cite{harrvistrupwahl} for details). There is a monoidal functor $F:\mathbf M_2\to \mathsf M$, defined by $F(S,I,J)=H_1(S,I \sqcup J;\Z)$ \cite[Proposition 2.6]{sierrawahl}. The pairing on $F(S,I,J)$ is given by intersection product, under the identification  $$H_1(S,I \sqcup J;\Z) \cong H_1(S';\Z),$$ where  $S'$ is the surface obtained by gluing $I$ to $J$. The map $\phi$ is given by the connecting homomorphism $H_1(S,I \sqcup J;\Z) \to \smash{\widetilde{H}_0}(I \sqcup J;\Z) \cong \Z$. The functor $F$ provides some motivation for studying the monoidal category $\mathsf M$. 
     
     Let $D$ be the object of $\mathbf M_2$ given by a disk with two marked intervals. Then $D^{\# (2g+1)}$ is a surface of genus $g$ with a boundary component, and $D^{\# (2g+2)}$ is a surface of genus $g$ with two boundary components (cf.\ Fig.\ \ref{figA}, \ref{figB}, \ref{figC}). In particular, $D \# D$ is a cylinder, whose mapping class group is generated by a single Dehn twist $\tau$ (Fig. \ref{figD}). As explained in \cite{harrvistrupwahl}, the isomorphism $\tau:D \# D\to D \# D$ does not induce a braiding on the full subcategory of $\mathbf M_2$ spanned by the powers $D^{\#n}$, but it satisfies the weaker condition of being a Yang--Baxter operator. Now $F(D)$ is the object $\Z$ of $\mathsf M$, and perhaps surprisingly it turns out $F(\tau)$ \emph{does}   induce a braiding on $\mathsf T$, i.e.\ the full subcategory of $\mathsf M$ spanned by the powers $\Z^{\oplus n}$; namely, Fig.\ \ref{figE} shows that $F(\tau)$ is the inverse Burau matrix $\left(\begin{smallmatrix} 0 & -1 \\1 & 2 \end{smallmatrix}\right)$, and hence the braiding is the one defined here.

\begin{figure}
    \begin{subfigure}{.13\textwidth}
    \centering    
    \resizebox{!}{3cm}{
\begin{tikzpicture}
    \draw[double distance=.8cm,line width=.1cm] (0.5,0) to[out=90,in=-90] (0.5,3);
    \draw[orange, line width=0.05cm] (0.5,-0.7) to (0.5,0) to[out=90,in=-90] (0.5,3) to (0.5,3.7);
    
   \draw[line width=.1cm] (0.05,0) to (0.05,-0.7) to (.95,-0.7) to (.95,0);
    \draw[line width=.1cm] (0.05,3) to (0.05,3.7) to (.95,3.7) to (.95,3);
\end{tikzpicture}}\caption{\hspace{0pt}}\label{figA}\end{subfigure}\begin{subfigure}{.18\textwidth}
    \centering
\resizebox{!}{3cm}{\begin{tikzpicture}[xscale=-1]
    \draw[double distance=.8cm,line width=.1cm] (0.5,0) to[out=90,in=-90] (1.6,3);
    \draw[orange, line width=0.05cm] (0.5,-0.7) to (0.5,0) to[out=90,in=-90] (1.6,3) to (1.6,3.7);
    
    \draw[double distance=.8cm,line width=.1cm] (1.6,0) to[out=90,in=-90] (0.5,3);
    \draw[cyan, line width=0.05cm] (1.6,-0.7) to (1.6,0) to[out=90,in=-90] (0.5,3) to (0.5,3.7);
    
   \draw[line width=.1cm] (0.05,0) to (0.05,-0.7) to (2.05,-0.7) to (2.05,0);
    \draw[line width=.1cm] (.9,0) to (1.2,0);
    \draw[line width=.1cm] (.9,3) to (1.2,3);
    \draw[line width=.1cm] (0.05,3) to (0.05,3.7) to (2.05,3.7) to (2.05,3);
\end{tikzpicture}}\caption{\hspace{0pt}}\label{figB}\end{subfigure}\begin{subfigure}{.25\textwidth}
    \centering
\resizebox{!}{3cm}{\begin{tikzpicture}[xscale=-1]
    \draw[double distance=.8cm,line width=.1cm] (0.5,0) to[out=90,in=-90] (2.7,3);
    \draw[orange, line width=0.05cm] (0.5,-0.7) to (0.5,0) to[out=90,in=-90] (2.7,3) to (2.7,3.7);
    \draw[double distance=.8cm,line width=.1cm] (1.6,0) to[out=90,in=-90] (1.6,3);
    \draw[cyan, line width=0.05cm] (1.6,-0.7) to (1.6,0) to[out=90,in=-90] (1.6,3) to (1.6,3.7);
    \draw[double distance=.8cm,line width=.1cm] (2.7,0) to[out=90,in=-90] (0.5,3);
    \draw[magenta, line width=0.05cm] (2.7,-0.7) to (2.7,0) to[out=90,in=-90] (0.5,3) to (0.5,3.7);
    \draw[line width=.1cm] (0.05,0) to (0.05,-0.7) to (3.15,-0.7) to (3.15,0);
    \draw[line width=.1cm] (.9,0) to (1.2,0);
    \draw[line width=.1cm] (2,0) to (2.3,0); 
    \draw[line width=.1cm] (.9,3) to (1.2,3);
    \draw[line width=.1cm] (2,3) to (2.3,3);
    \draw[line width=.1cm] (0.05,3) to (0.05,3.7) to (3.15,3.7) to (3.15,3);
\end{tikzpicture}}\caption{\hspace{0pt}}\label{figC}\end{subfigure}\begin{subfigure}{.18\textwidth}
    \centering
\resizebox{!}{3cm}{\begin{tikzpicture}[xscale=-1]
    \draw[double distance=.8cm,line width=.1cm] (0.5,0) to[out=90,in=-90] (1.6,3);
    \draw[violet, line width=0.05cm,->] (1.05,-0.4) to[out=180,in=-90] (0.5,0) to[out=90,in=-90] (1.6,3) to[out=90,in=0] (1.05,3.4);
    \draw[double distance=.8cm,line width=.1cm] (1.6,0) to[out=90,in=-90] (0.5,3);
    \draw[violet, line width=0.05cm,<-] (1.05,-0.4) to[out=0,in=-90] (1.6,0) to[out=90,in=-90] (0.5,3) to[out=90,in=180] (1.05,3.4);
    \draw[line width=.1cm] (0.05,0) to (0.05,-0.7) to (2.05,-0.7) to (2.05,0);
    \draw[line width=.1cm] (.9,0) to (1.2,0);
    \draw[line width=.1cm] (.9,3) to (1.2,3);
    \draw[line width=.1cm] (0.05,3) to (0.05,3.7) to (2.05,3.7) to (2.05,3);
\end{tikzpicture}}\caption{\hspace{0pt}}\label{figD}\end{subfigure}\begin{subfigure}{.18\textwidth}
    \centering
\resizebox{!}{3cm}{\begin{tikzpicture}[xscale=-1]
    \draw[double distance=.8cm,line width=.1cm] (0.5,0) to[out=90,in=-90] (1.6,3);
    \draw[cyan, line width=0.05cm] (1.05,-0.4) to[out=180,in=-90] (0.7,0) to[out=90,in=-90] (1.4,3) to[out=90,in=0] (1.05,3.4);

    \draw[orange, line width=0.05cm] (1.8,3.7) to[out=-90,in=90] (1.8,3) to[out=-90,in=-70] (1.5,2) to[out=110,in=-90] (1.6,3) to[out=90,in=0] (1.05,3.5);
    
    \draw[double distance=.8cm,line width=.1cm] (1.6,0) to[out=90,in=-90] (0.5,3);
    
    \draw[cyan, line width=0.05cm] (1.05,-0.4) to[out=0,in=-90] (1.5,0) to[out=90,in=-90] (0.4,3) to[out=90,in=-90] (0.4,3.7);
    \draw[cyan, line width=0.05cm] (1.8,-0.7) to[out=90,in=-90] (1.8,0) to[out=90,in=-90] (0.7,3) to[out=90,in=180] (1.05,3.4);

\draw[orange, line width=0.05cm] (1.05,-0.5) to[out=0,in=-90] (1.65,0) to[out=90,in=-90] (0.55,3) to[out=90,in=180] (1.05,3.5);

    \draw[orange, line width=0.05cm] (0.3,-0.7) to (0.3,0) to[out=90,in=110] (0.6,1) to[out=-70,in=90] (0.5,0) to[out=-90,in=180] (1.05,-0.5);
    \draw[line width=.1cm] (0.05,0) to (0.05,-0.7) to (2.05,-0.7) to (2.05,0);
    \draw[line width=.1cm] (.9,0) to (1.2,0);
    \draw[line width=.1cm] (.9,3) to (1.2,3);
    \draw[line width=.1cm] (0.05,3) to (0.05,3.7) to (2.05,3.7) to (2.05,3);
\end{tikzpicture}}\caption{\hspace{0pt}}\label{figE}\end{subfigure}    \caption{\small (a), (b) and (c) : $D^{\#n}$ for $n=1,2,3$. In each case, the coloured curves, oriented downwards, make up a basis for $F(D^{\#n}) \cong \mathbb Z^{\oplus n}$. (d) : The Dehn twist $\tau$ along the violet curve is our preferred generator of the mapping class group of $D \# D$. (e) : The effect of $\tau$ on the orange and blue curves from (b).}
\end{figure}
 
 The second motivation given above is the perspective taken in \cite{sierrawahl}, but also relates most directly to the setting of \cite{BDPW}. To a configuration of points, say in a square, one can associate the branched double cover of the square with branching at the marked points. After choosing an interval on the boundary of the square, the double cover naturally gets the structure of a bidecorated surface. This construction produces a monoidal functor from the groupoid of braid groups to the category $\mathbf M_2$. Indeed, the monoidal structure on configurations of points is given by gluing squares together, and if two such squares are equipped with branched double covers, then gluing together these covers amounts to performing a boundary connected sum along two marked intervals. See \cite[Sections 1.3, 4.1, 4.2]{BDPW} and in particular \cite[\S 4.2.4]{BDPW}. 

\begin{rem}
The distinction between using the braiding \eqref{eq:braiding2} or its inverse corresponds to the distinction between the reflection of the Burau representation or the Burau representation itself. As explained at the end of \cref{sec:BurauRep}, this is inconsequential from the point of view of taking homology. Another reason it is inconsequential is that for a braided monoidal groupoid $\mathsf{C}$ for which the destabilisation complex is defined, we may form another $\mathsf{C}^\text{rev}$ by inverting the braiding, and one can show that their destabilisation complexes are related by $W_\bullet(\mathsf C;n) \cong W_\bullet(\mathsf C^{\mathrm{rev}};n)^{\mathrm{op}}$, where $(-)^{\mathrm{op}}$ denotes the opposite semisimplicial set, formed by reversing the order of the face maps. It follows that their geometric realisations are homeomorphic.
\end{rem}

\subsubsection{Congruence subgroups of $T_n$}\label{sec:congruence}
We write $T_n[2]$ for the subgroup of $T_n$ of matrices reducing to the identity mod $2$. Under the isomorphism of \cref{Qn is symplectic group}, the group $T_n[2]$ is identified with the principal level $2$ congruence subgroup of $\Sp_{n-1}(\Z)$, i.e.\ the kernel of $\Sp_{n-1}(\Z)\to \Sp_{n-1}(\Z/2)$. 

The action of the symmetric group $\mathfrak S_n$ on $\Z^{\oplus n}$ does \emph{not} preserve the form $\langle -,-\rangle$, but it does preserve the reduction of this form mod $2$. Thus it makes sense to define $Q_n$ to be the subgroup of $T_n$ consisting of matrices reducing to a permutation matrix mod $2$. We have a chain of subgroups $T_n[2]\subset Q_n \subset T_n$, and two short exact sequences
\begin{align*}
	1 \to T_n[2] \to & \, Q_n \to \mathfrak S_n \to 1, \\
	1 \to T_n[2] \to & \, T_n \to \Sp_{n-1}(\Z/2) \to 1.
\end{align*} 
We will explain why the mod $2$ reduction map $T_n \to \Sp_{n-1}(\Z/2)$ is surjective in \cref{Tsurj}.
We write ${\mathsf Q}$ for the monoidal subgroupoid $\{Q_n\} \subset \{T_n\}$, which inherits a braiding from $\mathsf T$ by observing that the matrix representing the braiding \eqref{eq:braiding2} reduces to a permutation matrix mod 2.

The following result is \cite[Corollary C]{bloomquist}. To see that \cref{bps} agrees with the statement of \cite[Corollary C]{bloomquist}, note that $Q_n=T_n$ for $n \leq 3$, since the subgroup $\mathfrak S_n \subset \Sp_{n-1}(\Z/2)$ is the whole group in these cases. The case when $n$ is odd in \cref{bps} is due to A'Campo \cite{a'campo}. 

\begin{thm}[A'Campo, Bloomquist--Patzt--Scherich] \label{bps}
 The image of the integral Burau representation $\beta_n \to \mathrm{Sp}_{n-1}(\Z)$ is the congruence subgroup $Q_n$.    
\end{thm}
 
	\subsection{The complexes}\label{sec:the complexes}

With the braiding on $\mathsf T$ and $\mathsf Q$ in place, the destabilisation complexes $W_\bullet(\mathsf T;n)$ and $W_\bullet(\mathsf Q;n)$ are well-defined. Let us first make these complexes explicit. We focus on $W_\bullet(\mathsf Q;n)$, but it will be clear that the same description works mutatis mutandis in both cases.

\begin{lem}\label{stab}The  $Q_n$-stabiliser of $e_n \in \Z^{\oplus n}$ is isomorphic to $Q_{n-1}$, identified with the automorphisms of $\mathrm{Span}\{e_1,\ldots,e_{n-1}\} = \Z^{\oplus(n-1)}$. By induction, the stabiliser of $e_{n-p},\ldots,e_n$ is $Q_{n-p-1}$. Similarly for $T_n$ and $T_n[2]$. 
\end{lem}

\begin{proof}Note that $$\mathrm{Span}\{e_1,\ldots,e_{n-1}\}=\{u \in \mathbb Z^{\oplus n} : \langle u,e_n\rangle = \phi(u)\}.$$
Indeed, it is clear that the left-hand side is included in the right-hand side. But it is also clear that both sides are free abelian submodules of $\Z^{\oplus n}$ of rank $(n-1)$, and since the left-hand side is a direct summand of $\Z^{\oplus n}$, both must be equal. Thus the stabiliser of $e_n$ also preserves 
	$\mathrm{Span}\{e_1,\ldots,e_{n-1}\}$, which is equipped with the same bilinear form $\langle-,-\rangle$ and linear form $\phi$ as $\mathbb Z^{\oplus (n-1)}$. \end{proof}

\begin{defn}A tuple $(u_1,\dots,u_n)$ of vectors in $\Z^{\oplus n}$ is a \emph{$\mathsf Q$-basis}, if there is a matrix $A \in Q_n$ such that $Ae_i=u_i$ for $i=1,\dots,n$. A tuple $(u_{n-p},\dots,u_n) \in \Z^{\oplus n}$, where $p=0,\dots,n-1$, is called a \emph{partial $\mathsf Q$-basis} if it can be completed to a $\mathsf Q$-basis $(u_1,\dots,u_n)$.
\end{defn}

\begin{lem}\label{description of destab}
	The destabilisation complex $W_\bullet(\mathsf Q;n)$ is isomorphic to the semisimplicial set whose $p$-simplices are partial $\mathsf Q$-bases of size $p+1$, and whose face maps are given by deleting entries.
\end{lem}

 \begin{proof}
 	We need to unwind definitions from \cite[Section 2]{randalwilliamswahl}. The set $W_{p}(\mathsf Q;n)$ 
 	is the right coset $Q_n/Q_{n-p-1}$
 	where $Q_{n-p-1}$ sits as a subgroup of $Q_n$ as one factor of the subgroup $Q_{n-p-1} \times Q_{p+1}$. \cref{stab} implies that $W_{p}(\mathsf Q;n)$ is in bijection with the set of all partial $\mathsf Q$-bases of size $(p+1)$, via the function
 	\begin{equation}\label{eq:bijection}
 	\left( A \in Q_n = \mathrm{Aut}_{\mathsf Q}(\Z^{\oplus n}) \right) \mapsto (A e_{n-p},\dots, A e_n).
 	\end{equation}
 	We want to describe the face maps $d_i \colon W_{p}(\mathsf Q;n) \to W_{p-1}(\mathsf Q;n)$ in terms of the identification \eqref{eq:bijection}. The $0^{\text{th}}$ face map is the quotient map
 	\[ d_0\colon W_{p}(\mathsf Q;n) =Q_n/Q_{n-p-1} \longrightarrow Q_n/Q_{n-p} = W_{p-1}(\mathsf Q;n), \]
 	which corresponds under \eqref{eq:bijection} to $d_0(u_{n-p},\dots,u_n) = (u_{n-p+1},\dots,u_n)$.
 	
 	To describe the higher face maps, we note that $Q_{p+1}$ acts on $W_p(\mathsf Q;n)$ by right multiplication. Consider the element $c_{i,1}^{-1} \oplus \mathrm{id}_{p-i}$ in $Q_{p+1}$, where $c_{i,1}$ is the element of $Q_{i+1}$ defined by the braiding \eqref{eq:braiding2}. Then $d_i\colon W_{p}(\mathsf Q;n) \to W_{p-1}(\mathsf Q;n)$ is defined by first applying $c_{i,1}^{-1} \oplus \mathrm{id}_{p-i}$ and then $d_0$.\footnote{In \cite[Definition 2.1]{randalwilliamswahl}, $d_i$ is defined as the precomposition by $\mathrm{id}_{\Z^{\oplus i}} \oplus \iota_{\Z^{\oplus 1}} \oplus \mathrm{id}_{\Z^{\oplus p-i}}$ in $U\mathsf Q$. Because $U \mathsf Q$ is pre-braided (see \cite[Definition 1.5 and Proposition 1.8(ii)]{randalwilliamswahl}), $\mathrm{id}_{\Z^{\oplus i}} \oplus \iota_{\Z^{\oplus 1}} \oplus \mathrm{id}_{\Z^{\oplus p-i}} = c^{-1}_{i,1} \circ(\iota_{\Z^{\oplus 1}} \oplus \mathrm{id}_{\Z^{\oplus i}} \oplus \mathrm{id}_{\Z^{\oplus p-i}})$.} Using \eqref{eq:braiding2} we calculate
 	\[ c^{-1}_{i,1}(x_1,\dots,x_{i+1}) = (x_2+2x_1,x_3-2x_1,\dots,x_{i+1}+2(-1)^ix_1,(-1)^ix_1). \]
 	In particular, $c^{-1}_{i,1} \oplus \mathrm{id}_{p-i}$ sends the standard basis vectors $e_j$ to
 	\[\begin{cases}
 		(-1)^ie_{i+1}+2v_i &\text{if $j=1$,}\\
 		e_{j-1} &\text{if $2\le j \le i+1$,}\\
 		e_j &\text{if $i+2\le j\le p+1$.}
 	\end{cases}\]
 	Using the identification \eqref{eq:bijection}, let $(u_{n-p},\dots, u_n) = (A e_{n-p},\dots, A e_n)$ be an element of $W_{p}(\mathsf Q;n)$. Then $(u_{n-p},\dots, u_n)$ is sent to $(w, u_{n-p}, \dots,\widehat{u}_{n-p+i}, \dots, u_{n})$
 	with $w= 2u_{n-p}-2u_{n-p+1}+\dots \pm 2u_{n-p+i-1}\mp u_{n-p+i}$
 	after applying $c^{-1}_{i,1} \oplus \mathrm{id}_{p-i}$, and so to $(u_{n-p}, \dots,\widehat{u}_{n-p+i},\dots,u_n)$ by further applying $d_0$.
 \end{proof} 

\begin{rem}\label{simplicial vs semisimplicial} Although the destabilisation complexes are naturally semisimplicial sets, we will prefer to think of $W_\bullet(\mathsf Q;n)$ as a simplicial complex. To see that this is harmless, note that if $\{u_i\}_{i \in I}$ is a set of vectors, then there is at most one total order of the set $I$ for which $\langle u_i,u_j\rangle = 1$ for $i<j$ in $I$; the existence of such an ordering may be considered as a property of a set of vectors. Thus partial $\mathsf Q$-bases are canonically ordered. We refer the reader to \cite[Section 2.1]{randalwilliamswahl} for a general discussion of simplicial complexes vs.~semisimplicial sets tailored to the setting of this paper; in terms of the discussion in \cite[top of page 558]{randalwilliamswahl}, the complex $W_\bullet(\mathsf Q;n)$ is an instance of situation (B), and we are free to treat $W_\bullet(\mathsf Q;n)$ as a simplicial complex or a semisimplicial set, whichever is most convenient.\end{rem}

 \begin{rem}The complexes $W_\bullet(\mathsf T;n)$ were proven to be slope $1/3$ connected  by Sierra and Wahl \cite{sierrawahl}. Their argument is rather different from the one we use for $W_\bullet(\mathsf Q;n)$, and does not appear to generalize to our setting. \end{rem}

In our proof of high connectivity, we will quote results from Charney and Mirzaii--van~der~Kallen. Their papers are neither written in the setting of simplicial complexes nor semisimplicial sets, but in terms of ``posets of sequences'', following Maazen \cite{maazen}. To aid in translation, a ``subposet of $\mathcal O(V)$ satisfying the chain condition'' is exactly the same as a directed simplicial complex, in the sense of the following definition, with vertex set $V$. 
\begin{defn}\label{maazen}
    A \emph{directed simplicial complex} is a semisimplicial set $W_\bullet$ such that $W_n \to W_0^{\times (n+1)}$ is injective for all $n$, and no simplex has a repeated vertex.
\end{defn}

That is: if $V$ is a set, then a simplicial complex with vertex set $V$ consists of a family of finite subsets of $V$, closed under taking subsets; a directed simplicial complex with vertex set $V$ consist of a family of finite sequences of elements of $V$ (without repetition), closed under taking subsequences.

\subsubsection{Transitivity arguments} 

Let $R$ be a ring, $J$ an ideal, $G$ an algebraic group over $R$. In many situations one wants to know surjectivity of the reduction mod $J$ homomorphism $G(R) \to G(R/J)$. One standard method for proving such results is to show that $G(R/J)$ is generated by elementary matrices, and lifting elementary matrices. Alternatively, if $R$ is e.g.\ the ring of integers in a global field, a powerful sledgehammer for proving surjectivity is the strong approximation theorem (see Platonov--Rapinchuk \cite[Chapter 7]{StrongApprox}). We will need only the case that $R=\Z$, $J=2\Z$, and $G$ is (a parabolic subgroup of) $\GL_n$ or $\Sp_{2g}$. Thus the following simple argument will suffice for our purposes.

	\begin{lem}\label{strong approx}Let $R$ be ring, and $\mathfrak m$ a maximal ideal, such that the induced map $R^\times \to (R/\mathfrak m)^\times$ is surjective. 
		Let $G$ be an $R$-split reductive group scheme, or a parabolic subgroup thereof.  Then $G(R)\to G(R/\mathfrak m)$ is surjective.
	\end{lem}
	
	\begin{proof}Let $B \subset G$ be a split Borel. Since $R/\mathfrak m$ is a field, we see from the Bruhat decomposition that $G(R/\mathfrak m)$ is generated by $B(R/\mathfrak m)$ and $W$, where $W$ is the Weyl group. (If $G$ is a parabolic subgroup of an ambient reductive group $G'$, then we take $W$ to be the corresponding parabolic subgroup of the Weyl group of $G'$.) So it suffices to show that $B(R)\to B(R/\mathfrak m)$ is surjective. But $B$ has a composition series whose factors are isomorphic to $\mathbb G_a$ or $\mathbb G_m$, and both $\mathbb G_a(R)\to \mathbb G_a(R/\mathfrak m)$ and  $\mathbb G_m(R)\to \mathbb G_m(R/\mathfrak m)$ are surjective.\end{proof}

\begin{cor}\label{Tsurj}
	The mod $2$ reduction map $\Sp_n(\Z) \to \Sp_{n}(\Z/2)$ is surjective.
\end{cor}

\begin{proof}
	If $n$ is even, then this follows immediately from \cref{strong approx}. If $n$ is odd, then $\Sp_n$ is the stabilizer of a unimodular vector $v$ inside $\Sp_{n+1}$. Let $P \subset \Sp_{n+1}$ be the parabolic subgroup stabilising the line spanned by $v$. Take $A \in \Sp_n(\Z/2) = P(\Z/2)$. By \cref{strong approx}, we can lift $A$ to $B \in P(\Z)$. Both $B$ and $-B$ are lifts of $A$, and exactly one of them lies in $\Sp_n(\Z)$.
\end{proof}

\begin{rem}The reduction-mod-$m$ homomorphism $\GL_n(\Z)\to\GL_n(\Z/m)$ is surjective precisely for $m=2,3$, which is exactly the cases covered by \cref{strong approx}. In general, the image consists of the matrices with determinant $\pm 1$. By contrast, $\mathrm{SL}_n(\Z)\to \mathrm{SL}_n(\Z/m)$ is surjective for all $n$ and $m$, by strong approximation.
\end{rem}

Denote by $\Sp_{2g}(R,J)$ the subgroup of $\Sp_{2g}(R)$ of matrices reducing to the identity mod $J$. The following is part $(2.4)_n$ of \cite[Theorem 2.9]{vaserstein}.

\begin{thm}[Vaserstein]\label{vaserstein-thm1}Let $m$ be the Bass stable rank of $(R,J)$. If $g \geq (m+1)/2$, then the group $\mathrm{Sp}_{2g}(R,J)$ acts transitively on the set of unimodular vectors in $R^{2g}$ whose reduction mod $J$ is the vector $(1,0,0,\dots)$. 
\end{thm}

\begin{rem}\label{stable rank of Z}The Bass stable rank of $(\Z,2\Z)$ is $2$ (\cite[Example on page 14]{BassPaper} and \cite[V.3.2]{BassKthy}).
        
\end{rem}

When $(R,J)=(\Z,2\Z)$, we can do a little better than \cref{vaserstein-thm1}. To be consistent with the notation $T_n[2]$ introduced in \S\ref{sec:congruence} we set $\Sp_{2g}(\Z,2\Z) = \Sp_{2g}(\Z)[2]$.

\begin{lem}\label{vaserstein-thm}The group $\mathrm{Sp}_{2g}(\Z)[2]$ acts transitively on the set of unimodular vectors in $\Z^{2g}$ whose reduction mod $2$ is a fixed vector in $(\Z/2)^{2g}$.     
\end{lem}

\begin{proof}Using \cref{strong approx} and that $\Sp_{2g}(\Z/2)$ acts transitively on the nonzero vectors of $(\Z/2)^{2g}$, it suffices to prove this when the reduction mod $2$ equals $(1,0,0,\dots)$. If $g > 1$, the result follows from \cref{vaserstein-thm1} and \cref{stable rank of Z}. We give a direct argument when $g=1$: if $(x,y) \in \Z^2$ is unimodular with $x$ odd and $y$ even, then we prove by induction on the $\ell_1$-norm $|x|+|y|$ that there is an element $A$ of $\Sp_2(\Z)[2]$  carrying $(x,y)$ to $(1,0)$. The base case is $|x|+|y|=1$, in which case $y=0$ and so $A = \pm \mathrm{id}$ works. For the induction step, if $|y|<|x|$ then $A = (\begin{smallmatrix}
    1 & \pm 2 \\ 0 & 1
\end{smallmatrix})$ carries $(x,y)$ to $(x \pm 2y,y)$, one of which has strictly smaller $\ell_1$-norm. Similarly when $|y|>|x|$.\end{proof}

\begin{defn}Let $M$ be a symplectic module. A \emph{partial isotropic basis} for $M$ is a partial basis $w_1,\ldots,w_k$ such that $\langle w_i,w_j\rangle=0$ for all $1\leq i,j \leq k$. A \emph{partial symplectic basis} is a partial basis $a_1,b_1,\ldots,a_k,b_k$ such that $a_1,\dots,a_k$ and $b_1,\dots,b_k$ are both isotropic, and $\langle a_i,b_j\rangle = \delta_{ij}$. 
\end{defn}
 
	\begin{lem}\label{vaserstein}
		The group $\mathrm{Sp}_{2g}(\Z)[2]$ acts transitively on the set of partial isotropic bases (or partial symplectic bases) in $\Z^{2g}$ whose reduction mod $2$ is a fixed partial isotropic basis (or partial symplectic basis) in $(\Z/2)^{2g}$. 

	\end{lem}

\begin{proof} This proof is very similar to \cite[Lemma 5.3]{MvdK}.

It is enough to assume that the fixed partial isotropic basis (or partial symplectic basis) in $(\Z/2)^{2g}$ is given by the standard vectors $e_1, \dots, e_k$ (or $e_1,f_1,\dots,e_k,f_k$), because by \cite[Lemma 5.3]{MvdK} there is a symplectic matrix $A\in \Sp_{2g}(\Z/2)$ that sends the fixed partial basis to the standard one, and we can lift this symplectic matrix to $\Sp_{2g}(\Z)$.

We will prove the statement by induction on $k$, the length of the partial bases. If $k=1$, the partial isotropic case is given by \cref{vaserstein-thm}. For the partial symplectic case, let $(a_1,b_1)$ be a symplectic pair in $\Z^{2g}$ that reduces mod $2$ to $(e_1,f_1)$ in $(\Z/2)^{2g}$. We already know that we can send $a_1$ to $e_1$ by a matrix $A\in \Sp_{2g}(\Z)[2]$. It follows that 
\[Ab_1= 2\alpha_1 e_1 +f_1 + \sum_{i=2}^{g} ( 2\alpha_i e_i + 2\beta_i f_i)\]
for some $\alpha_i,\beta_i \in \Z$. 
The map
\begin{align*}
    e_1 & \longmapsto e_1 & e_i &\longmapsto e_i + 2\beta_ie_1  \qquad i>1\\
    f_1 & \longmapsto f_1 - 2\alpha_1e_1 -\sum_{i=2}^{g} ( 2\alpha_i e_i + 2\beta_i f_i) & f_i &\longmapsto f_i - 2\alpha_ie_1 \qquad  i>1
\end{align*}
is an element in $\Sp_{2g}(\Z)[2]$ that sends $Ab_1$ to $f_1$.

For $k>1$, it is enough to show the partial symplectic case, because every partial isotropic basis $a_1,\dots,a_k$ is part of a partial symplectic basis $(a_1,b_1),\dots, (a_k,b_k)$. By induction, we can assume that $(a_i,b_i) = (e_i,f_i)$ for $i<k$. Now, $(a_k,b_k)$ is a symplectic pair in the complement of $(e_1,f_1),\dots,(e_{k-1},f_{k-1})$. We can send $(a_k,b_k)$ to $(e_k,f_k)$ by a symplectic matrix of the complement that reduces to the identity mod $2$. And so we are done.
\end{proof}



\begin{defn}\label{definition of rho} For a vector $a = \sum_{i=1}^n a_i e_i  \in \Z^{\oplus n}$, let $\rho(a) = \{i \in \{1,\ldots,n\} : a_i \equiv 1 \pmod 2\}$. 
\end{defn}

In the proofs of the following \cref{orbit} and \cref{transitive 2}, we will repeatedly make use of the observations made in the proofs of \cref{Qn is symplectic group} and \cref{distinguished vector}.

\begin{lem}\label{orbit}Let $k>1$. The orbit of $(e_k,\ldots,e_n)$ under $Q_n$, i.e.\ the set of partial $\mathsf Q$-bases of cardinality $n-k+1$, is the set of vectors $(u_k,\ldots,u_n)$ such that\begin{enumerate}
		\item \label{item1} $\{u_k,\ldots,u_n,v_n\}$ is a partial basis for $\Z^{\oplus n}$ as an abelian group.
		\item $\langle u_i,u_j\rangle = 1$ for $i<j$.
		\item The sets $\rho(u_k), \dots, \rho(u_n)$ are disjoint singletons. 
		\item \label{item4} $\phi(u_i) = 1$ for all $i$.
	\end{enumerate}
    The vector $v_n$ in \eqref{item1} is the distinguished vector of \cref{distinguished vector}.
\end{lem}

\begin{proof}It is clear that every element in the orbit of $(e_k,\dots,e_n)$ satisfies \eqref{item1}---\eqref{item4}. Now take a tuple $(u_k,\ldots,u_n)$ satisfying conditions \eqref{item1}---\eqref{item4}. We construct $A \in Q_n$ taking each $u_i$ to $e_i$. Since $Q_n\to \mathfrak S_n$ is surjective, we can assume $\rho(u_i)=\{i\}$ for all $i$. It then suffices to construct  $A\in T_n[2]$ taking $u_i$ to $e_i$ for all $i$. It suffices to prove this for $i=n$; once we know this, we may without loss of generality take $u_n=e_n$, in which case \cref{stab} gives the result by induction. 	

Case 1: $n$ is even.  Then $\langle-,-\rangle$ is a symplectic form on $\Z^{\oplus n}$, and $(v_n,e_n)$ and $(v_n,u_n)$ are hyperbolic pairs with the same reduction mod $2$ in this symplectic module. By \cref{vaserstein} there is an element of $\mathrm{Sp}_n(\Z)[2]$ carrying $(v_n,u_n)$ to $(v_n,e_n)$. But $T_n[2]$ is the $v_n$-stabiliser inside $\mathrm{Sp}_n(\mathbb Z)[2]$, as observed in \cref{Qn is symplectic group}, so we have found an element of $T_n[2]$ taking $u_n$ to $e_n$.
 
	Case 2: $n$ is odd. Then there is an isomorphism $T_n[2] \cong \mathrm{Sp}_{n-1}(\mathbb Z)[2]$ induced by the action of $T_n[2]$ on $\ker(\phi)$, which is a symplectic module of rank $n-1$, as observed in \cref{Qn is symplectic group}. By \cref{vaserstein} there is an element $A \in T_n[2]$ taking $u_n-v_n$ to $e_n-v_n$, as both are unimodular vectors in $\ker(\phi)$ with the same reduction mod $2$. But $A$ also fixes $v_n$. \end{proof}

\begin{rem}
	When $k=1$, the orbit of $(e_1,\ldots,e_n)$ can be described by replacing condition \eqref{item1} in \cref{orbit} with the condition that $\{u_1,\ldots,u_n\}$ is a basis and that $u_1-u_2+... + (-1)^{n-1} u_n = v_n$.  
\end{rem}



\begin{lem}
    \label{transitive 2}Let $k<n/2$. The orbit of $\{e_1-e_2,e_3-e_4,\ldots,e_{2k-1}-e_{2k}\}$ under $Q_n$ is the set of vectors $\{x_1,\ldots,x_k\}$ such that\begin{enumerate}
		\item \label{item11}$\{x_1,\ldots,x_k,v_n\}$ is a partial basis for $\Z^{\oplus n}$ as an abelian group.
		\item $\langle x_i,x_j\rangle = 0$ for all $i,j$.
		\item The sets $\rho(x_1), \dots, \rho(x_k)$ are disjoint two-element sets. 
		\item \label{item44} $\phi (x_i) = 0$ for all $i$.
	\end{enumerate}
\end{lem}

\begin{proof}It is clear that every element in the orbit of $\{e_1-e_2,e_3-e_4,\ldots,e_{2k-1}-e_{2k}\}$  satisfies \eqref{item11}---\eqref{item44}. Take $\{x_1,\ldots,x_k\}$ satisfying \eqref{item11}---\eqref{item44}. We will construct $A \in Q_n$ taking each $x_i$ to $e_{2i-1}-e_{2i}$. Since $Q_n\to \mathfrak S_n$ is surjective, by first acting with some lift of an appropriate permutation we may assume that $\rho(x_i)=\{2i-1,2i\}$ for all $i$. It then suffices to construct an element of $T_n[2]$ taking each $x_i$ to $e_{2i-1}-e_{2i}$. 

Case 1: $n$ is even. Then $\langle-,-\rangle$ is a symplectic form on $\Z^{\oplus n}$, and $\{x_1,\ldots,x_k,v_n\}$ is a partial isotropic basis, since $\langle v_n,-\rangle = \phi(-)$ for $n$ even, as observed in \cref{distinguished vector}. By \cref{vaserstein} there is an element of $\mathrm{Sp}_n(\mathbb Z)[2]$  carrying $\{x_1,\ldots,x_k,v_n\}$ to $\{e_1-e_2,\ldots,e_{2k-1}-e_{2k},v_n\}$. But $T_n[2]$ is the $v_n$-stabiliser inside $\mathrm{Sp}_n(\mathbb Z)[2]$, as observed in \cref{Qn is symplectic group}. So we have found an element of $T_n[2]$ taking each $x_i$ to $e_{2i-1}-e_{2i}$. 

	Case 2: $n$ is odd. Then there is an isomorphism $T_n[2] \cong \mathrm{Sp}_{n-1}(\mathbb Z)[2]$ induced by the action of $T_n[2]$ on $\ker(\phi)$, which is a symplectic module of rank $n-1$, as observed in \cref{Qn is symplectic group}. The family $\{x_1,\ldots,x_k\}$ is a partial isotropic basis in this symplectic module. So the result follows by \cref{vaserstein}.

\end{proof}

\subsubsection{The complexes used in the proofs}
Let us now define the complexes that will play the main role in our arguments. Note that here and later in the paper, $\Z^{\oplus n}$ always comes equipped with the linear form $\phi\colon \Z^{\oplus n} \to \Z$ and the bilinear form $\langle - , - \rangle \colon \Z^{\oplus n} \otimes \Z^{\oplus n} \to \Z^{\oplus n}$ from the beginning of \cref{sec:evenandoddsp}.

\begin{defn}
	Let $Z_n$ be the simplicial complex whose $p$-simplices are sets $\{u_0,\ldots,u_p\}$ of vectors in $\Z^{\oplus n}$, such that:
	\begin{enumerate}
		\item $\{u_0,\ldots,u_p,v_n\}$ is a partial basis for $\Z^{\oplus n}$ as an abelian group.
		\item The sets $\rho(u_0), \dots, \rho(u_p)$ are pairwise disjoint.
	\end{enumerate}
\end{defn}

The complex $Z_n$ itself will not actually play any role in the arguments of the paper; however, all the complexes we care about will be subcomplexes of it.

\begin{defn}
	We let $Y_n \subset Z_n$ be the subcomplex consisting of simplices $\{u_0,\ldots,u_p\}$ such that $\phi(u_i)=0$, and $|\rho(u_i)|=2$, for all $i$. 
\end{defn}

\begin{defn}
	We let $\IX_n \subset Y_n$ be the subcomplex consisting of simplices $\{u_0,\ldots,u_p\}$ such that $\langle u_i,u_j\rangle=0$ for all $i,j$. 
\end{defn}

\begin{rem}\label{transitive rem}\cref{transitive 2} says precisely that $Q_n$ acts transitively on $p$-simplices of $\IX_n$. The condition $p<n/2-1$ is automatically fulfilled because otherwise $u_0,\dots,u_p,v_n$ cannot be a partial basis after reducing mod $2$.\end{rem}

\begin{defn}
	We let $X_n \subset Z_n$ be the subcomplex consisting of simplices $(u_0,\ldots,u_p)$ such that $\phi(u_i)=1$ and $|\rho(u_i)|=1$, for all $i$, and 
	$\langle u_i,u_j\rangle=1$ for all $i<j$. 	 
\end{defn}

\begin{rem}\label{transitive rem X}\cref{orbit} says precisely that $Q_n$ acts transitively on $p$-simplices of $X_n$.\end{rem}

\begin{lem}\label{rem:X_n is destabilisation}
	The complex $X_n$ is isomorphic to the $(n-2)$-skeleton of the destabilisation complex $W_\bullet({\mathsf Q};n)$. 
\end{lem}
\begin{proof}
This follows from \cref{orbit} and \cref{description of destab}.\end{proof}

\begin{rem}
    While simplices in $Z_n$ are unordered, we think of simplices of $X_n$ as being ordered, just as in \cref{simplicial vs semisimplicial}. In particular, the following definition applies to $X_n$ (but not $\IX_n$ or $Y_n$). 
\end{rem}

\begin{defn}For $W$ a directed simplicial complex, the \emph{left-link} of a simplex $\sigma = (v_0,\ldots,v_p) \in W_p$ is the directed simplicial complex $\mathrm{LLk}_{W}(\sigma)$  whose $q$-simplices are $(u_0,\dots,u_q) \in W_q$ such that $(u_0,\dots,u_q,v_0,\ldots,v_p) \in W_{p+q+1}$.  
\end{defn}

\begin{lem}\label{leftlink}If $\sigma$ is a $p$-simplex of $X_n$, then the left-link of $\sigma$ in $X_n$ is isomorphic to $X_{n-p-1}$. 
\end{lem}

\begin{proof}By \cref{transitive rem X}, $Q_n$ acts transitively on the set of $p$-simplices of $X_n$. Using this action, we may assume that $\sigma = (e_{n-p},\ldots,e_n)$. 

There is an embedding $\Z^{\oplus (n-p-1)}\hookrightarrow \Z^{\oplus n}$, taking the basis vector $e_i$ to the basis vector $e_i$. The embedding is compatible with $\langle-,-\rangle$ and the map $\phi$, and if a set of vectors in $\Z^{\oplus (n-p-1)}$ forms a partial basis with $v_{n-p-1}$ then their images in $\Z^{\oplus n}$ form a partial basis with $v_n$. Thus there is induced an inclusion $X_{n-p-1} \hookrightarrow X_n$. 

We now check that the two subcomplexes $X_{n-p-1}$ and $\mathrm{LLk}_{X_n}(\sigma)$ of $X_n$ are equal. Choose a simplex $\tau=(u_0,\dots,u_q)$ in $X_n$. Then $\tau$ lies in $\mathrm{LLk}_{X_n}(\sigma)$ if and only if the following conditions are satisfied:
\begin{enumerate}[(A)]
    \item \label{a}$\langle u_i,e_j\rangle = 1$ for $0\le i \le q$ and $n-p\le j\le n$.
    \item \label{b}$\{u_0,\dots,u_q,e_{n-p},\dots,e_n,v_n\}$ is a partial basis for $\Z^{\oplus n}$ as an abelian group.
    \item \label{c}The sets $\rho(u_0), \dots, \rho(u_q),\rho(e_{n-p}),\dots,\rho(e_n)$ are disjoint subsets of $\{1,\dots,n\}$.
\end{enumerate}
Similarly, $\tau$ lies in $X_{n-p-1}$ if and only if  the following conditions are satisfied:
\begin{enumerate}[(a)]
     \item \label{aa}$u_i \in \mathrm{Span}\{e_1,\ldots,e_{n-p-1}\}$ for $0 \leq i \leq q$.
        \item \label{bb}$\{u_0,\dots,u_q,v_{n-p-1}\}$ is a partial basis for $\Z^{\oplus (n-p-1)}$ as an abelian group.
        \item \label{cc}The sets $\rho(u_0), \dots, \rho(u_q)$ are disjoint subsets of $\{1, \dots, n-p-1\}$.
\end{enumerate}
    Now note that $\mathrm{Span}\{e_1,\ldots,e_{n-p-1}\} = \{u \in \Z^{\oplus n} : \langle u,e_j\rangle = \phi(u) \text{ for } j=n-p,\ldots,n\}$. (This is the same argument as \cref{stab}.) Since $\phi(u_i)=1$ for all $i$, we see that \ref{a} and \ref{aa} are equivalent. It is moreover straightforward to see that \ref{b} and \ref{bb} are equivalent, and that \ref{c} and \ref{cc} are equivalent.
\end{proof}

\begin{defn}
    For $\sigma = \{u_0,\ldots,u_p\}$ a simplex in $Z_n$, we let $Z_n(\sigma^\perp)$ be the subcomplex of $\Lk_{Z_n}(\sigma)$ consisting of simplices $\{x_0,\ldots,x_q\}$ such that $\langle x_i,u_j\rangle=0$ for all $i,j$. For $W \subset Z_n$ a subcomplex (which will always be one of $X_n$, $Y_n$, or $\IX_n$), we write similarly $W(\sigma^\perp) = W \cap Z_n(\sigma^\perp)$. 
\end{defn}

	\begin{lem}\label{IXtau}
		Let $\sigma$ be a $p$-simplex in $X_n$. Then $\IX_n(\sigma^\perp)$ is isomorphic to $\IX_{n-p-1}$.
	\end{lem}
	
	\begin{proof}
    As in the previous lemma, we can again assume that $\sigma = \{e_{n-p}, \dots, e_n\}$, and the embedding $\Z^{\oplus (n-p-1)} \hookrightarrow \Z^{\oplus n}$ induces an embedding of $\IX_{n-p-1}$ in $\IX_n$.  We check that the two subcomplexes $\IX_n(\sigma^\perp)$ and $\IX_{n-p-1}$ of $\IX_n$ are equal. Fix a simplex $\tau = \{u_0,\dots,u_q\}$ in $\IX_n$. Then $\tau$ lies in $\IX_n(\sigma^\perp)$ if and only if:
  \begin{enumerate}[(A)]
      \item $\langle u_i,e_j \rangle =0$ for $0\leq i \leq q$ and $n-p \leq j \leq n$.
      \item $\{u_0,\dots,u_q,e_{n-p},\dots,e_n,v_n\}$ is a partial basis for $\Z^{\oplus n}$ as an abelian group.
      \item The sets $\rho(u_0), \dots, \rho(u_q),\rho(e_{n-p}),\dots,\rho(e_n)$ are disjoint subsets of $\{1,\dots,n\}$.
  \end{enumerate}
Similarly, $\tau$ lies in $\IX_{n-p-1}$ if and only if:
\begin{enumerate}[(a)]
     \item $u_i \in \mathrm{Span}\{e_1,\ldots,e_{n-p-1}\}$  for $0 \leq i \leq q$.
        \item $\{u_0,\dots,u_q,v_{n-p-1}\}$ is a partial basis for $\Z^{\oplus (n-p-1)}$ as an abelian group.
        \item The sets $\rho(u_0), \dots, \rho(u_q)$ are disjoint subsets of $\{1, \dots, n-p-1\}$.
\end{enumerate}
We are thus exactly in the same situation as the previous lemma.
  	\end{proof}

 The following construction goes back at least to Maazen \cite{maazen}. It is a special case of the notion of a \emph{complete join} \cite[Section 3]{hatcherwahl}.

\begin{defn}\label{bracket}If $W$ is a simplicial complex with vertex set $W_0$, and $S$ is a set, then we define $W\langle S\rangle$ to be the simplicial complex with vertex set $W_0 \times S$, where $\{(w_0,s_0),\ldots,(w_p,s_p)\}$ is a $p$-simplex of $W\langle S\rangle$ if and only if $\{w_0,\ldots,w_p\}$ is a $p$-simplex of $W$. Similarly, if $W$ is a semisimplicial set, then $W \langle S\rangle$ denotes the semisimplicial set whose $n$-simplices are $W_n \times S^{\times n+1}$, with the obvious face maps. This causes no ambiguity when $W$ may be considered as either a simplicial complex or a semisimplicial set.
\end{defn}

	\begin{lem}\label{Xsigma}
		Let $\sigma$ be a $p$-simplex in $\IX_n$. Then $X_n(\sigma^\perp)$ is isomorphic to $X_{n-2p-2}\langle(2\Z)^{p+1}\rangle$.
	\end{lem}
	
	\begin{proof}
To ease notation in the following proof, we define $$h_i := e_{n-2p-1+2i}-e_{n-2p+2i},$$ for $i=0,\dots , p$. By \cref{transitive rem}, we can use the action of $Q_n$ on $Z_n$ (which preserves the subcomplexes $\IX_n$ and $X_n$ and is transitive on the set of $p$-simplices of $\IX_n$) to assume that $\sigma = \{h_0,\dots,h_p\}.$ 

Consider the map $\Z^{\oplus (n-2p-2)} \hookrightarrow \Z^{\oplus n}$ taking the basis vector $e_i$ to the basis vector $e_i$, and the map $(2\Z)^{p+1} \hookrightarrow \Z^{\oplus n}$ defined by
$$ (a_0,\dots,a_p) \mapsto \sum_{i=0}^p a_i \cdot h_i.$$
The sum of these two maps defines a function $\iota:\Z^{\oplus (n-2p-2)} \times (2\Z)^{p+1} \to \Z^{\oplus n}$, and we claim that it induces an embedding
\[ X_{n-2p-2}\langle (2\Z)^{p+1}\rangle \hookrightarrow X_n.\]
Indeed, one checks that for all $u,v \in \Z^{\oplus (n-2p-2)}$ and $a, b \in (2\Z)^{p+1}$, we have $\langle \iota(u,a),\iota(v,b)\rangle = \langle u, v\rangle$, that $\phi(\iota(u,a)) = \phi(u)$, and $\rho(\iota(u,a
)) = \rho(u)$. Moreover, if the vectors $u_0, \dots, u_q \in \Z^{\oplus (n-2p-2)}$ form a partial basis with $v_{n-2p-2}$, and $a_i \in (2\Z)^{p+1}$ are arbitrary, then $\iota(u_0,a_0),\dots, \iota(u_q,a_q)$ form a partial basis with $v_n$. 

We check that the two subcomplexes $X_n(\sigma^\perp)$ and $X_{n-2p-2}\langle (2\Z)^{p+1}\rangle $ of $X_n$ are equal. Take a simplex $\tau = \{u_0,\dots,u_q\} 	\in X_n$. Then $\tau$ lies in $X_n(\sigma^\perp)$ if and only if:
\begin{enumerate}[(A)]\setcounter{enumi}{3}
	 \item \label{dd}$\langle u_i,h_j \rangle =0$ for $0\leq i \leq q$ and $0 \leq j \leq p$.
	\item \label{ee}$\{u_0,\dots,u_q,h_0,\dots,h_p,v_{n}\}$ is a partial basis for $\Z^{\oplus n}$ as an abelian group.
	\item \label{ff}The sets $\rho(u_0), \dots, \rho(u_q),\rho(h_0),\dots,\rho(h_p)$ are disjoint subsets of $\{1,\dots,n\}$.
\end{enumerate}
Write each vertex $u_i$ of $\tau$ as $x_i+y_i$, with $x_i \in \mathrm{Span}\{e_1,\dots,e_{n-2p-2}\}$ and $y_i \in \mathrm{Span}\{e_{n-2p+1},\dots,e_n\}$. Then $\tau$ lies in $X_{n-2p-2}\langle (2\Z)^{p+1}\rangle $ if and only if:
\begin{enumerate}[(a)]\setcounter{enumi}{3}
	\item \label{d} $y_i$ lies in $\mathrm{Span}\{2h_0,\dots,2h_p\}$ for all $0 \leq i \leq q$.
	\item \label{e}$\{x_0,\dots,x_q,v_{n-2p-2}\}$ is a partial basis for $\Z^{\oplus (n-2p-2)}$ as an abelian group.
	\item \label{f}The sets $\rho(x_0), \dots, \rho(x_q)$ are disjoint subsets of $\{1,\dots,n-2p-2\}$.
\end{enumerate}
Note now that the orthogonal complement of $\{h_0,\dots,h_p\}$ is $\{e_1,\dots,e_{2n-2p-2},h_0,\dots,h_p\}$. It follows that the conjunction of \ref{dd} and \ref{ff} is equivalent to the conjunction of \ref{d} and \ref{f}. Moreover, given that each $y_i$ is in the span of the vectors $h_0,\dots,h_p$, it is elementary that \ref{ee} is equivalent to \ref{e}. 
	\end{proof}
 
	\subsection{The connectivity argument}
	
	We now proceed to the argument for high connectivity of $X_n$. The road map is as follows:
	
	\begin{itemize}
		\item Deduce from Charney's work and a bad simplex argument that $Y_n$ is highly connected.
		\item Use a nerve lemma of Mirzaii--van der Kallen to deduce high connectivity of $\IX_n$ from high connectivity of $Y_n$.
		\item Use the 
  nerve lemma again to deduce high connectivity of $X_n$ from high connectivity of $\IX_n$. 
	\end{itemize}
	\newcommand{\dir}[1]{{#1}^{\mathrm{ord}}}
	\begin{defn}
	    Let $W$ be a simplicial complex. We define $\dir W$ to be the directed simplicial complex whose $p$-simplices are the $p$-simplices of $W$ with an order of their vertices. (Thus there are $(p+1)!$ simplices of $\dir W$ for each $p$-simplex of $W$.)
	\end{defn}

\begin{rem}\label{retract lemma}
    The geometric realization of $\dir W$ contains the geometric realization of $W$ as a retract, by arbitrarily choosing an ordering of the vertices of $W$. Hence if $\dir W$ is $d$-connected, then so is $W$. 
\end{rem}

We will need the notion of a complex being \emph{weakly Cohen--Macaulay}, or wCM for short.

	\begin{defn}A simplicial complex $W$ is \emph{wCM of dimension $d$}  if for all simplices $\sigma \subset W$, $\mathrm{Lk}_W(\sigma)$ is $(d-\dim(\sigma)-2)$-connected. In particular, taking $\sigma=\varnothing$, $W$ itself must be $(d-1)$-connected. 
	\end{defn}

 \begin{defn}\label{definition of PB}
     Let $R$ be a ring, $J$ an ideal. Let $V$ be a finitely generated free $R$-module, and $\mathcal B$ a partial basis of $V/J=V \otimes_R R/J$. Let $\PB(V;\mathcal B)$ be the simplicial complex with simplices $\{x_{0},\dots, x_{p}\}$ where
		\begin{enumerate}
			\item $x_{0},\dots, x_p$ is a partial basis of $V$.
			\item $\{x_{0},\dots, x_p\}$ modulo $J$ is a subset of $\mathcal B$.
		\end{enumerate}
    The complex $\PB(V;\mathcal B)$ is a subcomplex of the well-known complex of partial bases $\PB(V)$, which drops condition (2).
 \end{defn}

The following theorem follows directly from Charney \cite[Theorem on p.~2094]{charneycongruence}.

 \begin{thm} \label{Charney1}
 Let $m$ be the Bass stable rank of $(R,J)$. Let $\mathcal B$ be the first $t$ basis vectors $\{e_1,\dots,e_t\}$ in $(R/J)^n$, with $t\leq n$. Then $\PB(R^n;\mathcal B)$ is wCM of dimension $t-m$.
\end{thm}

\begin{proof}
Let $\sigma=\{v_0,\dots,v_p\}$ be a $p$-simplex of $\PB(R^n;\mathcal B)$. We need to see that $\Lk_{PB(R^n;\mathcal B)}(\sigma)$ is $(t-m-p-2)$-connected.

Let $v$ denote the ordered tuple $(v_0,\dots,v_p)$ of the vertices in $\sigma$ and $I=\ker(R^n \to (R/J)^n)$. Then the barycentric subdivision of $\dir{\mathrm{Lk}_{\PB(R^n;\mathcal B)}(\sigma)}$ is isomorphic to the geometric realization of the poset which Charney denotes $\mathcal O_J(R^t+I) \cap U_v$, and which she proves to be $(t-m-p-2)$-connected. (Note that the integer that Charney denotes $\mathrm{sdim}$ equals $m-1$.) This gives the result by \cref{retract lemma}. 
\end{proof}

Applying this theorem to $(R,J)= (\Z,2\Z)$ 
results in the following corollary.
	
	\begin{cor}\label{Charney}
		Let $V$ be a finitely generated free abelian group and $\mathcal B$ a partial basis of $V/2$. Then
		$\PB(V;\mathcal B)$ is wCM of dimension $(|\mathcal B|-2)$.
	\end{cor}

\begin{proof}Since $\GL_n(\Z/2)$ acts transitively on partial bases, and using \cref{strong approx}, it suffices to prove the corollary when $\mathcal B$ consists of the first $|\mathcal B|$ standard basis vectors. The result follows from \cref{Charney1} since $(\Z,2\Z)$ has Bass stable rank $2$ (\cref{stable rank of Z}).
\end{proof}

\begin{lem}\label{lem:complement}
    Let $M$ be a finitely generated free abelian group, and $U,V$ be direct summands of $M$. Let $\mathcal B$ be a partial basis for $V/2$. There exists a complement $W$ of $U$ in $M$ such that $\mathcal B \cap ((V \cap W)/2)$ has at least $|\mathcal B|-\operatorname{rank} U$ many elements.
\end{lem}

\begin{proof}Since $\GL(V)$ surjects onto $\GL(V/2)$ by \cref{strong approx}, we may lift $\mathcal B$ to a partial basis of $V$. After replacing $V$ with the span of this partial basis, we may assume that $\mathcal B$ is in fact a basis for $V/2$. Choose a map $p\colon M\to V$ splitting the inclusion. Let $T$ be the saturation of $p(U)$ in $V$; that is, $T$ consists of those elements of $V$ for which a nonzero multiple lies in $p(U)$. By construction, $V/T$ is torsion-free, so that $T$ is a direct summand of $V$. Also, $\operatorname{rank} T = \dim_\Q T\otimes\Q = \dim_\Q p(U)\otimes\Q \leq \dim_\Q U\otimes\Q = \operatorname{rank} U$.

Consider the surjection $V/2 \to (V/T)/2$. Since $\mathcal B$ is a basis for $V/2$, we can choose a subset $\mathcal L$ of $\mathcal B$ which maps to a basis of $(V/T)/2$. We have $$|\mathcal L| = \dim_{\Z/2}(V/T)/2 = \dim_{\Z/2} V/2 - \dim_{\Z/2} T/2 = |\mathcal B| - \operatorname{rank} T \geq |\mathcal B| - \operatorname{rank} U.$$ Let $\mathcal K$ be a basis for $T/2$, so $\mathcal K \cup \mathcal L$ is a basis for $V/2$.  

Denote by $\GL(V,T)$ the parabolic subgroup of $\GL(V)$ fixing $T$ setwise. The mod 2 reduction map $\GL(V,T)\to \GL(V/2,T/2)$ is surjective by \cref{strong approx}. It follows from surjectivity that we can lift the basis $\mathcal K\cup \mathcal L$ for $V/2$ to a basis $\mathcal K' \cup \mathcal L'$ for $V$, such that $\mathcal K'$ is a basis of $T$. In more detail, first choose a basis $\mathcal C$ of $V$ for which the first $\mathrm{rank}(T)$ vectors form a basis of the summand $T$. Let $A$ be the element in $\GL(V/2)$ that sends $\mathcal C/2$ to $\mathcal K \cup \mathcal L$. Then $A$ is in $\GL(V/2,T/2)$, since it maps $T/2$ to itself. Choose a preimage $A'\in \GL(V,T)$, using \cref{strong approx}. Then $A'$ applied to $\mathcal C$ gives a basis $\mathcal K' \cup \mathcal L'$ of $V$ such that $\mathcal K'$ is a basis of $T$ and reduces to $\mathcal K$ mod $2$, and $\mathcal L'$ reduces to $\mathcal L$ mod $2$.

Let $T' = \operatorname{Span} \mathcal L'$. Since $V=T \oplus T'$, and $p$ is surjective, we have $M=p^{-1}(T)\oplus p^{-1}(T')$. Now $U \subseteq p^{-1}(T)$ is a direct summand: indeed, this is equivalent to $U$ being equal to its saturation in $p^{-1}(T)$, and $U$ is by assumption equal to its saturation in the larger module $M$. Let $U'$ be a complement of $U$ in $p^{-1}(T)$, and let $W=U' \oplus p^{-1}(T')$. Then $W$ is a complement of $U$, and $V \cap W \supseteq T'$, so $\mathcal B \cap ((V \cap W)/2) \supseteq \mathcal L$, which we have seen has at least $|\mathcal B|-\operatorname{rank} U$ many elements.
\end{proof}

 We will apply a version of the ``bad simplex argument''. There are many variants of the bad simplex argument in the literature; the following statement is exactly \cite[Proposition 2.5]{grw-stability1}.

\begin{prop}\label{bad simplex}
	Let $X$ be a simplicial complex and $Y \subset X$ a full subcomplex. Suppose that there exists an integer $n$ such that whenever a nonempty simplex $\tau \subset X$ has no vertex in $Y$, then $\mathrm{Lk}_X(\tau)\cap Y$ is $(n-\dim(\tau)-1)$-connected. Then the pair $(X,Y)$ is $n$-connected.
\end{prop}
We now show the following generalization of \cref{Charney}. The case where $\sigma$ is in $\PB(M;\mathcal B)$ already appears in Charney's paper. 

\begin{lem}\label{CharneyCor}Let $M$ be a finitely generated free abelian group, $V$  a summand of $M$, and $\mathcal B$ a partial basis of $V/2$. Let $\sigma$ be a $q$-dimensional simplex of $\PB(M)$, possibly empty. Then $\PB(V;\mathcal B)\cap \Lk_{\PB(M)}(\sigma)$ is wCM of dimension $(|\mathcal B|-3-q)$.
\end{lem}

\begin{proof}
The case that $\sigma$ is empty is precisely \cref{Charney}, and we can therefore assume that $\sigma$ is nonempty, i.e. $q\ge 0$. We first show that $\PB(V;\mathcal B)\cap \Lk_{\PB(M)}(\sigma)$ is $(|\mathcal B|-4-q)$-connected. We prove this by induction on the rank of $M$; that is, we assume the claimed statement for all data $M$, $V$, $\mathcal B$, and $\sigma$ with strictly smaller value of $\operatorname{rank} (M)$. The base case is that $\operatorname{rank}(M) \le 2$, which implies that $|\mathcal B|-4-q \le -2$, and every space is $(-2)$-connected (even the empty space).

    Let $U$ be the summand of $M$ spanned by $\sigma$. Using \cref{lem:complement}, choose a complement $W$ of $U$ such that $\mathcal B' = \mathcal B \cap ((V \cap W)/2)$ has at least $|\mathcal B|-q-1$ elements. 
    Since $V$ and $W$ are summands of the free module $M$ over a PID, their intersection $V\cap W$ is also a (free) summand of $M$. Because of this, and the fact that every partial basis of $W$ forms a partial basis together with the elements of $\sigma$, we obtain an inclusion
    \begin{equation}\label{eq:Prop3.50Application}
        \PB(V \cap W; \mathcal B') \hookrightarrow \PB(V;\mathcal B)\cap \Lk_{\PB(M)}(\sigma).
    \end{equation}
    It suffices to show that this inclusion
    is $(|\mathcal B|-4-q)$-connected, since $\PB(V \cap W; \mathcal B')$ is $(|\mathcal B'|-3)$-connected by \cref{Charney}, and $|\mathcal B'| \geq |\mathcal B|-q-1$. To do so, we check the conditions of \cref{bad simplex}. 
    
    First, the inclusion is a full subcomplex:
    Given a set of vertices in $\PB(V\cap W; \mathcal B')$, it is a simplex in $\PB(V;\mathcal B) \cap \Lk_{\PB(M)}(\sigma)$ if and only if it is a partial basis in $V$. Note that mod $2$ the vectors are automatically in $\mathcal B$ and the simplex is automatically in the link of $\sigma$ because the vertices are in $W$. 
    Every partial basis of $V$ consisting of vectors in $V\cap W$ is a partial basis of $V\cap W$. This implies that the given set is a simplex in $\PB(V\cap W; \mathcal B')$ and hence the inclusion is full.
    
    Next, assume $\tau = \{t_0,\dots,t_p\}$ is a
$p$-simplex of $\PB(V;\mathcal B)\cap \Lk_{\PB(M)}(\sigma) $ with no vertices in $\PB(V \cap W;\mathcal B')$.
Using the decomposition $M = U \oplus W$, let $t_i = t^U_i + t^W_i$ with $t^U_i \in U$ and $t^W_i\in W$. Since the vectors of $\sigma$ give a basis of $U$, and $\sigma \cup \tau$ is a partial basis of $M$, $\tau^W = \{t^W_0,\dots,t^W_p\}$ is a partial basis of $W$. We claim that 
$$\Lk_{\PB(V;\mathcal B)\cap \Lk_{\PB(M)}(\sigma) }(\tau) \cap \PB(V \cap W;\mathcal B')=
\Lk_{\PB(W)}(\tau^W) \cap \PB(V \cap W;\mathcal B').$$
Indeed, take a simplex $\rho \subset \PB(V \cap W;\mathcal B')$. Then $\rho$ lies in the left-hand side if and only if $\sigma \cup \tau \cup \rho$ is a disjoint union and it is a partial basis of $M$, and on the right-hand side if and only if $\tau^W \cup \rho$ is a disjoint union and it is a partial basis of $W$.
But it is easy to check that $\sigma \cup \tau \cup \rho$ is a disjoint union and a partial basis in $M$ if and only if 
$\sigma \cup \tau^W \cup \rho$ is a disjoint union and a partial basis in $M$, if and only if $\tau^W \cup \rho$ is a disjoint union and a partial basis of $W$.


Because $\operatorname{rank}(W) = \operatorname{rank}(M) - q-1 < \operatorname{rank}(M)$,
 $\Lk_{\PB(W)}(\tau^W) \cap \PB(V \cap W;\mathcal B')$ is $(|\mathcal B|-p-q-5)$-connected by the inductive hypothesis. By \cref{bad simplex} the inclusion \eqref{eq:Prop3.50Application} is $(|\mathcal B|-4-q)$-connected, as required.

    To show that $\PB(V;\mathcal B)\cap \Lk_{\PB(M)}(\sigma)$ is in fact wCM of dimension $(|\mathcal B|-3-q)$, let $\tau$ be a simplex of it and observe that $\Lk_{\PB(V;\mathcal B)\cap \Lk_{\PB(M)}(\sigma)}(\tau) = \PB(V;\mathcal B)\cap \Lk_{\PB(M)}(\sigma\cup \tau)$, which by the above is $(|\mathcal B|-4-\dim(\sigma\cup\tau))$-connected, i.e.\  $(|\mathcal B|-3-q-\dim(\tau)-2)$-connected, which is precisely what is required to be wCM of dimension $(|\mathcal B|-3-q)$.
\end{proof}

\begin{lem}\label{inequality lemma} Let $k,a\in \mathbb R$ with $a\ge 2$. If a topological space is $ \frac k 2 $-connected, then it is $ \frac {k-a+2} a$-connected. 
\end{lem}

\begin{proof}
	We have $\lfloor \frac {k-a+2} a\rfloor  \leq \max(\lfloor \frac k 2 \rfloor,-2)$.
\end{proof}

\begin{lem}\label{Y(sigma)}
	Let $\sigma$ be a $q$-dimensional simplex of $Y_n$, possibly empty. Then $Y_n(\sigma^\perp)$ is $ \frac{n-4q-16}6$-connected.
\end{lem}

\begin{proof}Let $M= \Z^{\oplus n}$, equipped with the linear form $\phi$ and bilinear form $\langle-,-\rangle$ from \cref{sec:evenandoddsp}. 

The operation $\rho$ of \cref{definition of rho} descends to a well-defined function on $M/2$, which we continue to denote by the same symbol. Let $\mathcal B \subset M/2$ be a maximal set of vectors such that $\{\rho(x)\}_{x \in \mathcal B}$ are pairwise disjoint $2$-element sets, each of which is also disjoint from $\rho(v)$ for every vertex $v$ of $\sigma$.  Notice that $\mathcal B$ is necessarily a partial basis modulo $2$. Since $\sigma$ lies in $Y_n$, the sets $\{\rho(v)\}_{v \in \sigma}$ are also pairwise disjoint $2$-element sets. Hence $\vert\mathcal B\vert=\lfloor\frac{n-2q-2}2\rfloor$.



	Let $V=\ker \phi \cap \sigma^\perp$, where $\sigma^\perp$ denotes the subspace of vectors $x$ satisfying $\langle x, v\rangle=0$ for every vertex $v$ of $\sigma$. In other words, $V$ is the intersection of the kernels of the functionals $\phi$ and $\langle -,v\rangle$ for each vertex $v$ of $\sigma$. This means $V$ is the kernel of the direct sum of these functionals, and we get the sequence 
	\begin{equation}\label{SES}
	    0 \to V \to M \to \Z^{q+2} \to 0.
	\end{equation} This sequence is left exact, by definition, but it is even short exact. We will prove $M \to \Z^{q+2}$ is surjective by arguing that the $q+2$ functionals are a partial basis of $M^\vee$.
    Consider first the case that $n$ is even. Then $\langle-,-\rangle$ is a symplectic form on $M$ (\cref{Qn is symplectic group}). Note  that $\ker \phi \cap \sigma^\perp = (\sigma \cup \{v_n\})^\perp$, since $\phi=\langle v_n,-\rangle$ for $n$ even (\cref{distinguished vector}). Since $\sigma \cup \{v_n\}$ is a partial basis for $M$, the functionals $\phi$ and $\langle -, v\rangle$ for $v$ in $\sigma$ are a partial basis for $M^\vee$. The case that $n$ is odd is similar. In this case the pairing $\langle-,-\rangle$ is a symplectic form on $\ker \phi$ (\cref{Qn is symplectic group}), and $\sigma$ spans a summand of $\ker \phi$, since $\phi(v) =0$ for all $v\in \sigma$. Then the functionals $\langle -, v\rangle$ for $v \in \sigma$ are a partial basis for $(\ker \phi)^\vee$, which is a summand of $M^\vee$. Because $\phi$ is surjective, the functionals $\langle -, v\rangle$ for $v \in \sigma$ together with $\phi$ are a partial basis of $M^\vee$.

Now we will explain why $\mathcal B \subset V/2$. Reducing \eqref{SES} modulo 2 shows that $V/2$ consists of those vectors in $M/2$ which are in the kernel of $\phi/2$ and which are orthogonal to the modulo $2$ reduction of $v$, for all $v \in \sigma$. Here $\phi/2 : M/2 \to \Z/2$ is the reduction of $\phi$ modulo $2$, and orthogonality is considered with respect to the  reduction of the pairing $\langle-,-\rangle$ modulo $2$. To check that these conditions hold for all $x \in \mathcal B$, first observe that as $\rho(x)$ is a 2-element set we have $(\phi/2)(x)=0$, and then observe that as $\rho(v)$ is disjoint from $\rho(x)$, $x$ is orthogonal to the modulo $2$ reduction of $v$.

Let $P_n = \PB(V;\mathcal B) \cap \Lk_{\PB(M)}(\sigma \cup \{v_n\})$, where $\sigma \cup v_n$ is indeed a simplex in $\PB(M)$ because this is a requirement for $\sigma$ to be a simplex in $Y_n$. If $u$ is a vertex of $P_n$, then $u\in V$, so $\phi(u) = 0$ and $u\in \sigma^\perp$. Furthermore, $|\rho(u)| = 2$, and $\sigma \cup \{u,v_n\}$ is a partial basis, so $u$ is a vertex of $Y_n(\sigma^\perp)$. In fact $P_n$ is a full subcomplex of $Y_n(\sigma^\perp)$, because a set $\pi =\{u_0,\dots,u_p\}$ of vertices of $P_n$ spans a simplex precisely when $\pi \cup \sigma \cup \{v_n\}$ is a partial basis and the sets $\rho(u_0),\dots,\rho(u_p)$ are disjoint, and this is the same condition as spanning a simplex of $Y_n(\sigma^\perp)$.\mnote{orw: I rewrote this paragraph to say what I think we want to say.}

By \cref{CharneyCor}, $P_n$ is $(\vert\mathcal B\vert -4-(q+1))$-connected, i.e.~$\frac{n-4q-12}{2}$-connected, and hence $\frac{n-4q-16}6$-connected by \cref{inequality lemma}. Therefore, if we can prove that $P_n \hookrightarrow Y_n(\sigma^\perp)$ is $\frac{n-4q-16}6$-connected using \cref{bad simplex}, we showed that $Y_n(\sigma^\perp)$ is $\frac{n-4q-16}6$-connected.

	
    We now set up to use \cref{bad simplex}. Let $\tau \subset Y_n(\sigma^\perp)$ be a $p$-simplex with no vertex in $P_n$. (We will not use that $\tau$ has no vertices in $P_n$ for what follows.) We need to show that $\mathrm{Lk}_{Y_n(\sigma^\perp)}(\tau) \cap P_n$ is $( \frac{n-4q-16} 6 - p - 1)$-connected. Let $\mathcal C(\tau) = \{x \in \mathcal B : \rho(v) \cap \rho(x) = \varnothing \text{ for all } v \in\tau\}$. Since  $|\rho(v)|=2$ for all $v \in \tau$, and the sets $\rho(x)$ for $x \in \mathcal B$ are pairwise disjoint, we must have $|\mathcal C(\tau)| \geq |\mathcal B| - 2(p+1)$. We have
	$$ \mathrm{Lk}_{Y_n(\sigma^\perp)}(\tau) \cap P_n = \PB(V;\mathcal C(\tau)) \cap \Lk_{\PB(M)}(\sigma \cup \tau \cup \{v_n\}).
 		  $$
	By \cref{CharneyCor}, $\mathrm{Lk}_{Y_n(\sigma^\perp)}(\tau) \cap P_n$ is 
    $(|\mathcal C(\tau)|-4-(p+q+2))$-connected.    Since $$\vert\mathcal C(\tau)\vert \geq \vert\mathcal B \vert-2(p+1) = \big\lfloor \frac{n-2q-4p -6}2\big\rfloor ,$$ it is therefore 
    $\frac{n-4q-6p-18}2$-connected. By \cref{inequality lemma} it is also $\frac{n-4q-6p-22}6 = ( \frac{n-4q-16} 6 - p - 1)$-connected. \end{proof}

	\begin{thm}[Mirzaii--van der Kallen]\label{MvdK}
		Let $N\ge-1$ and let $F$ and $W$ be simplicial complexes. Let $W_? \colon P(F) \to S(W)$ be an order reversing map from the poset of simplices in $F$ to the poset of subcomplexes in $W$; that is, $W_{\sigma'} \subseteq W_\sigma \subseteq W$ if $\sigma$ is a face of $\sigma'$ for simplices $\sigma,\sigma'$ of $F$. Assume the following conditions:
		\begin{enumerate}
			\item \label{union contains skeleton} The union of all $W_\sigma$ for nonempty simplices $\sigma$ of $F$ contains the $(N+1)$-skeleton of $\,W$.
			\item $F$ is $N$-connected.
			\item For a nonempty simplex $\sigma$ in $F$, $W_\sigma$ is $\min(N-1,N-\dim \sigma)$-connected.
			\item For a nonempty simplex $\tau \subset W$, the subcomplex $F_\tau = \{\sigma \in F \, : \, \tau \in W_\sigma\}$ of $F$ 
   is $(N-\dim \tau)$-connected.
			\item For every vertex $v$ in $F$, there exists an $N$-connected $Y_v$ such that $W_v \subseteq Y_v \subseteq W$.
   
		\end{enumerate}
		Then $W$ is $N$-connected.
	\end{thm}

 Indeed, this is \cite[Theorem 4.7]{MvdK}. They formulate the result in terms of posets of sequences (i.e.~directed simplicial complexes) rather than simplicial complexes, but every simplicial complex can be considered a directed simplicial complex after arbitrarily choosing a total order on the set of vertices, without changing the geometric realization. Moreover, rather than (\ref{union contains skeleton}) their hypothesis is that $W=\bigcup_{\sigma \in F} W_\sigma$. But this is no loss of generality, as we may in any case replace $W$ by its $(N+1)$-skeleton.

	\begin{lem}\label{IX}
		$\IX_n$ is $\frac{n-18}8$-connected.
	\end{lem}

	\begin{proof}
        We will prove by descending induction on $p \in \{-1,0,1,2,\ldots\}$ that for each $p$-simplex $\sigma$ of $Y_n$, the complex $\IX_n(\sigma^\perp)$ is $ \frac{n-4p-22}8$-connected. 
        For $p=-1$ this proves the assertion. 	For the base case of the induction we may take $p>\frac{n-14}4$, where the conclusion is vacuously true. 
        For the induction step, we assume that $-1 \leq p\le \frac{n-14}4$, that $\sigma$ is a $p$-simplex of $Y_n$, and that $\IX_n((\sigma')^\perp)$ is $ \frac{n-4p'-22}8$-connected for all $p'$-simplices $\sigma'$ of $Y_n$ whenever $p'>p$. Note that $n-4p-14 \ge 0$, which we will use several times.
        
		Let $F = Y_n(\sigma^\perp)$,  $W=\IX_n(\sigma^\perp)$, and for a simplex $\tau$ in $F$, let $W_\tau = \IX_n((\sigma\cup \tau)^\perp)$. 
        Then of course $W_{\tau'} \subseteq W_\tau$ for $\tau \subseteq \tau'$. We check the other conditions in Theorem \ref{MvdK} with $N= \frac{n-4p-22}8$:
		\begin{enumerate}
        
            \item 
            Let $\pi$ be a simplex in $W$ of dimension $q \le \frac{n-4p-14}4$. It suffices to prove that $\pi$ lies in $W_\tau$ for some $\tau$, since 
            \[\frac{n-4p-14}4\ge \frac{n-4p-14}8 = N+1.\]  
            Note that $\sigma \cup \pi$ is a simplex in $Y_n$ of dimension $p+q+1$, and the given condition on $q$ implies that $\tfrac{n-4(p+q+1)-16}{6} \geq \tfrac{-6}{6}=-1$, so  
            it follows from \cref{Y(sigma)} that $Y_n((\sigma \cup \pi)^\perp)$ is nonempty. If $\tau$ is a nonempty simplex of $Y_n((\sigma \cup \pi)^\perp)$, then $\pi$ is a simplex in $\IX_n((\sigma \cup \tau)^\perp) = W_{\tau}$.
						
   			\item $F$ is $N$-connected by \cref{Y(sigma)}, because
            \[\frac{n-4p-22}8 = \frac{n-4p-14}8 + 1 \le \frac{n-4p-14}6 +1  = \frac{n-4p-20}6 \le \frac{n-4p-16}6.\] 
            
			\item For a nonempty $q$-simplex $\tau$ in $F$, we may suppose by induction that $W_\tau = \IX_n((\sigma \cup \tau)^\perp)$ is $ \frac{n-4p-4q-26}8$-connected. This implies that it is $(N-1)$-connected if $q=0$. It also implies that it is $(N-q)$-connected if $q\ge 1$, because
            \[\frac{n-4p-4q-26}8  = \frac{n-4p-22}8 - \frac12(q-1) \ge N-q. \]
            In particular, $W_\tau$ is $\min(N-1,N-q)$-connected.
			\item For a nonempty $q$-simplex $\pi$ in $W$, we have that $F_\pi = Y_n((\sigma \cup \pi)^\perp)$. This is $(N-q)$-connected, because $Y_n((\sigma\cup \pi)^\perp)$ is $\frac{n - 4p-4q-20}6$-connected by \cref{Y(sigma)} and 
			\[\frac{n - 4p-4q-20}6 = \frac{n-4p-14}6 -\frac23q -1 \ge \frac{n-4p-14}8 -q -1 = \frac{n-4p-22}8-q = N - q.\]
		
			\item 
            Let $v$ be a vertex of $F$. If $\tau$ is a simplex of $W_v=\IX_n((\sigma \cup \{v\})^\perp)$, then $\tau \cup \{v\}$ is a simplex of $\IX_n(\sigma^\perp)$ because $\langle v,w\rangle = 0$ for all vertices $w$ in $\tau$. Hence $\IX_n((\sigma \cup \{v\})^\perp) \subseteq \IX_n((\sigma  \cup \{v\} )^\perp) \ast v \subseteq \IX_n(\sigma^\perp)$, and the cone $\IX_n((\sigma  \cup \{v\})^\perp ) \ast v $ is contractible. \qedhere
		\end{enumerate}
	\end{proof}

Let $W$ be a semisimplicial set, and $S$ a set. Recall from \cref{bracket} the construction $W\langle S\rangle$. For any $s \in S$, there is a natural map $W \to W\langle S\rangle$ given by $W \cong W\langle \{s\} \rangle \hookrightarrow W\langle S\rangle$. Recall also from \cref{maazen} the notion of directed simplicial complex.

	\begin{prop}\label{bracket-prop}
	    Let $W$ be a directed simplicial complex, $S$ a nonempty set. If the left-link of $\sigma$ is $(d-p-1)$-connected for all nonempty $p$-simplices $\sigma$ of $W$, and $W$ is $\min(1,d-1)$-connected, then $W \to W\langle S \rangle$ is $d$-connected for all $s \in S$.
	\end{prop}

\begin{proof} Charney proves this in \cite[Proposition on page 2088]{charneycongruence}, using different notation. In Charney's notation, a directed simplicial complex $W$ is a subset  $F \subset\mathcal O(X)$, where $X$ is the vertex set of $W$, that satisfies the chain condition. Her $F_{\underline v}$ is the left-link of the the simplex $\sigma = (v_0, \dots, v_p) = \underline v$.
She refers in her proposition to the maps of the previous page, which includes the map we are interested in, which is her ii). Her notation $F<S>$ is the same as our $W\langle S\rangle$. 
\end{proof}
 
 \begin{thm}\label{thm:high connectivity of X_n}
		$X_n$ is $\frac{n-19}8$-connected.
	\end{thm}
	
	\begin{proof}
        We proceed by induction on $n$. If $n \leq 10$ then the claim is vacuous, which provides the base cases: we now suppose that $n \geq 11$. We will apply \cref{MvdK}. Let $F = \IX_n$ and $W = X_n$. For a simplex $\tau$ in $F$, let $W_\tau=X_n(\tau^\perp)$. Then of course $W_{\tau'} \subseteq W_\tau$ for $\tau \subseteq \tau'$. We check the other conditions for $N=\frac{n-19}8$:
		\begin{enumerate}
			\item The union of all $W_\tau$ for nonempty simplices $\tau$ in $F$ contains the $(n-4)$-skeleton of $W$, so in particular contains the $(N+1)$-skeleton, because $n-4 \ge N+1$ holds as long as $n\ge 4$. To see this, take a $p$-simplex of $X_n$, $p \leq n-4$. By \cref{transitive rem X}, we can use the action of $Q_n$ on $X_n$ to assume that the simplex is $(e_{n-p},\ldots,e_n)$, which lies in $X_n(\{e_1-e_2\}^\perp) = W_{\{e_1-e_2\}}$.			
            \item $F$ is $\frac{n-18}8$-connected by \cref{IX}, so is in particular $N$-connected.
			\item \label{step3}For a nonempty $q$-simplex $\tau$ in $F$, we have 
			\[W_\tau = X_n(\tau^\perp) \cong X_{n-2q-2}\langle(2\Z)^{q +1}\rangle\]
			by \cref{Xsigma}. 
   From \cref{bracket-prop} with $d=\frac{n-2q-14}{8}$ we see that the pair $(W_\tau,X_{n-2q-2})$ is $d$-connected: indeed, the left-link of a $p$-simplex in $X_{n-2q-2}$ is isomorphic to $X_{n-2q-2-p-1}$ by \cref{leftlink} and so is $ \frac{n-2q-p-22}{8}$-connected by induction, so in particular $(d-p-1)$-connected; moreover, taking $p=-1$, the connectivity of $X_{n-2q-2}$ itself is at least $ \frac{n-2q-21}{8}>d-1$. Because $ \frac{n-2q-21}{8}\le d$, we get that $W_\tau$ is also $\frac{n-2q-21}8$-connected. As
			\[\frac{n-2q -21}8 \ge \frac{n-8q-19}8= N-q\quad\text{if $q \ge1$}\]
			and 
			\[\frac{n-2q-21}8 \ge \frac{n-27}8 = N-1\quad \text{if $q = 0$},\]
            the complex $W_\tau$ is in particular $\min(N-1,N-q)$-connected.
			\item For a nonempty $p$-simplex $\pi$ in $X_n$, the subcomplex $F_\pi$ of $F$ is $\IX_n(\pi^\perp) \cong \IX_{n-p-1}$ by \cref{IXtau}, which is $\frac{n-p-19}8$-connected by \cref{IX}. As
			\[ \frac{n-p-19}8 \ge N- p,\]
            the complex $F_\pi$ is in particular $(N-p)$-connected.
			\item Let $v$ be a vertex of $F$. By \cref{transitive rem}, we can use the action of $Q_n$ on $Z_n$ (which preserves the subcomplexes $\IX_n$ and $X_n$ and is transitive on the set of $p$-simplices of $\IX_n$) to assume $v = e_{n-1}-e_n$. 
            Then by (the proof of) \cref{Xsigma}, a simplex $\{x_0+y_0,\dots,x_q+y_q\}$ of $W = X_n$ with $x_i \in \mathrm{Span}\{e_1,\dots,e_{n-2}\}$ and $y_i \in \mathrm{Span}\{e_{n-1},e_n\}$ is in $W_v = X_n(\{e_{n-1}-e_n\}^\perp)$ if and only if: 
            \begin{enumerate}[(a)]\setcounter{enumii}{3}
	\item \label{ddd} $y_i$ lies in $\mathrm{Span}\{2(e_{n-1}-e_n)\}$ for all $0 \leq i \leq q$.
	\item \label{eee}$\{x_0,\dots,x_q,v_{n-2}\}$ is a partial basis for $\Z^{\oplus (n-2)}$ as an abelian group.
	\item \label{fff}The sets $\rho(x_0), \dots, \rho(x_q)$ are disjoint subsets of $\{1,\dots,n-2\}$.
\end{enumerate}
Consider $X_{n-2}*\{e_n\}$ as a subcomplex of $W=X_n$, as if $\sigma$ is a simplex of $X_{n-2}$, then $\sigma \cup \{e_n\}$ is a simplex of $X_n$. Let $Y_v$ be the union of $X_{n-2}*\{e_n\}$ and $W_v$ in $X_n$. We observe that $W_v \cap (X_{n-2} * \{e_n\}) = X_{n-2}$. Therefore, $Y_v$ is the pushout $W_v \cup_{X_{n-2}} (X_{n-2} *\{e_n\})$.
It remains to show that $Y_v$ is $N$-connected.
%
   In step (\ref{step3}) we have seen that the pair $(W_v,X_{n-2})$ is $\frac{n-14}8$-connected, and in particular $N$-connected. Because $X_{n-2}$ is nonempty for $n\ge 10$ (even $n\ge 2$), $W_v \cup_{X_{n-2}} (X_{n-2} \ast \{e_n\}) \simeq W_v/X_{n-2}$ is $N$-connected. \qedhere
		\end{enumerate}
	\end{proof}

\begin{cor}\label{connectivity corollary}
    The complex $W_\bullet(\mathsf Q;n)$ is $\frac{n-19}{8}$-connected.
\end{cor}

\begin{proof}Combine
     \cref{rem:X_n is destabilisation} and \cref{thm:high connectivity of X_n}.
\end{proof}

\begin{rem}
    Our proof that the complexes $W_\bullet(\mathsf Q;n)$ are highly connected can be adapted to show that $W_\bullet(\mathsf T;n)$ are highly connected, too. In fact the argument would significantly simplify. Without the congruence condition,    the analogue of $\IX_n$ would be the complex of isotropic partial bases $\{u_0,...,u_p\}$ such that $\phi(u_i)=0$ for all $i$, forming a partial basis with $v_n$. When $n$ is odd this is simply the complex of isotropic partial bases in $\ker(\phi)$, and when $n$ is even this is the link of $v_n$ in the complex of isotropic partial bases in $\Z^{\oplus n}$. It follows from \cite[Theorem 5.7]{MvdK} and \cref{retract lemma} that both are  highly connected. From this one can deduce high connectivity of the analogue of $X_n$ as in \cref{thm:high connectivity of X_n}. This would however only show slope $\frac 1 4$ connectivity of $W_\bullet(\mathsf T;n)$ (the slope drops when doing the analogue of step (\ref{step3}) of the proof of \cref{thm:high connectivity of X_n}), which is worse than the slope $\frac 1 3$ connectivity proven by Sierra--Wahl \cite{sierrawahl}.
\end{rem}

	\section{Verifying the axioms}\label{sec:Verifying}
	
	The goal in this section is to deduce Theorems \ref{thmA},  \ref{thmB}, \ref{thmH}, and \ref{thmC} from the generic  \cref{thm:GenThm}, i.e.~verifying that Axioms 
	\ref{axiomIII} and \ref{axiomII} of \cref{thm:GenThm} are satisfied in the respective situations. 
    
    We begin with Axiom \ref{axiomIII}, which requires putting together various proofs of high connectivity of simplicial complexes from the literature (or from \cref{connectivity section}, in the case of \cref{thmC}). We then turn to Axiom \ref{axiomII}, which requires some amount of representation theory, and in particular a detailed study of the behaviour of the various coefficient systems under shifting.

\subsection{The groupoids and their connectivities}
We return to the situation of Sections \ref{section1} and \ref{section2}: we have $\mathsf G = \coprod_n \Gamma_n$ a family of groups arising as symmetries of a sequence of topological objects, and $\mathsf Q = \coprod_n Q_n$ a family of arithmetic groups arising by taking the induced symmetries on homology. 

\subsubsection{Mapping class groups} We take $\Gamma_n$ to be the mapping class group of a genus $n$ surface with a boundary component, and $Q_n = \mathrm{Sp}_{2n}(\Z)$. The groupoid $\mathsf G$ has a monoidal structure given by ``pair of pants''-product. In fact the pair of pants-product makes the classifying space $\coprod_n B\Gamma_n$ into an $E_2$-algebra, as first observed by Miller \cite{miller}, which in particular endows $\mathsf G$ with a braiding. See \cite[Section 5.6]{randalwilliamswahl} for a careful construction. The monoidal structure on $\mathsf Q$ is simply given by block-sum of matrices, which is in an evident way symmetric monoidal. 

Surjectivity of $\Gamma_n \to \mathrm{Sp}_{2n}(\Z)$ is a classical fact in surface topology --- the symplectic group is generated by transvections, and each transvection can be lifted to a Dehn twist in the mapping class group. 

\begin{lem}\label{lem:mcg verifies III} If $\mathsf G$ is the groupoid of mapping class groups, and $\mathsf Q$ the groupoid of integral symplectic groups, then Axiom \ref{axiomIII} of \cref{thm:GenThm} is satisfied with $\nu = \tfrac 1 2$.
\end{lem}
\begin{proof}
The splitting complexes $S_\bullet^{E_1}(\mathsf G;n)$ satisfy the standard connectivity estimate by \cite[Theorem 3.4]{gkrw-secondary}; the splitting complexes $S_\bullet^{E_1}(\mathsf Q;n)$ satisfy the standard connectivity estimate by \cite{looijengavanderkallen}. \cref{prop:VerifyingIII} therefore implies that $H^{\overline{\mathbf{R}}}_{n,d}(\underline{\bk}_\mathsf{Q})=0$ for $d < \tfrac{2}{3} \cdot n$. The kernels $K_n$ are the Torelli groups $\mathcal{I}_{n,1}$, and Johnson \cite{JohnsonAb} has proved that the Johnson homomorphism $H_1(\mathcal{I}_{n,1};\bk) \to \Lambda^3_\bk H_1(\Sigma_{n,1};\bk)$ is an isomorphism for $n \geq 3$. He has also proved \cite[Theorem 1]{JohnsonHomomorphism} that it is an epimorphism for $n \geq 2$. As the image of $\Lambda^3_\bk H_1(\Sigma_{2,1};\bk) \to \Lambda^3_\bk H_1(\Sigma_{n,1};\bk)$ generates the target as an $\mathrm{Sp}_{2n}(\mathbb{Z})$-representation, \cref{lem:LowDegRModHomology} \ref{it:LowDegRModHomology4} implies that $H^{\overline{\mathbf{R}}}_{n,2}(\underline{\bk}_\mathsf{Q})=0$ for $n \geq 3$, and hence by \cref{lem:LowDegRModHomology} \ref{it:LowDegRModHomology5}  that $H^{\overline{\mathbf{R}}}_{n,d}(\underline{\bk}_\mathsf{Q})=0$ for $d < \tfrac{1}{2} \cdot n + 1$.
\end{proof}


\subsubsection{Automorphisms of free groups} 
We take $\Gamma_n = \mathrm{Aut}(F_n)$ the automorphism group of a free group on $n$ generators, and $Q_n = \mathrm{GL}_{n}(\Z)$. The groupoid $\mathsf G$ is symmetric monoidal in an evident way, being equivalent to the maximal subgroupoid of the category of finitely generated free groups, with monoidal structure given by the free product. The monoidal structure on $\mathsf Q$ is again simply given by block-sum of matrices, with the obvious symmetry. 

Surjectivity of $\Gamma_n \to \mathrm{GL}_{n}(\Z)$ is an equally classical fact as in the mapping class group case, proven by lifting elementary matrices to the analogues of Dehn twists inside $\mathrm{Aut}(F_n)$.

\begin{lem}\label{lem:aut verifies III} If $\mathsf G$ is the groupoid of automorphisms of free groups, and $\mathsf Q$ the groupoid of integral general linear groups, then Axiom \ref{axiomIII} of \cref{thm:GenThm} is satisfied with $\nu = \tfrac 1 2$. 
\end{lem}
\begin{proof}
The splitting complexes $S_\bullet^{E_1}(\mathsf G;n)$ satisfy the standard connectivity estimate by \cite[Theorem 4.4]{hepworth}; the splitting complexes $S_\bullet^{E_1}(\mathsf Q;n)$ satisfy the standard connectivity estimate by \cite{charneyGLn}. The kernels $K_n$ are the Torelli groups $\mathrm{IA}_{n}$, and also in this case there are Johnson homomorphisms $H_1(\mathrm{IA}_{n};\bk) \to \mathrm{Hom}_\bk(H_1(F_n;\bk), \Lambda^2_\bk H_1(F_n;\bk))$. These are surjective for all $n$, and  isomorphisms for $n \geq 3$ \cite[Theorem 6.1]{KawazumiMagnus}. As in the mapping class group case, \cref{lem:LowDegRModHomology} \ref{it:LowDegRModHomology4} implies that $H^{\overline{\mathbf{R}}}_{n,2}(\underline{\bk}_\mathsf{Q})=0$ for $n \geq 3$, and hence by \cref{lem:LowDegRModHomology} \ref{it:LowDegRModHomology5}  that $H^{\overline{\mathbf{R}}}_{n,d}(\underline{\bk}_\mathsf{Q})=0$ for $d < \tfrac{1}{2} \cdot n + 1$.
\end{proof}


\subsubsection{Handlebody groups} We take $\Gamma_n $ to be the mapping class group of a genus $n$ handlebody with  a marked disk on its boundary, and $Q_n=\mathrm{GL}_n(\Z)$. The braided monoidal structure on $\mathsf G$ is described in \cite[Section 5.7]{randalwilliamswahl}. It is classical that $\Gamma_n \to \mathrm{Aut}(F_n)$ is surjective (see e.g.\  \cite[Theorem 6.3]{hensel} in the case that there is no marked disk on the boundary), and then so is the composite $\Gamma_n\to Q_n$.

\begin{lem} \label{lem:handlebody verifies III} If $\mathsf G$ is the groupoid of handlebody mapping class groups, and $\mathsf Q$ the groupoid of integral general linear groups, then Axiom \ref{axiomIII} of \cref{thm:GenThm} is satisfied with $\nu = \tfrac 1 3$.
\end{lem}
\begin{proof}
The destabilisation complexes $W_\bullet(\mathsf G;n)$ are $\tfrac{n-3}{2}$-connected \cite{hatcherwahl}, so their reduced homology vanishes in degrees $d < \tfrac{1}{2} \cdot n - 1$. The splitting complexes $S_\bullet^{E_1}(\mathsf Q;n)$ satisfy the standard connectivity estimate, as noted above, so \cite[Theorem 7.1]{RWclassical} implies that the reduced homology of the destabilisation complexes $W_\bullet(\mathsf Q;n)$ vanishes in degrees $d < n-1$ for $n>1$ (and trivially also for $n \leq 1$): in particular it vanishes in degrees $d < \tfrac{1}{2}\cdot n - \tfrac{1}{2}$ and $n \geq 1$. Applying \cref{prop:verifyingIIIviadestab} with $\xi'=0$ and $\xi''=\tfrac{1}{2}$ shows that $H^{\overline{\mathbf{R}}}_{n,d}(\underline{\bk}_\mathsf{Q})=0$ for $d < \tfrac{1}{2} \cdot n + \tfrac{1}{2}$ and $n \geq 1$. By \cref{lem:LowDegRModHomology} \ref{it:LowDegRModHomology3} it also vanishes for $d < \tfrac{1}{3} \cdot n + 1$.
\end{proof}


\subsubsection{The integral Burau representation}\label{sec:Burau}


 In proving \cref{thmC}, we will take $\mathsf G$ to be the free braided monoidal groupoid on one generator, so that $\mathsf G = \coprod_n \beta_n$, and we let $\mathsf Q$ be defined as in \cref{connectivity section}. There is a unique braided monoidal functor $\mathsf G \to \mathsf Q$ sending the generator to the object $\mathbb{Z} \in \mathsf{Q}$: by our choice of braiding on $\mathsf{Q}$ this realises the reflected integral Burau representation.

\begin{lem}\label{lem:burau verifies III} If $\mathsf G$ is the groupoid of braid groups, and $\mathsf Q$ the image of the integral Burau representation in the even and odd symplectic groups, then Axiom \ref{axiomIII} of \cref{thm:GenThm} is satisfied with $\nu = \frac{1}{17}$. 
\end{lem}
\begin{proof}
The destabilisation complexes $W_\bullet(\mathsf{G}; n)$ associated to the braid groups are in fact \emph{contractible}, by a theorem of Damiolini \cite[Theorem 2.48]{damiolini} (see \cite[Proposition 3.2]{hatchervogtmann3} for a published reference). The destabilisation complexes $W_\bullet(\mathsf Q;n)$ are $\frac{n-19}{8}$-connected by \cref{connectivity corollary}, so their reduced homology vanishes in degrees $d < \tfrac{n-18}{8}$. Applying \cref{prop:verifyingIIIviadestab} with $\xi'=+\infty$ and $\xi''=-\tfrac{10}{8}$ shows that $H^{\overline{\mathbf{R}}}_{n,d}(\underline{\bk}_\mathsf{Q})=0$ for $d < \tfrac{1}{8} \cdot n -\tfrac{10}{8}$ and $n \geq 1$. The kernels $K_n$ are the Torelli braid groups, as studied in \cite{BrendleMargalitPutman}. By Theorem C of that paper (and the discussion after it) it follows that the image of $H_1(K_5;\bk) \to H_1(K_n;\bk)$ generates the target as a $Q_n$-module for any $n > 5$, so \cref{lem:LowDegRModHomology} \ref{it:LowDegRModHomology4} implies that $H^{\overline{\mathbf{R}}}_{n,2}(\underline{\bk}_\mathsf{Q})=0$ for $n > 5$, and hence by \cref{lem:LowDegRModHomology} \ref{it:LowDegRModHomology5}  that $H^{\overline{\mathbf{R}}}_{n,d}(\underline{\bk}_\mathsf{Q})=0$ for $d < \tfrac{1}{17} \cdot n + 1$ and $n \geq 1$.
\end{proof}


\begin{rem}\label{degenerate remark}
Appealing to \cref{prop:verifyingIIIviadestab} in the proof  of \cref{lem:burau verifies III} is needlessly convoluted. When $\mathsf G$ is the braid groupoid, one can more directly identify $Q_\bL^{\overline{\mathbf{R}}}(\underline{\bk}_\mathsf{Q})$ with the suspension of $W_\bullet(\mathsf Q;n)$, making \cref{prop:verifyingIIIviadestab} rather degenerate in this case; moreover, high connectivity of the destabilisation complexes of the braid groupoid is in any case implicitly used in the proof of \cref{prop:verifyingIIIviadestab}. We have opted to present the argument in this way in order to make the methods of proof of Lemmas \ref{lem:mcg verifies III}, \ref{lem:aut verifies III}, \ref{lem:handlebody verifies III}, and \ref{lem:burau verifies III} as uniform as possible.
\end{rem}
    
	\subsection{Stable real cohomology of arithmetic groups}
	
	Let $G$ be a semisimple algebraic group over $\Q$, $\Gamma \subset G$ an arithmetic subgroup, and $V$ an irreducible real representation of $G$. Borel \cite{borelstablereal1,borelstablereal} has famously calculated the cohomology $H^*(\Gamma;V)$ in a stable range. For the proof, Borel studies the homomorphism 
	\begin{equation}\label{matsushima}
		(\mathcal H^k \otimes V)^G \to H^k(\Gamma;V),
	\end{equation} 
	where the left-hand side of \eqref{matsushima} denotes the space of $G$-equivariant $V$-valued harmonic $k$-forms on the symmetric space of $G$. The left-hand side of  \eqref{matsushima} is very computable: for nontrivial $V$, it vanishes, and for constant coefficients it coincides with the cohomology of the compact dual of $G$. The goal, then, is to show that \eqref{matsushima} is an isomorphism in a range of degrees.
	
	Many authors have subsequently revisited and/or improved on Borel's work. We will make no attempt at a complete historical survey, but mention in particular \cite{vz,franke,grobner}. The recent papers \cite{tshishiku,lisun,badersauer} all show, under different hypotheses, that \eqref{matsushima} is an isomorphism for all $k < \mathrm{rank}_\R(G)$. This bound is in particular independent of $V$, unlike the one originally obtained by Borel. Intriguingly, \cite{badersauer} establish such a result by means of geometric group theory, rather than by automorphic methods, and $V$ is allowed to be any unitary representation.

	We will be interested in the groups $G=\mathrm{Sp}_{2g}$ and $G=\mathrm{SL}_n$, whose real ranks are $g$ and $n-1$, respectively. In these cases Borel's results (with subsequent improvements) read as follows:
	
	\begin{thm}\label{thm:borel-sp}
		Let $G=\mathrm{Sp}_{2g}$, $\Gamma \subset G$ be an arithmetic subgroup, and $V$ be an irreducible real representation of $G$. Then for $k < g$ one has
		$H^k(\Gamma;V)=0$, if\, $V$ is nontrivial; the cohomology with trivial coefficients agrees below degree $g$ with a polynomial algebra with generators in degrees $2,6,10,\ldots$  
	\end{thm}

	\begin{thm}\label{thm:borel-sl}
		Let $G=\mathrm{SL}_{n}$, $\Gamma \subset G$ be an arithmetic subgroup, and $V$ be an irreducible real representation of $G$. Then for $k < n-1$ one has
		$H^k(\Gamma;V)=0$, if\, $V$ is nontrivial; the cohomology with trivial coefficients agrees below degree $n-1$ with an exterior algebra with generators in degrees $5,9,13,\ldots$  
	\end{thm}Both \cref{thm:borel-sp} and \cref{thm:borel-sl} are special cases of \cite[Theorem C]{badersauer}.
	
	\begin{rem}\label{rem:matsushima}
		The homomorphism \eqref{matsushima} is known to be \emph{injective} in a significantly larger range of degrees. In particular, for $G=\mathrm{Sp}_{2g}$ and $G=\mathrm{SL}_n$, one can show that the homomorphism \eqref{matsushima} is injective for all $k < 2\cdot \mathrm{rank}_\R(G)$. See \cite[Section 4]{gkt}. 
	\end{rem}
	
	In our applications to automorphisms of free groups and mapping class groups of handlebodies we will be more concerned with the groups $\mathrm{GL}_n(\Z)$, which are not directly covered by Borel's theorem since $\mathrm{GL}_n$ is not semisimple. We have instead the following.

	\begin{prop}\label{thm:borel-gl}
		Let  $V$ be an irreducible real representation of\,  $\mathrm{GL}_n$. Let $\Gamma \subset \GL_n$ be an arithmetic subgroup containing the diagonal matrices with entry $\pm 1$. Then $H^k(\Gamma;V)=0$ for $k<n-1$, unless $V$ is an even tensor power of the determinant representation. If\, $V$ is an even tensor power of the determinant representation then $H^\ast(\Gamma;V)$ agrees below degree $n-1$ with an exterior algebra with generators in degrees $5,9,13,\ldots$ 
	\end{prop}
	
	\begin{proof}Let $\overline \Gamma = \Gamma \cap \mathrm{SL}_n$.
		By Shapiro's lemma, one has
		$$ H^k(\overline \Gamma;V) \cong H^k(\Gamma;V) \oplus H^k(\Gamma,V \otimes \det). $$
		If $V$ is not a tensor power of the determinant representation, then $V$ is nontrivial as a representation of $\mathrm{SL}_n$, and the left-hand side vanishes by \cref{thm:borel-sl}, and hence so does $H^k(\Gamma;V)$. To finish the proof, we need to argue that if $V$ is an odd tensor power of the determinant representation, then $H^k(\Gamma;V)=0$ for $k<n-1$. If $n$ is odd, then $H^k(\Gamma;V)$ vanishes for all $k$ because  ``center kills''. If $n$ is even, the result follows from the commutative diagram 
		\[ \begin{tikzcd}
			H^k(\Gamma \cap \mathrm{SL}_{n-1};V)  \arrow[r,hookleftarrow]&
			H^k(\Gamma \cap \GL_{n-1};V)=0 \\
			H^k(\overline \Gamma;V) \arrow[u] \arrow[r,hookleftarrow]& 
			H^k(\Gamma;V) \arrow[u]
		\end{tikzcd}\]since the left vertical arrow is injective in the stable range $k<n-1$. 
	\end{proof}

\subsection{The groups and the coefficients}\label{sec:the groups}

We will be interested in sequences of representations of three families of algebraic groups:
\begin{gather*}
	\ldots  \subset\mathrm{Sp}_{2n} \subset \mathrm{Sp}_{2n+2} \subset \mathrm{Sp}_{2n+4} \subset \ldots \\
	\ldots \subset \mathrm{GL}_n \subset \mathrm{GL}_{n+1} \subset \mathrm{GL}_{n+2} \subset \ldots \\
	\ldots \subset \mathrm{Sp}_{n-1} \subset \mathrm{Sp}_{n} \subset \mathrm{Sp}_{n+1} \subset  \ldots
\end{gather*}
	
	In all three cases, there is a monoidal structure present: for any ring $R$, the families
	$$ \coprod_{n \geq 0} \mathrm{Sp}_{2n}(R), \qquad \coprod_{n \geq 0} \mathrm{GL}_{n}(R), \qquad \text{and}\qquad \coprod_{n \geq 0} \mathrm{Sp}_{n-1}(R)$$
	form braided monoidal groupoids satisfying the conditions \ref{it:ass:2} and \ref{it:ass:3} of \S\ref{subsec:formulation}: their monoids of objects are the natural numbers, and writing $Q_n$ for $\mathrm{Aut}(n)$ one has that $Q_0$ is trivial, and $Q_a \times Q_b \to Q_{a+b}$ is injective. In the first two cases, the monoidal structure is the obvious one, given by block-sum of matrices, and the braiding is given by the evident symmetry. The monoidal structure in the third case was described in detail in \cref{connectivity section}.
	
	\subsubsection{The coefficient systems}\label{subsec:coefficient systems}
	
	Let us now define the coefficient systems of interest to us for the above three families of groups. All of our examples will come from ``highest weight theory'' --- algebraic representations of the groups in question can be naturally indexed by partitions; fixing a partition defines a system of representations, and the family of coefficient systems thus defined will satisfy homological stability with a uniform stable range, say when $R=\Z$. 
	
	That said, we will not actually appeal to highest weight theory in defining the relevant coefficient systems --- indeed, the irreducible representation of a classical group associated with a dominant weight is unique only up to noncanonical isomorphism, which is hardly enough data to define a meaningful coefficient system. We will therefore instead follow Weyl's approach \cite{Weyl} to constructing the irreducible representations of the classical groups, by realizing them inside tensor powers of the defining representation via the invariant theory of the symmetric groups. To avoid repetition, representations are always considered over a characteristic zero base field $\bk$. 
	
	By a \emph{partition} we mean a descending sequence of nonnegative integers $\lambda = (\lambda_1 \geq \lambda_2 \geq \ldots )$ which eventually reaches zero. We set $\vert\lambda\vert = \sum_i \lambda_i$, and $l(\lambda)=\max \{i:\lambda_i \neq 0\}$. Partitions with $\vert\lambda\vert=r$ are in bijection with irreducible representations of the symmetric group $\mathfrak S_r$, and we denote by $\sigma_\lambda$ the representation associated with $\lambda$, a Specht module. We define the \emph{Schur functor} $S^\lambda$ associated to $\lambda$ by 
	$$ S^\lambda(V) := V^{\otimes r} \otimes_{\mathfrak S_r} \sigma_\lambda, $$
	where $r=\vert\lambda\vert$. It is classical that $S^\lambda(V)$ is nonzero if and only if $\dim(V)\geq l(\lambda)$, and that $S^\lambda(V)$ is always an irreducible representation of $\mathrm{GL}(V)$. The representations of the form $S^\lambda(V)$ are precisely the \emph{polynomial} representations of $\GL(V)$.
	
	
	
	\subsubsection{Symplectic groups.} \label{subsec:symplectic} 
	Let $V(n)=\bk^{n}$ be the defining representation of $\mathrm{Sp}_{n}$, where $n$ can be even or odd (the defining representations of odd symplectic groups were defined in \cref{def:defining}). Then $S^\lambda(V(n))$ is generally not irreducible as a representation of $\mathrm{Sp}_{n}$. Let $V(n)^{\langle r\rangle} \subset V(n)^{\otimes r}$ be the subspace of \emph{traceless tensors}, i.e.\ the joint kernel of the $\binom r 2$ evident contraction maps $V(n)^{\otimes r} \to V(n)^{\otimes (r-2)}$. Define, for $r= \vert\lambda\vert$,
	\begin{equation}\label{definition of Vlambda}
		V_\lambda(n) := S^\lambda(V(n)) \cap V(n)^{\langle r \rangle},
	\end{equation}
	the intersection taking place inside $V(n)^{\otimes r}$.

	If $n=2g$ is even, then this definition is due to Weyl. He showed that $V_\lambda(n)$ is nonzero if and only if $g \geq l(\lambda)$, that $V_\lambda$ is always an irreducible representation of $\mathrm{Sp}_{2g}$, and that all irreducible representations of $\mathrm{Sp}_{2g}$ arise in this way. 
	
	If $n=2g-1$ is odd, then this definition of $V_\lambda(n)$ was proposed by Proctor \cite{proctor}. He showed that the groups $V_\lambda(n)$ are nonzero if and only if $g\geq l(\lambda)$, and that they ``naturally'' interpolate between the irreducible representations of the usual symplectic groups; one precise statement is in particular that there exist character formulas for the representations $V_\lambda(n)$ nearly identical to the classical Weyl character formula. Another natural such family of representations interpolating between the irreducible representations of the usual symplectic groups was independently proposed by Shtepin  \cite[Definition 4.2]{shtepin}. It is not immediately clear that the two proposed definitions agree --- and we will not need to know this --- although the representations have the same character. The odd symplectic groups are not reductive, and the representations $V_\lambda(n)$ are not irreducible (but they are indecomposable).
	
	Let $M_\lambda(n) := V_\lambda(2n)$, and define $E_\lambda(n) := V_\lambda(n-1)$ if $n>1$, and $E_\lambda(0)=0$. Then $M_\lambda$ is a coefficient system for the groupoid $\coprod_{n \geq 0} \Sp_{2n}(\Z)$, and $E_\lambda$ is a coefficient system for the groupoid $\mathsf T = \coprod_{n \geq 0} \Sp_{n-1}(\Z)$. The construction is the same in both cases, so we will be content to explain it for $M_\lambda$. The definition \eqref{definition of Vlambda} makes sense more generally if $V(n)$ is any vector space with a bilinear form, and even more generally if $V(n) $ is replaced with a \emph{coefficient system} $M$ equipped with a map $M \boxtimes M \to \underline \bk$. We apply this to  $M =  \{V(2n)\}$. It forms a coefficient system with structure maps $\phi_{a,b} : \bk \otimes M(b) \to M(a + b)$ given by the inclusion of the last $2b$ basis vectors. The map $M \boxtimes M \to \underline \bk$ is given by the symplectic form; it is a map of coefficient systems  since the $\phi_{a,b}$ are isometric embeddings. Hence we also get a coefficient system $M_\lambda$ with $M_\lambda(n)= S^\lambda(M(n)) \cap M(n)^{\langle r \rangle}$. In fact it is a polynomial coefficient system for the symplectic groups, of degree $\vert\lambda\vert$, consisting by construction of a sequence of irreducible representations of the symplectic groups. 
	
	\subsubsection{General linear groups.} The coefficient systems we will consider for the general linear groups, in connection with automorphisms of free groups and handlebody groups, are of a very similar nature. 
	Let $W(n)=\bk^n$ be the defining representation of $\mathrm{GL}_n$. Define the subspace of \emph{traceless tensors}
	$$ W(n)^{\langle r,s \rangle} \subset W(n)^{\otimes r} \otimes (W(n)^\ast)^{\otimes s}$$
	to be the joint kernel of the $rs$ evident contraction maps $$W(n)^{\otimes r} \otimes (W(n)^\ast)^{\otimes s} \to W(n)^{\otimes r-1} \otimes (W(n)^\ast)^{\otimes s-1}.$$ For a pair of partitions $\lambda$ and $\mu$, with $\vert\lambda\vert=r$ and $\vert\mu\vert = s$, define 
	$$ W_{\lambda,\mu}(n) = \big(S^\lambda(W(n)) \otimes S^\mu(W(n)^\ast)\big) \cap \big(W(n)^{\langle r, s\rangle}\big),$$
	the intersection taking place inside $W(n)^{\otimes r} \otimes (V(n)^\ast)^{\otimes s}$. Then $W_{\lambda,\mu}(n)$ is nonzero if and only if $l(\lambda)+l(\mu) \leq n$, and $W_{\lambda,\mu}(n)$ is an irreducible representation of $\mathrm{GL}_n$. All irreducible representations of the general linear groups arise in this way.  
	
	The $\{W(n)\}$ form a coefficient system, with structure maps $\phi_{a,b} : \bk \otimes W(b) \to W(a + b)$ given by the inclusion of the last $b$ basis vectors; similarly, the $\{W(n)^*\}$ form a coefficient system, with structure maps $\phi_{a,b} : \bk \otimes W(b)^* \to W(a + b)^*$ given by the dual of the projection to the last $b$ basis vectors. These endow the $\{W(n)^{\otimes r} \otimes (W(n)^\ast)^{\otimes s}\}$ with the structure of a coefficient system, and just as in the symplectic groups case the $\{W_{\lambda, \mu}(n)\}$ form a sub-coefficient system.

	\subsection{Littlewood complexes}
	
	It will be important for us to understand how the coefficient systems $M_\lambda$, $E_\lambda$, and $W_{\lambda\mu}$ described in the previous subsection behave with respect to the shift operator $\Sigma$ on coefficient systems. Since 
	$$ \Sigma M_\lambda(n) \cong \operatorname{Res}^{\Sp_{2n+2}}_{\Sp_{2n}}M_\lambda(n+1),$$
	determining how $M_\lambda$ behaves under shifts entails in particular  understanding a \emph{branching rule} for how irreducible representations decompose under restriction along $\Sp_{2n} \hookrightarrow \Sp_{2n+2}$; similarly for $E_\lambda$ and $W_{\lambda\mu}$. Such branching rules are already known (\cref{rem: history of branching}).	However, it will be useful for us to understand each $\Sigma M_\lambda$ as a coefficient system (and similarly for $E_\lambda$ and $W_{\lambda\mu}$), which requires somewhat more refined information than just understanding each representation $ \Sigma M_\lambda(n)$ separately. To carry this out we will use the \emph{Littlewood complexes} introduced by Sam--Snowden--Weyman \cite{littlewoodcomplex}. Let us first recall their construction. We will freely use the language of \emph{analytic functors}, as described carefully in \cite[Section 2]{BDPW}. There are two parallel constructions, one for symplectic groups and one for general linear groups.
	
\subsubsection{Littlewood complexes for symplectic groups}
Let $V$ be a vector space with a skew-symmetric bilinear form $\omega :\wedge^2 V \to \bk$. Consider the analytic functor 
\[ A \mapsto \Phi_V(A) := \mathrm{Sym}( A \otimes V \oplus (\wedge^2 A)[-1])\]
from graded vector spaces to graded vector spaces. 
We may in fact think of it as a functor from graded vector spaces to graded coalgebras, thinking of $\mathrm{Sym}(-)$ as the cofree conilpotent cocommutative coalgebra functor. To define a coderivation on $\Phi_V(A)$, it suffices to describe its projection onto the cogenerators. We define a map
\[ \partial : \mathrm{Sym}(A \otimes V \oplus \wedge^2 A[-1]) \to (A \otimes V)[1] \oplus \wedge^2 A \]
as the composite
\[ \mathrm{Sym}(A \otimes V \oplus \wedge^2 A[-1]) \stackrel{\text{proj.}}{\longrightarrow} \mathrm{Sym}^2(A \otimes V) \stackrel{\text{proj.}}{\longrightarrow} \wedge^2 A \otimes \wedge^2 V \stackrel{ \mathrm{id}\otimes \omega}{\longrightarrow} \wedge^2 A \hookrightarrow (A \otimes V)[1] \oplus \wedge^2 A.  \]
Since $\partial$ has degree $1$, it extends to a degree $1$ coderivation of $\Phi_V(A)$. We may think of $\Phi_V$ as an analytic functor from graded vector spaces to cocommutative dg coalgebras, with the differential on $\Phi_V(A)$ given by $\partial$.

An equivalent description is that  
\[ A \mapsto (A\otimes V)[-1] \oplus (\wedge^2 A)[-2]\]
may be seen as a graded Lie algebra in the category of analytic functors, with two-step nilpotent bracket given by the map
\[ \wedge^2 ((A \otimes V)[-1]) = (\mathrm{Sym}^2 (A \otimes V))[-2] \to (\wedge^2 A \otimes \wedge^2 V)[-2] \to (\wedge^2 A)[-2].\]
Then $\Phi_V(A)$ is the Chevalley--Eilenberg complex of this Lie algebra. This perspective is taken in \cite{sam-heisenberg}.

Either way, the Taylor coefficients of $\Phi_V$ form a symmetric sequence in cochain complexes. We denote by $L^\ast_{(r)}(V)$ the $r$th Taylor coefficient. It is a cochain complex of $\mathfrak S_r$-modules. If $|\lambda|=r$, then we denote the $\lambda$-isotypical component of   $L^\ast_{(r)}(V)$ by $L_\lambda^\ast(V)$, and call this the \emph{dual Littlewood complex} associated with $\lambda$ and $V$. We call it the dual Littlewood complex, since it is the linear dual of the complex defined by Sam--Snowden--Weyman \cite{littlewoodcomplex}. Since $\Sp(V)$ acts naturally on $\Phi_V$, $L^\ast_\lambda(V)$ is a complex of $\Sp(V)$-modules.

It is easy to describe the zeroth cohomology of $L_\lambda^\ast(V)$. We have $L^0_{(r)}(V)=V^{\otimes r}$, and $L^1_{(r)}(V) = \operatorname{Ind}_{\mathfrak S_2 \times \mathfrak S_{r-2}}^{\mathfrak S_r} \mathrm{sgn}_2 \boxtimes V^{\otimes (r-2)}$. The differential $\partial : L^0_{(r)}(V) \to L^1_{(r)}(V)$ is the sum of the $\binom r 2$ ways of contracting using the bilinear form $\omega$. Hence $H^0(L^\ast_{(r)}(V)) \cong V^{\langle r\rangle}$, the space of traceless tensors. In particular, if $V$ is the defining representation of $\Sp_n$, where $n$ may be even or odd, then $H^0(L^\ast_\lambda(V))$ is the representation of $\Sp_n$ associated with the partition $\lambda$ as we have defined it. 

If $V \cong \bk^{2n}$ is the defining representation of $\Sp_{2n}$, then the cohomology of $L^\ast_\lambda(V)$ in all degrees was calculated by Sam--Snowden--Weyman. The result in general, reminiscent of the Borel--Weil--Bott theorem, is that the complex $L_\lambda^\ast(V)$ has at most one nonzero cohomology group, that this cohomology group is always an irreducible representation of $\Sp_{2n}$, and they give explicit combinatorial rules to determine the precise representation and in what degree it lives. We will not have any need for this, but will be content with the following:

\begin{thm}[Sam--Snowden--Weyman]\label{samsnowdenweyman}
	Let $V \cong \bk^{2n}$ be the defining representation of $\Sp_{2n}$. Let $\lambda$ be a partition with $n \geq l(\lambda)-1$. Then $H^i(L^\ast_\lambda(V))=0$ for $i>0$, and $H^0(L^\ast_\lambda(V)) \cong V_\lambda(n)$.   
\end{thm}

\begin{rem}Note in particular that \cref{samsnowdenweyman} says that the dual Littlewood complex is acyclic if $n=l(\lambda)-1$. This is the only case of \cref{samsnowdenweyman} which requires an application of the ``modification rule'' \cite[Section 3.4]{littlewoodcomplex}; the other cases are covered by \cite[Proposition 3.2]{littlewoodcomplex}.
\end{rem}

\subsubsection{Shifts of symplectic Littlewood complexes}

In our definition of $L^\ast_\lambda(V)$ we assumed  $V$ to be a vector space with a skew-symmetric form $\omega : \wedge^2 V \to \bk$. However, the same definitions carry through without change if $V$ instead is a \emph{coefficient system} equipped with such a skew-symmetric form, now taking values in the trivial coefficient system $\underline \bk$. 

We apply this in particular to $M$, the coefficient system given by the defining representations of the groups $\Sp_{2n}(\Z)$, and $E$, the defining representations of the groups $\Sp_{n-1}(\Z)$. Hence we get a coefficient system $L^\ast_\lambda(M)$ for the groupoid $\coprod_{n \geq 0} \mathrm{Sp}_{2n}(\Z)$, satisfying
$L^\ast_\lambda(M)(n) = L^\ast_\lambda(M(n))$, and similarly $L^\ast_\lambda(E)(n) = L^\ast_\lambda(E(n))$. Then $H^0(L^\ast_\lambda(M))$ agrees with the coefficient system $M_\lambda$, and $H^0(L^\ast_\lambda(E))$ is the coefficient system $E_\lambda$.

We now investigate how these coefficient systems behave under shifting. In all cases, the result will be that the shift carries a filtration for which we can explicitly describe its associated graded. Filtrations arise naturally in this context, even in the case of the groupoid $\coprod_{n \geq 0} \mathrm{Sp}_{2n}(\Z)$, where all representations involved are (individually) semisimple: consider the fact that the surjection $\wedge^2 V \to \underline \bk$ splits pointwise, but not as a coefficient system. We refer to \cite[Section 2]{BDPW} for the symmetric function formalism used in the following proof.

\begin{notation}
	Let $\lambda$ and $\mu$ be partitions. We write $\lambda \searrow_k \mu$ if $\mu$ can be obtained from $\lambda$ by removing $k$ boxes from the Young diagram of $\lambda$, no two in the same column. Equivalently, $\lambda \searrow_k \mu$ if and only if $\lambda_1 \geq \mu_1 \geq \lambda_2 \geq \mu_2 \geq \dots $ and $|\lambda| = |\mu|+k$. We write $\lambda \searrow \mu$ if $\lambda \searrow_k \mu$ for some $k \geq 0$. 
\end{notation}

\begin{thm}\label{thm: odd symplectic branching}
	The coefficient system $\Sigma L_\lambda^\ast(E)$ carries an increasing filtration
	\[ 0 = F_{-1} \Sigma L_\lambda^\ast(E) \subset F_0 \Sigma L_\lambda^\ast(E) \subset \dots \subset F_{|\lambda|}\Sigma L_\lambda^\ast(E) = \Sigma L_\lambda^\ast(E)\]
	such that $\operatorname{Gr}^F_k \Sigma L_\lambda^\ast(E) \cong \bigoplus_{\lambda \searrow_k \mu} L^\ast_\mu(E)$.
\end{thm}

\begin{proof}
Let $V$ be a vector space with a skew-symmetric form $\omega$. Suppose that $V$ is equipped with an increasing filtration $F$ such that $F_{-1}V=0$.\footnote{This condition ensures that $\omega$ is filtration-preserving, when $\bk$ is given the trivial filtration.} This induces a filtration of $\Phi_V(A)$, and then also a filtration of $L^\lambda_\ast(V)$. In the case of a two-step filtration --- $W=F_0V$ and $V=F_1V$ --- we have the formula 
\begin{equation}\label{branching littlewood}\operatorname{Gr}^F_k \Phi_V(A) \cong \Phi_W(A) \otimes \mathrm{Sym}^k(V/W \otimes A).\end{equation}
In particular, if $V/W \cong \bk$ is one-dimensional, then \eqref{branching littlewood} reads 
\begin{equation}\label{branching littlewood 2}
	\operatorname{Gr}^F_k \Phi_V(A) \cong \Phi_W(A) \otimes \mathrm{Sym}^k(A).\end{equation}
We now specialize to $V=\Sigma E$, which has a two-step filtration induced from 
\[ 0 \to E \to \Sigma E \to \underline \bk \to 0. \] 
We find that $\operatorname{Gr}^F_k \Phi_{\Sigma E}(A) \cong \Phi_E(A) \otimes \mathrm{Sym}^k(A)$, which upon equating Taylor coefficients of analytic functors expresses $\operatorname{Gr}^F_k L^\ast_\lambda(\Sigma E) = \operatorname{Gr}^F_k \Sigma L^\ast_\lambda(E)$ in terms of the coefficients $L^\ast_\mu(E)$. The result is in fact exactly the statement of the theorem. We will be content with explaining this on the level of Taylor series, i.e.\ after taking classes in an appropriate Grothendieck group (say of graded $\Sp(E)$-modules). We have
\begin{gather*}
	[\operatorname{Gr}^F_k L^\ast_\lambda(\Sigma E)] = \langle s_\lambda, [\operatorname{Gr}^F_k \Phi_{\Sigma E}] \rangle = \langle s_\lambda, h_k \cdot [\Phi_E] \rangle = \langle h_k^\perp s_\lambda, [\Phi_E] \rangle
\end{gather*}
where $\langle-,-\rangle$ is the standard inner product of symmetric functions, and $h_k^\perp$ is the adjoint of multiplication with $h_k$. The skew Schur function $h_k^\perp s_\lambda = s_{\lambda/(k)}$ is calculated by Pieri's rule: $h_k^\perp s_\lambda = \sum_{\lambda \searrow_k \mu} s_\mu$. \end{proof}

\begin{lem}\label{length lemma}
	If $\lambda \searrow \mu$, then $l(\lambda)-1 \leq l(\mu) \leq l(\lambda)$.
\end{lem}

\begin{proof}
	Clear.
\end{proof}

Now \cref{thm: odd symplectic branching} allows us to ``cheaply'' extend \cref{samsnowdenweyman} to the odd symplectic groups.

\begin{cor}\label{odd samsnowdenweyman}
	Let $V = V(2n-1) \cong \bk^{2n-1}$ be the defining representation of $\Sp_{2n-1}$, with $n>0$. Let $\lambda$ be a partition with $n  \geq l(\lambda)$. Then $H^i(L^\ast_\lambda(V))=0$ for $i>0$, and $H^0(L^\ast_\lambda(V)) \cong V_\lambda(2n-1)$. 
\end{cor}

\begin{proof}We have already explained that $H^0(L^\ast_\lambda(V(2n-1))) \cong V_\lambda(2n-1)$, so what needs to be shown is vanishing of cohomology in higher degrees. It suffices to prove this after restriction to $\Sp_{2n-2}$. But
	\[\operatorname{Res}^{\Sp_{2n-1}}_{\Sp_{2n-2}} H^i(L^\ast_\lambda(V(2n-1)) = H^i (\Sigma L^\ast_\lambda(E)(2n-1)),\]
	and \cref{thm: odd symplectic branching} expresses $L^\ast_\lambda(\Sigma E)(2n-1)$ as an iterated extension of complexes $L_\mu^\ast(E)(2n-1) = L_\mu^\ast(V(2n-2))$, each of which has no higher cohomology by \cref{samsnowdenweyman}, \cref{length lemma}, and our restriction on the length of $\lambda$. The result follows.
\end{proof}

\begin{defn}Let $V$ be a coefficient system. For $c \in \Z$, define a new coefficient systems $\tau_{\geq c}V$ by
\[ \tau_{\geq c}V(n) = \begin{cases} 0 & n < c \\ V(n) & n \geq c \end{cases} \]
and the evident structure maps. There is a natural map $\tau_{\geq c}V \to V$ for all $c$. \end{defn}

\begin{cor}\label{shift of Elambda}
	Let $\lambda$ be a partition, and let $m=l(\lambda)$. The coefficient system $\Sigma E_\lambda$ carries an increasing filtration $F$ such that $\operatorname{Gr}^F_k \Sigma E_\lambda \cong \bigoplus_{\lambda \searrow_k\mu} \tau_{\geq (2m-1) }E_\mu$.
\end{cor}

\begin{proof}We take the filtration $F$ on $\Sigma E_\lambda = H^0(\Sigma L^\ast_\lambda(E))$ to be the one induced from the filtration $F$ on $\Sigma L_\lambda^\ast(E)$ described in \cref{thm: odd symplectic branching}. We now need to compute its associated graded pieces. 	The spectral sequence associated with a filtered object reads
	\[ E^1_{pq} = H^{-p-q}(\operatorname{Gr}^F_{p} \Sigma L_\lambda^\ast E) = \bigoplus_{\lambda \searrow_{p} \mu} H^{-p-q}(L_\mu^\ast E) \implies H^{-p-q}(\Sigma L_\lambda^\ast(E)).\]
	If $n \geq 2m-1$ then $H^i (L_\mu^\ast E(n)) = 0$ for $i>0$ and $\lambda \searrow \mu$ by \cref{length lemma}, \cref{samsnowdenweyman}, and \cref{odd samsnowdenweyman}. It follows that in this range, the spectral sequence expresses $\Sigma E_\lambda$ as an iterated extension of coefficient systems $E_\mu$ for $\lambda \searrow \mu$. On the other hand, if $n < 2m-1$, then $\Sigma E_\lambda(n)=0$. 
\end{proof}

The same argument as in \cref{thm: odd symplectic branching} also allows us to describe the shifts $\Sigma L_\lambda^\ast(M)$. 

\begin{thm}\label{thm: symplectic branching}
	The coefficient system $\Sigma L^\ast_\lambda(M)$ carries two filtrations $F$ and $G$ such that 
	\[ \operatorname{Gr}^G_j \operatorname{Gr}^F_i \Sigma L^\ast_\lambda(M) \cong \bigoplus_{\lambda \searrow_i \mu} \bigoplus_{\mu \searrow_j \nu} L^\ast_\nu(M).\]
\end{thm}

\begin{proof}We proceed as in the proof of \cref{thm: odd symplectic branching}, but consider instead the short exact sequence
	\[ 0 \to M \to \Sigma M \stackrel p\longrightarrow \underline{\bk}\{e,f\} \to 0,\]
	where $\bk\{e,f\}$ denotes a fixed two-dimensional symplectic vector space. We define two-step filtrations $F$ and $G$ of $\Sigma M$ by $F_0 \Sigma V = p^{-1}\mathrm{Span}\{e\}$, and $G_0 \Sigma M = M$. The argument in the proof of \cref{thm: odd symplectic branching} shows that $\operatorname{Gr}^F_i \Sigma L^\ast_\lambda(M) \cong \bigoplus_\mu L^\ast_\mu(F_0 M)$. Repeating the argument again with the two-step filtration induced by $G$ on $F_0M$ gives the result. 
\end{proof}

\begin{cor}\label{shift of Vlambda}
	Let $\lambda$ be a partition, and let $m=l(\lambda)$. The coefficient system $\Sigma M_\lambda$ carries two increasing filtrations $F$ and $G$ such that $$\operatorname{Gr}^G_j \operatorname{Gr}^F_i \Sigma M_\lambda \cong \bigoplus_{\lambda \searrow_i\mu} \bigoplus_{\mu \searrow_j \nu} \tau_{\geq (m-1) }M_\nu.$$
\end{cor}

\begin{proof}This is the same argument as the proof of \cref{shift of Elambda}.
\end{proof}

\subsubsection{Littlewood complexes for general linear groups}

We will now define a version of the Littlewood complexes for the general linear groups. We permit ourselves to be brief, as the results are completely parallel to those for symplectic groups.

Let $W$ be a vector space. Consider the analytic functor of two variables
\[ (A,A')\mapsto \Psi_W(A,A') := \mathrm{Sym}(A\otimes W \oplus A' \otimes W^\ast \oplus A \otimes A' [-1])\]
from pairs of graded vector spaces to graded vector spaces. As in the symplectic case, $\Psi_W(A,A')$ is a cocommutative coalgebra. We define a coderivation on $\Psi_W(A,A')$ of degree 1, extending the map
\begin{gather*} \mathrm{Sym}(A\otimes W \oplus A' \otimes W^\ast \oplus A \otimes A' [-1]) \to A\otimes W[1] \oplus A' \otimes W^\ast[1] \oplus A \otimes A' \end{gather*}
given by the composite \begin{gather*}\mathrm{Sym}(A\otimes W \oplus A' \otimes W^\ast \oplus A \otimes A' [-1])  \to  \mathrm{Sym}^2(A\otimes W \oplus A' \otimes W^\ast \oplus A \otimes A' [-1]) \to \\
\to A \otimes W \otimes A' \otimes W^\ast \to A \otimes A' \hookrightarrow A\otimes W[1] \oplus A' \otimes W^\ast[1] \oplus A \otimes A'. \end{gather*}
The Taylor coefficients of $\Psi_W$ form a bisymmetric sequence. We let $K_{\lambda\mu}^\ast(W)$ be the $(\lambda,\mu)$-isotypical components of these Taylor coefficients. This is the \emph{dual Littlewood complex} for the general linear groups. Instead of taking $W$ to be a vector space, it can itself be a coefficient system. If $W$ is the coefficient system given by the defining representations of $\GL_n$, then $H^0(K_{\lambda\mu}^\ast(W)) =W_{\lambda,\mu}$, by an argument much like the one in the symplectic case. 

\begin{thm}[Sam--Snowden--Weyman]
	\label{samsnowdenweyman GL}Let $W$ be the defining representation of $\GL_n$, and suppose $n \geq l(\lambda)+l(\mu) -1$. Then $H^i(K_{\lambda\mu}(W)) = 0$ for $i>0$.
\end{thm}

We let $W$ be the coefficient system given by family of defining representations of $\GL_n(\Z)$. The short exact sequence of coefficient systems
\[ 0 \to W \to \Sigma W \to \underline \bk \to 0 \]
is split. Using this splitting, we obtain an short exact sequence $$0 \to W^\ast \to \Sigma W^\ast \to \underline \bk \to 0.$$
From this we get two 2-step filtrations on $\Psi_{\Sigma W}(A,A')$: one (denoted $F$) induced from the 2-step filtration of $\Sigma W$, and one (denoted $G$) from the 2-step filtration of $\Sigma W^\ast$. We have
\[ \operatorname{Gr}^F_i \operatorname{Gr}^G_j \Psi_{\Sigma W}(A,A') \cong \Psi_W(A,A') \otimes \mathrm{Sym}^i(A) \otimes \mathrm{Sym}^j(A'),\]
by arguments like those for \eqref{branching littlewood 2}. From this, arguments like those for \cref{thm: odd symplectic branching} and \cref{thm: symplectic branching} prove the following result.

\begin{thm}
	The coefficient system $\Sigma K^\ast_{\lambda\mu}(W)$ carries two filtrations $F$ and $G$ such that 
	\[ \operatorname{Gr}^G_j \operatorname{Gr}^F_i \Sigma K^\ast_{\lambda\mu}(W) \cong \bigoplus_{\lambda \searrow_i \nu} \bigoplus_{\mu \searrow_j \xi} K^\ast_{\nu\xi}(W).\]
	\end{thm}

Now arguments exactly like those for \cref{shift of Elambda} and \cref{shift of Vlambda} prove the following.

\begin{cor}\label{shift of Wlambda}
	Let $\lambda$ and $\mu$ be partitions, and let $m=l(\lambda)+l(\mu)$. The coefficient system $\Sigma W_{\lambda,\mu}$ carries two filtrations $F$ and $G$ such that 
	
	\[ \operatorname{Gr}^G_j \operatorname{Gr}^F_i \Sigma W_{\lambda,\mu} \cong \bigoplus_{\lambda \searrow_i \nu} \bigoplus_{\mu \searrow_j \xi} \tau_{\geq (m-1) }W_{\nu,\xi}.\]
\end{cor}

	\begin{rem}\label{rem: history of branching}

		\cref{shift of Wlambda} implies a branching rule for the inclusion $\GL_{n-1}\hookrightarrow \GL_{n}$: if $\lambda,\mu$ are partitions with $l(\lambda)+l(\mu)\leq n$, then 
		\[\operatorname{Res}^{\GL_{n}}_{\GL_{n-1}} W_{\lambda,\mu}(n) \cong \bigoplus_{\lambda \searrow \nu} \bigoplus_{\mu \searrow \xi} W_{\nu,\xi}(n-1).\]
		Indeed,	\cref{shift of Wlambda} expresses $\Sigma W_{\lambda,\mu}(n-1)$ as an iterated extension of representations $W_{\nu,\xi}(n-1)$, and the extension is necessarily a direct sum, since $\GL_{n-1}$ is reductive. This formula is classical, going back at least to Weyl. By the same argument, \cref{shift of Vlambda} implies a branching rule for the inclusion $\Sp_{2n-2}\hookrightarrow \Sp_{2n}$: if $\lambda$ is any partition with $l(\lambda)\leq n$, then 
		\[\operatorname{Res}^{\Sp_{2n}}_{\Sp_{2n-2}} V_\lambda(2n) \cong \bigoplus_{\lambda \searrow \mu} \bigoplus_{\mu \searrow \nu} V_\nu(2n-2).\]
		This formula was likely first obtained by Zhelobenko. 
		Finally, \cref{shift of Elambda}  implies a branching rule for $\Sp_{n-1} \hookrightarrow \Sp_n$. One finds that 
		$\operatorname{Res}^{\Sp_{2n}}_{\Sp_{2n-1}} V_\lambda(2n)$ carries a filtration $F$ such that 
		\[\operatorname{Gr}^F \operatorname{Res}^{\Sp_{2n}}_{\Sp_{2n-1}} V_\lambda(2n) \cong  \bigoplus_{\lambda \searrow \mu} V_\mu(2n-1), \]
		 a result due to Shtepin \cite{shtepin}. The filtration does not in general split; the extensions are nontrivial. Similarly, $$\operatorname{Res}^{\Sp_{2n-1}}_{\Sp_{2n-2}} V_\lambda(2n-1) \cong \bigoplus_{\lambda \searrow\mu} V_\mu(2n-2).$$ \end{rem}

	\subsubsection{Resolving $\Sp_n$-representations by $\Sp_{n+1}$-representations}

\newcommand{\E}{\mathfrak{E}}

In this subsection we use the symplectic Littlewood complexes to prove a dual form of \cref{thm: odd symplectic branching}. As in \cref{thm: odd symplectic branching}, the starting point is the short exact sequence
\[ 0 \to E \to \Sigma E \to \underline \bk \to 0. \]
But instead of considering it as defining a two-term filtration of $\Sigma E$ with graded pieces $E$ and $\underline \bk$, we can (from a derived perspective) just as well consider it as a two-term filtration of $E$ with graded pieces $\underline \bk[-1]$ and $\Sigma E$. The result is the following.

\begin{thm}\label{dual littlewood shift}
	The coefficient system $L^\ast_\lambda(E)$ admits a decreasing filtration $F$ in the derived category of coefficient systems, such that 
	\[ \operatorname{Gr}_F^k L^\ast_\lambda(E) \cong \bigoplus_{\lambda' \searrow_k \mu'} L^\ast_\mu(\Sigma E)[-k],\]
	where $\lambda'$ and $\mu'$ are the conjugate partitions of $\lambda$ and $\mu$, and $L^\ast_\mu(\Sigma E)[-k] = L^{\ast-k}_\mu(\Sigma E)$.
\end{thm}

\begin{proof}
	We can apply the Littlewood complex functor $L^\ast_\lambda(-)$ not just to a coefficient system, but to a \emph{chain complex} of coefficient systems. In particular, we may apply it to the two-term chain complex 
\[ \E := \big(\Sigma E \to \underline \bk \big) \]
with $\Sigma E$ placed in degree $0$ and $\underline \bk$ in degree $-1$. The symplectic form on $\Sigma E$ defines a skew-symmetric form $\omega: \wedge^2 \E \to \underline \bk$. Hence $L^\ast_\lambda(\E)$ is well-defined, and the quasi-isomorphism $E \stackrel{\sim}{\longrightarrow} \E$ induces a quasi-isomorphism $L^\ast_\lambda(E) \simeq L^\ast_\lambda(\E)$.

Consider, in general, a vector space $V$ with a skew-symmetric form $\omega$ and a \emph{decreasing} two-term filtration $V=F^0V \supset F^1 V \supset F^2V = 0$. 
Let $F^1V = W$. We are assuming that $\omega$ is compatible with the filtration, i.e.\ that $\omega$ vanishes on $W \otimes V$. Then $\omega$ descends to a form on $V/W$. As in the proof of \cref{thm: odd symplectic branching}, the filtration induces a filtration on $\Phi_V(A)$ for any $A$, and on $L^\ast_\lambda(V)$, but now it satisfies
$$ \operatorname{Gr}_F^{k} \Phi_V (A) = \operatorname{Sym}^k(W \otimes A) \otimes \Phi_{V/W}(A). $$
We may now give $\E$ a decreasing two-term filtration as above,
by setting $F^{1}\E = \underline \bk[-1] $. We find that 
\[ \operatorname{Gr}_F^{k} \Phi_\E(A) \cong \operatorname{Sym}^k( \underline \bk[-1] \otimes A) \otimes \Phi_{\Sigma E}(A)\cong \big(\!\wedge^k\!(A) \otimes \Phi_{\Sigma E}(A)\big)[-k].\]
We now finish as in the proof of \cref{thm: odd symplectic branching}, except replacing $\operatorname{Sym}^k$ with $\wedge^k$ leads to replacing $h_k$ (the character of the trivial representation of $\mathfrak S_k$) with $e_k$ (the sign character), and $e_k^\perp s_\lambda = \sum_{\lambda' \searrow_k\mu'} s_\mu$.  \end{proof}

\begin{cor}\label{dual littlewood shift corollary}Fix $\lambda$ a partition, and choose $n$ with $n \geq 2l(\lambda)-1$. There exists an exact sequence of $\Sp_n$-modules 
	$$ 0 \to V_\lambda(n) \to R^0 \to R^1 \to R^2 \to \dots $$
	with each $R^k$ the restriction of a representation of $\Sp_{n+1}$, and $R^k =0$ for $k> l(\lambda)$: explicitly, $R^k \cong \bigoplus_{\lambda' \searrow_k \mu'} V_\mu(n+1)$. \end{cor}

\begin{proof}We apply \cref{dual littlewood shift}. By our assumption on $n$, the complex $L^\ast_\lambda(E)(n+1)$ has cohomology only in degree $0$, where it is given by $V_\lambda(n)$. Each graded piece $\operatorname{Gr}_F^k L^\ast_\lambda(E)(n+1)$ has cohomology only in degree $k$. Hence these graded pieces assemble to a cochain complex of the stated form. To see that $R^k=0$ for $k>l(\lambda)$, note that $\lambda' \searrow_k \mu'$ means that $\mu$ can be obtained from the Young diagram of $\lambda$ by removing $k$ boxes, no two in the same row. In particular, $k$ must be at most the number of rows of the partition $\lambda$, i.e.\ $l(\lambda)$. \end{proof}

\begin{rem}
It is somewhat unexpected that the Littlewood complexes can be used to construct resolutions of $\Sp_n$-representations by $\Sp_{n+1}$-representations, given that their original construction is most naturally understood as a functorial resolution of $\Sp_n$-representations by $\GL_n$-representations. 
\end{rem}

\begin{rem}One may similarly use the Littlewood complexes for the general linear groups to construct a resolution of an  $\GL_n$-representation by a complex of  $\GL_{n+1}$-representations, but we will not have a need to do so.  \end{rem}
	
	\subsection{Twisted stability for even and odd symplectic groups}\label{subsec:kostant}
	
	For the groupoids $\coprod_{n \geq 0} \Sp_{2n}(\Z)$ and $\coprod_{n \geq 0} \GL_{n}(\Z)$, twisted homological stability follows from Borel's work. For the system of even and odd symplectic groups, twisted homological stability is less immediate. Our goal in this section is to prove  homological stability of the even and odd symplectic groups with $E_\lambda$ coefficients. Our strategy will be to use a theorem of Kostant to express the homology of the odd symplectic groups in terms of that of the even symplectic groups in a stable range. 
    
    We first make the observation that the inclusion $\Sp_{2n-2} \hookrightarrow \Sp_{2n-1}$ always splits. Intrinsically (but not canonically), the splitting is given by projection to the Levi factor; in coordinates, $\Sp_{2n-1}$ is the subgroup of $\Sp_{2n}$ of block-matrices of the form 
	\[\begin{bmatrix} A & * & 0 \\
	0 & 1 & 0 \\
	* & * & 1	
\end{bmatrix}\]
where $A \in \Sp_{2n-2}$, and the projection $f:\Sp_{2n-1}\to \Sp_{2n-2}$ takes such a matrix to $A$. Let $U_{2n-1} = \ker(f)$ be the unipotent radical of $\Sp_{2n-1}$, the $(2n-1)$-dimensional Heisenberg group. 

\begin{prop}\label{firstreduction}
	Let $\Gamma'$ be an arithmetic subgroup of $U_{2n-1}$, and $\lambda$ a partition.  Consider $V_\lambda(2n)$ as a representation of $U_{2n-1}$ via the inclusions $U_{2n-1} \hookrightarrow \Sp_{2n-1} \hookrightarrow \Sp_{2n}$. The homology groups $H_\ast(\Gamma';V_\lambda(2n))$ are representations of $\Sp_{2n-2}$, and the multiplicities of the trivial representation are given by 
	\[ \dim H_d(\Gamma';V_\lambda(2n))^{\Sp_{2n-2}} = \begin{cases}
		 1 & l(\lambda) \leq 1 \text{ and } d \in \{0,2n-1\} \\
		 0 & \text{otherwise.}
	\end{cases}\]
\end{prop}

The proof of \cref{firstreduction} is a computation using Kostant's theorem in Lie algebra cohomology. It suffices to prove  \cref{firstreduction} in the case that the ground field $\bk$ from which we take our coefficients is the complex numbers $\mathbb C$. We assume this for the rest of the proof of \cref{firstreduction}, to be consistent with the literature on Lie algebra cohomology.
	
	\subsubsection{Kostant's theorem}
	
	Let $\mathfrak g$ be a finite dimensional complex reductive Lie algebra. Let $\mathfrak p \subset \mathfrak g$ be a parabolic subalgebra, with nilpotent radical $\mathfrak u$ and Levi subalgebra $\mathfrak l$. Let $V(\lambda)$ be the finite-dimensional irreducible representation of $\mathfrak g$ of highest weight $\lambda$. Kostant's theorem \cite[Theorem 5.14]{kostant} computes the Lie algebra cohomology
	$ H^q(\mathfrak u;V(\lambda))$
	as a representation of $\mathfrak l$. See also e.g.\ Vogan \cite[Chapter 3, \S2]{vogan} for a textbook treatment.
	
	We let $W$ be the Weyl group of $\mathfrak g$, $W_P$ the Weyl group of $\mathfrak l$, and $W^P$ the set of distinguished coset representatives of $W/W_P$. We denote by $\rho$ the half-sum of all positive roots of $\mathfrak g$. For $w \in W$ and $\lambda$ a weight of $\mathfrak g$, we define
	$$ w \bullet \lambda = w(\lambda+\rho)-\rho.$$
	If $w \in W^P$, then $w^{-1} \bullet (-)$ takes dominant weights of $\mathfrak g$ to dominant weights of $\mathfrak l$. We denote by $\ell : W \to \Z_{\geq 0}$ the length function of $W$. 
	
	\begin{thm}[Kostant] With notation as above we have
		$$ H^q(\mathfrak u;V(\lambda)) = \bigoplus_{\substack{w \in W^P \\ \ell(w)=q}} V_{\mathfrak l}(w^{-1}\bullet\lambda),$$
		where $V_{\mathfrak l}(\mu)$ denotes the irreducible representation of $\mathfrak l$ of highest weight $\mu$. 	
	\end{thm}
	
	\begin{proof}[Proof of \cref{firstreduction}]
    Let $G= \Sp_{2n}$, $P \subset G$ the parabolic subgroup stabilizing a line, and $L$ its Levi factor $\Sp_{2n-2} \times \mathbb G_m$. Note that we have 
    $$ 1 \to \Sp_{2n-1} \to P \to \mathbb G_m \to 1,$$
    and that the unipotent radical of $P$ is $U_{2n-1}$. We apply Kostant's theorem, with $\mathfrak g, \mathfrak p, \mathfrak l$, $\mathfrak u$ the Lie algebras of $G$, $P$, $L$, and $U_{2n-1}$.

	First of all, by Nomizu's theorem \cite{nomizu}, we have $H_*(\Gamma';V_\lambda(2n))=H_*(\mathfrak u;V_\lambda(2n))$.

	All algebraic representations of the usual symplectic group are self-dual (this fails for the odd symplectic groups). Applied to $\mathfrak g$ and $\mathfrak l$ this means in particular that Kostant's theorem in this case dualizes to a homology isomorphism $$H_q(\mathfrak u;V(\lambda)) = \bigoplus_{\substack{w \in W^P \\ \ell(w)=q}} V_{\mathfrak l}(w^{-1}\bullet\lambda).$$
	We have that $W$ is the hyperoctahedral group, i.e. the group of signed permutations, on $n$ letters. The vector $\rho$ is $(n,n-1,\ldots,1)$, the smallest regular weight of $\mathrm{Sp}_{2n}$. The subgroup $W_P$ is the group of signed permutations on $n-1$ letters, which we think of as the subgroup stabilising the first entry. We have $w \in W^P$ precisely when 
	$$ w^{-1}(\rho)_2 > w^{-1}(\rho)_3 > \dots > w^{-1}(\rho)_{n} > 0.$$
	Thus for $w\in W^P$, $w^{-1}(\rho)_1$ can be any element of $\{\pm 1,\pm 2,\ldots,\pm n\}$ and the remaining entries are uniquely determined by the condition of being positive and in descending order.  For $\lambda$ a dominant weight of $\Sp_{2n}$ and $w\in W^P$, the dominant weight of $\Sp_{2n-2}$ given by $w^{-1}\bullet \lambda$ is simply the vector obtained by deleting the first entry of the vector $w^{-1}\bullet \lambda$. 
	
	\begin{example}\label{kostant example}
		Take $n=5$. Let $w$ be the unique element of $W^P$ with $w^{-1}(\rho)_1=-3$. Then we will have 
		$$ w^{-1}(\rho) = (-3,5,4,2,1)$$
		and for $\lambda = (\lambda_1,\lambda_2,\lambda_3,\lambda_4,\lambda_5)$ we have
		$$ w^{-1}\bullet\lambda = (-8-\lambda_3,1+\lambda_1,1+\lambda_2,\lambda_4,\lambda_5).$$
		The corresponding dominant weight of $\mathrm{Sp}_{8}$ is obtained by deleting the first entry in the vector, giving $(1+\lambda_1,1+\lambda_2,\lambda_4,\lambda_5)$. \hfill $\blacksquare$
	\end{example}
	
	To finish the proof of \cref{firstreduction}, we need to determine for which values of $\lambda$ and $w$ the weight $w^{-1}\bullet\lambda$ gives the trivial representation of $\Sp_{2n-2}$, i.e\ the all zero vector of length $(n-1)$. By pondering \cref{kostant example}, it is clear that this happens precisely when $ \lambda = (\lambda_1,0,\ldots,0)$ and moreover
	$$ w^{-1}(\rho)=(n,n-1,n-2,\ldots,1) \qquad \text{or}\qquad w^{-1}(\rho)=(-n,n-1,n-2,\ldots,1).$$
	That is, we must have $l(\lambda)\leq 1$ and $w\in \{\mathrm{id},\theta\}$, where $\theta$ denotes the element that swaps the sign of the first entry. We have $\ell(\mathrm{id})=0$ and $\ell(\theta)=2n-1$. The result follows.\end{proof}

\begin{rem}
    One can also give an explicit general formula for $w^{-1} \bullet \lambda$. If $w^{-1}(\rho) = (a_1,\dots,a_n)$, then $w^{-1}(\lambda) = (\mathrm{sgn}(a_1)\lambda_{n+1-|a_1|},\dots, \mathrm{sgn}(a_n)\lambda_{n+1-|a_n|} ) $. Hence also
    \[ w^{-1} \bullet \lambda = (\mathrm{sgn}(a_1)\lambda_{n+1-|a_1|} + a_1 - n,\dots, \mathrm{sgn}(a_n)\lambda_{n+1-|a_n|} +a_n - 1).\]
    We may also derive from this formula the fact that if $w \in W^P$, then $w^{-1}\bullet\lambda$ gives the trivial representation of $\Sp_{2n-2}$ if and only if $l(\lambda) \leq 1$ and $w \in \{\mathrm{id},\theta\}$. Indeed if $w \in W^P$ then $a_i>0$ for $i>1$, so there must exist $i>1$ with $a_i \in \{n,n-1\}$. Then the $i$th entry of $w^{-1} \bullet \lambda$ is $\lambda_{n+1-a_i} + a_i - (n+1-i)$, which is strictly positive unless $i=2$, $a_i=n-1$, and $\lambda_2=0$. This gives the result.  
\end{rem}

	%

\begin{prop}
	\label{secondreduction}
	Let $\Gamma'$ be an arithmetic subgroup of $U_{2n-1}$, and $\lambda$ a partition. Consider $V_\lambda(2n-1)$ as a representation of $U_{2n-1}$. The homology groups $H_\ast(\Gamma';V_\lambda(2n-1))$ are representations of $\Sp_{2n-2}$. If $d<2n-1-l(\lambda)$, then the multiplicities of the trivial representation are given by 
	\[ \dim H_d(\Gamma';V_\lambda(2n-1))^{\Sp_{2n-2}} = \begin{cases}
		1 & \lambda = (0) \text{ and } d =0 \\
		0 & \text{otherwise.}
	\end{cases}\]
\end{prop}

\begin{proof}Consider the resolution 
	$$ V_\lambda(2n-1) \to R^0 \to R^1 \to R^2 \to \dots $$
	constructed in \cref{dual littlewood shift corollary}. We get an associated second quadrant hyperhomology spectral sequence
\[ E^1_{pq} = H_q(\Gamma';R^{-p}) \implies H_{p+q}(\Gamma';V_\lambda(2n-1)).\]	
By semisimplicity, we also get a spectral sequence of $\Sp_{2n-2}$-invariants: 
\[ E^1_{pq} = H_q(\Gamma';R^{-p})^{\Sp_{2n-2}} \implies H_{p+q}(\Gamma';V_\lambda(2n-1))^{\Sp_{2n-2}}.\]	
	
Since $R^k=0$ unless $0 \leq k \leq l(\lambda)$, and each $R^k$ is an algebraic representation of $\Sp_{2n}$, \cref{firstreduction} implies that the latter spectral sequence is nonzero on the $E^1$-page only for $q \in \{0,2n-1\}$ and $-l(\lambda)\leq p \leq 0$. In particular, all nonzero entries along the row $q=2n-1$ contribute to homology above degree $2n-1-l(\lambda)$.

It thus suffices to consider the row $q=0$. Since the abutment has no homology in negative degrees, the only class which could potentially survive is 
$$E^1_{00} = H_0(\Gamma';R^0)^{\Sp_{2n-2}} = H_0(\Gamma';V_\lambda(2n))^{\Sp_{2n-2}} \cong \begin{cases}
	\bk & l(\lambda) \leq 1 \\0 & \text{otherwise},
\end{cases}$$
using \cref{firstreduction} in the last step. Hence if $l(\lambda)>1$ there is nothing to prove along the row $q=0$. 
On the other hand, let us suppose that $l(\lambda)=1$, so that $\lambda = (k)$ and $V_{(k)}(n) = \mathrm{Sym}^k\, V(n)$. Then the resolution from \cref{dual littlewood shift corollary} reduces to a short exact sequence
$$ 0 \to V_{(k)}(2n-1) \to R^0 \to R^1 \to 0 $$
with $R^0 = V_{(k)}(2n)$ and $R^1 = V_{(k-1)}(2n)$. (The associated hyperhomology spectral sequence reduces in this case to the long exact sequence in homology.) There are exactly two nonzero entries along the row $q=0$, given by $H_0(\Gamma';R^0) \cong H_0(\Gamma';R^1) \cong \bk$. The $E^1$-differential $E^1_{00} \to E^1_{-1,0}$ must be an isomorphism as the class in $E^1_{-1,0}$ cannot survive to $E^\infty$.  
\end{proof}

\begin{thm}\label{odd stability}
	Let $\Gamma$ be an arithmetic subgroup of $\Sp_{2n-1}$, and $\lambda$ a partition. If $\lambda$ is nontrivial, then $H_d(\Gamma;V_\lambda(2n-1))=0$ for $d<n-1$. 
	If $\lambda = (0)$, then $\dim H_d(\Gamma;V_\lambda(2n-1)) \cong \dim H_d(\Gamma'';\bk)$ for $d\leq n-1$, where $\Gamma''$ is the image of $\Gamma$ in the Levi factor $\Sp_{2n-2}$.
\end{thm}

\begin{proof}
	From the short exact sequence $1 \to U_{2n-1} \to \Sp_{2n-1} \to \Sp_{2n-2} \to 1$ we get a short exact sequence of discrete groups 
	$$1 \to \Gamma' \to \Gamma \to \Gamma'' \to 1$$ and a Lyndon--Hochschild--Serre spectral sequence 	\[
	E^2_{pq} = H_p(\Gamma''; H_q(\Gamma';V_\lambda(2n-1)) \implies H_{p+q}(\Gamma;V_\lambda(2n-1)).
	\]
	Suppose first $\lambda \neq (0)$. Then the trivial representation of $\Sp_{2n-2}$ is not a summand of $H_q(\Gamma';V_\lambda(2n-1))$ for $q<2n-1-l(\lambda)$ by \cref{secondreduction}, hence in particular not for $q<n-1$. Using Borel's vanishing result (\cref{thm:borel-sp}) we see that $E^2_{p,q}=0$ in the Lyndon--Hochschild--Serre spectral sequence for all $p<n-1$ and $q<n-1$, since $\Gamma''$ is an arithmetic subgroup of $\Sp_{2n-2}$.  We therefore see that $H_d(\Gamma;V_\lambda(2n-1)) $ vanishes for $d <n-1$.
	
	If $\lambda = (0)$ then our computation with Kostant's theorem\footnote{Or, more elementarily, the fact that we are computing $H_0$ and $H_{2n-1}$ of the compact oriented $(2n-1)$-manifold $U_{2n-1}(\R)/\Gamma'$.} shows  $E^2_{p,q} \cong H_p(\Gamma'';\bk)$ for $q \in \{0,2n-1\}$ and all $p \geq 0$. By \cref{secondreduction} and Borel stability we have $E^2_{pq}=0$ for $q \notin \{0,2n-1\}$, for all $p<n-1$. Again the result follows. 
\end{proof}

	\begin{cor}\label{odd stability constant coeffs}
	For each $n$, let $\Gamma_n$ be an arithmetic subgroup of $\Sp_n$, such that $\Gamma_{n-1} \subset \Gamma_n$. Then $H_d(\Gamma_n,\Gamma_{n-1};\bk)=0$ for $d<\tfrac 1 2 n$. 
	\end{cor}
	
	\begin{proof}
		The family of homomorphisms \eqref{matsushima} are compatible with each other as $n$ varies, which together with \cref{thm:borel-sp} implies that \begin{equation}
			\label{eq: two stabilisations}H_d(\Gamma_{2n-2};\bk) \to H_d(\Gamma_{2n};\bk)\end{equation} is an isomorphism for $d<n-1$. From \cref{rem:matsushima} we also see that \eqref{eq: two stabilisations} is surjective for $d=n-1$. Let $f:\Sp_{2n-1}\to\Sp_{2n-2}$ be the projection, and set $\Gamma_{2n-2}' = f(\Gamma_{2n-1})$. The composite  
		\[ H_d(\Gamma_{2n-2};\bk)\to H_d(\Gamma_{2n-1};\bk) \to H_d(\Gamma_{2n-2}';\bk)\]
		is surjective for all $d$ since $\Gamma_{2n-2} \subseteq \Gamma_{2n-2}'$ is a finite index subgroup, and an isomorphism for $d<n-1$ by \cref{thm:borel-sp} since the homology in the stable range is independent of the arithmetic group. By \cref{odd stability} we then conclude that $H_d(\Gamma_{2n-2};\bk)\to H_d(\Gamma_{2n-1};\bk)$ is an isomorphism for $d<n-1$ and surjective for $d=n-1$.
	\end{proof}

	\subsection{Explicit stable ranges for symplectic and general linear groups}
	
	We can now put things together and verify Axiom \ref{axiomIIbis} in the respective situations of Theorems \ref{thmA}, \ref{thmB}, \ref{thmH}, and \ref{thmC}. Let us first record the behaviour of the various coefficient systems under iterated shifts.

	\begin{prop}\label{iterated extension}
		Let $\lambda$ and $\mu$ be partitions. Let $a \in \N$. 
		\begin{enumerate}
			\item The coefficient system $\Sigma^a E_\lambda$ is an iterated extension of coefficient systems $\tau_{\geq c}E_\nu$, where $\nu$ satisfies
			\[ \lambda = \lambda^{(0)} \searrow \lambda^{(1)} \searrow \dots \searrow \lambda^{(a)} = \nu\]
			and $c = \max_i \{2l(\lambda^{(i)})-a+i \} $.
			\item The coefficient system $\Sigma^a M_\lambda$ is an iterated extension of coefficient systems $\tau_{\geq c}M_\nu$, where $\nu$ satisfies
			\[ \lambda = \lambda^{(0)} \searrow \lambda^{(1)} \searrow \dots \searrow \lambda^{(2a)} = \nu\]
			and $c = \max_i \{l(\lambda^{(2i)})-a+i \} $.
			\item The coefficient system $\Sigma^a W_{\lambda\mu}$ is an iterated extension of coefficient systems $\tau_{\geq c} W_{\nu\xi}$, where 
			\[\lambda = \lambda^{(0)} \searrow \lambda^{(1)} \searrow \dots \searrow \lambda^{(a)} = \nu \qquad \text{and} \qquad \mu = \mu^{(0)} \searrow \mu^{(1)} \searrow \dots \searrow \mu^{(a)} = \xi\]
			and $c= \max_i \{ l(\lambda^{(i)}) + l(\mu^{(i)})-a+i\}$. 
		\end{enumerate}
	\end{prop}

\begin{proof}
	This follows by induction on $a$ using \cref{shift of Elambda}, \cref{shift of Vlambda}, and \cref{shift of Wlambda}, respectively. Note that if $V$ is a coefficient system, then $\Sigma \tau_{\geq c} V = \tau_{\geq c-1} \Sigma V$. 
\end{proof}

\subsubsection{The symplectic case}

\begin{lem}\label{elementary lemma}
	Suppose given real numbers $m_0,m_1,\dots,m_a$ such that $|m_i-m_{i-1}|\leq 1$ for $i=1,\dots,a$. Then $\max_i m_i \leq \tfrac 1 2(m_0+m_a+a)$.
\end{lem}

\begin{proof}For each $i$, take the average of the inequalities $m_i \leq m_0 +i$ and $m_i \leq m_a + (a-i)$.
\end{proof}

\begin{thm} \label{odd symplectic stable range} Let $\mathsf Q$ be a monoidal subgroupoid of $\coprod_n \Sp_{n-1}(\Q)$ consisting of arithmetic subgroups, and let $\lambda$ be a partition. Then 
	$$H_k(Q_{n},Q_{n-1}; \Sigma^a E_\lambda(n),\Sigma^a E_\lambda(n-1))=0$$
	if $k < \tfrac 1 2 (n-2)$ and $n>l(\lambda)$. That is, the coefficient system $E_\lambda$ satisfies Axiom \ref{axiomIIbis} of \cref{cor:GenCor2} with $\theta = \tfrac 1 2$, $\tau = 2$, and $\beta = l(\lambda)$.
	
\end{thm}

\begin{proof}We use \cref{iterated extension} to write $\Sigma^a E_\lambda$ as an iterated extension of coefficient systems $\tau_{\geq c}E_\nu$, where \[ \lambda = \lambda^{(0)} \searrow \lambda^{(1)} \searrow \dots \searrow \lambda^{(a)} = \nu\]
	and $c = \max_i \{2l(\lambda^{(i)})-a+i \} $. It suffices to show that all such coefficient systems $\tau_{\geq c}E_\nu$ satisfy 
	\begin{equation}
		\label{eq:vanishing}H_k(Q_{n},Q_{n-1}; \tau_{\geq c}E_\nu(n),\tau_{\geq c}E_\nu(n-1))=0.
	\end{equation}     
	If $\nu$ is nontrivial, then $H_k(Q_j;E_\nu(j))=0$ for $k<\tfrac 1 2(j-2)$ by \cref{thm:borel-sp} (if $j$ is odd) and \cref{odd stability} (if $j$ is even). Thus, if $k<\tfrac 1 2 (n-2)$, then \eqref{eq:vanishing} holds regardless of the value of $c$. 
	
	So suppose that $\nu=0$. By \cref{length lemma}, each element of the sequence
	\[ 2l(\lambda^{(0)})-a, \, 2l(\lambda^{(1)})-a+1, \, 2l(\lambda^{(2)})-a+2,\, \dots \, , \, 2l(\lambda^{(a)})\]
	differs from the preceding element by $\pm 1$. Hence by \cref{elementary lemma} we get $c = \max_i \{2l(\lambda^{(i)})-a+i \} \leq l(\lambda)$. Since $n > l(\lambda)$ we may therefore disregard the truncation, and then \eqref{eq:vanishing} holds if $k < \tfrac 1 2 (n-1)$ by \cref{odd stability constant coeffs}. \end{proof}

This is the last ingredient needed to apply \cref{thm:GenThm} to deduce \cref{thmC}, uniform homological stability for braid groups. This then implies \cref{thmD} on moments of families of quadratic $L$-functions. 
	
	\begin{thm} \label{even symplectic stable range} Let $\mathsf Q$ be a monoidal subgroupoid of $\coprod_n \Sp_{2n}(\Q)$ consisting of arithmetic subgroups, and let $\lambda$ be a partition. Then 
		$$H_k(Q_{n},Q_{n-1}; \Sigma^a M_\lambda(n),\Sigma^a M_\lambda(n-1))=0$$
		if $k<n$ and $n>\tfrac 1 2 l(\lambda)$.  That is, the coefficient system $M_\lambda$ satisfies Axiom \ref{axiomIIbis} of \cref{cor:GenCor2} with $\theta = 1$, $\tau = 0$, and $\beta =  \tfrac 1 2 l(\lambda)$.
	\end{thm}

\begin{proof}The proof is the same as the previous one, using instead \cref{iterated extension} to write $\Sigma^a M_\lambda$ as an iterated extension of coefficient systems $\tau_{\geq c}M_\nu$, where \[ \lambda = \lambda^{(0)} \searrow \lambda^{(1)} \searrow \dots \searrow \lambda^{(2a)} = \nu,\] and that each element of the sequence
	\[ l(\lambda^{(0)})-a, \, l(\lambda^{(2)})-a+1, \, l(\lambda^{(4)})-a+2,\, \dots \, ,\, l(\lambda^{(2a)}) \]
	differs from the preceding element at most by $1$.
\end{proof}

This is the last ingredient needed to apply \cref{thm:GenThm} to deduce \cref{thmA}, uniform homological stability for mapping class groups. 

\subsubsection{The general linear group case}

We now want to prove an analogue of \cref{odd symplectic stable range} and \cref{even symplectic stable range} for the general linear groups and the coefficient systems $W_{\lambda\mu}$.

\begin{defn}
	If $\lambda$ and $\mu$ are partitions, then we define the function
	\[N_{\lambda\mu}(n) := \min \left\{ a  \;\middle|\;
	\begin{aligned}
		& \smash{\operatorname{Res}^{\GL_{n}}_{\GL_{n-a}}W_{\lambda\mu}(n)} \text{ contains a summand}\\
		& \text{ isomorphic to \smash{$\mathrm{det}^{\otimes 2k}$} for $k \in 2\Z\setminus \{0\}$}
	\end{aligned}
	\right\}. \]
	If the above set is empty --- this happens if and  only if $W_{\lambda\mu}(n)=0$, or $W_{\lambda\mu}(n)$ is an odd tensor power of the determinant representation --- then $N_{\lambda\mu}(n)=\infty$.
\end{defn}

\begin{lem}\label{lem:contains trivial summand}Let $0 \leq a < n$, and let $k>0$. Let $\lambda$ and $\mu$ be partitions with $l(\lambda)+l(\mu)\leq n$. Then the representation $\operatorname{Res}^{\GL_n}_{\GL_{n-a}} W_{\lambda\mu}(n)$ contains $\det^{\otimes k}$ as a summand if and only if $\lambda_{1+a} \leq k$ and $\lambda_{n-a} \geq k$, and $\operatorname{Res}^{\GL_n}_{\GL_{n-a}} W_{\lambda\mu}(n)$ contains $\det^{\otimes -k}$ as a summand if and only if $\mu_{1+a} \leq k$ and $\mu_{n-a} \geq k$.
\end{lem}

\begin{proof}It suffices to prove the first assertion, as the second follows by taking duals.  We do induction on $a$, the base case $a=0$ being that $W_{\lambda\mu}(n) \cong \det^{\otimes k}$ if and only if $\lambda = (k^n)$ and $\mu = (0)$. 
	
In general, $\operatorname{Res}^{\GL_n}_{\GL_{n-a}} W_{\lambda\mu}(n)$ contains $\det^{\otimes k}$ as a summand, if and only if there exist $(\nu,\xi)$ with $\lambda \searrow \nu$ and $\mu \searrow \xi$ and $\operatorname{Res}^{\GL_{n-1}}_{\GL_{n-a}} W_{\nu\xi}(n-1)$ contains $\det^{\otimes k}$ as a summand (using the branching rule of \cref{rem: history of branching}), if and only there exists $\nu$ with $\lambda \searrow \nu$ and $\nu_{a} \leq k$ and $\nu_{n-a} \geq k$ (by the induction hypothesis), if and only if $\lambda_{1+a} \leq k$ and $\lambda_{n-a} \geq k$ (since $\nu_i$ will take on all values between $\lambda_{i+1}$ and $\lambda_{i}$). \end{proof}

\begin{cor}\label{corollary of lemma}
	$N_{\lambda\mu}(n) \geq \min(n-\lambda_2',n-\mu_2')$, where $\lambda'$ and $\mu'$ denote the conjugate partitions of $\mu$ and $\lambda$. 
\end{cor}

\begin{proof}
	Suppose that $\operatorname{Res}^{\GL_n}_{\GL_{n-a}} W_{\lambda\mu}(n)$ contains $\det^{\otimes 2k}$ as a summand. Suppose first that $k>0$. By \cref{lem:contains trivial summand} we must have $\lambda_{n-a} \geq 2k \geq 2$, so $n-a \leq \lambda_2'$. If $k<0$ then we have instead $\mu_{n-a} \geq 2k \geq 2$, so $\mu_{n-a} \geq 2$, so $n-a \leq \mu_2'$. 
\end{proof}

\begin{thm}\label{GL stable range}Let $\mathsf Q$ be a monoidal subgroupoid of $\coprod_n \GL_{n}(\Q)$ consisting of arithmetic subgroups, such that\footnote{This hypothesis is not essential, but it ensures that we are in the situation of \cref{thm:borel-gl} and that precisely the even tensor powers of the determinant have nonzero homology in the stable range.} $Q_1 = \{\pm 1\}$. Let $\lambda$ and $\mu$ be partitions. Then $$H_k(Q_{n},Q_{n-1};\Sigma^a W_{\lambda\mu}(n), \Sigma^a W_{\lambda\mu}(n-1))=0$$
	if $k<n-1$ and $n > \max(\tfrac 1 2 (l(\lambda)+l(\mu)),1+\lambda_2',1+\mu_2')$. That is, the coefficient system $W_{\lambda\mu}$ satisfies Axiom \ref{axiomIIbis} of \cref{cor:GenCor2} with $\theta = 1$, $\tau = 1$, and $\beta = \max(\tfrac 1 2 (l(\lambda)+l(\mu)),1+\lambda_2',1+\mu_2')$.
\end{thm}	

\begin{proof}Note first that $a<N_{\lambda\mu}(n+a)$ and $a<N_{\lambda\mu}(n-1+a)$, using \cref{corollary of lemma} and the assumption that $n > 1+\max(\lambda_2',\mu_2')$. In other words, neither $\Sigma^a W_{\lambda\mu}(n)$ nor $\Sigma^a W_{\lambda\mu}(n-1)$ can contain a summand isomorphic to $\det^{\otimes 2k}$ for $k \neq 0$.

It is now the same argument as in the symplectic case. We use \cref{iterated extension} to write $\Sigma^a W_{\lambda\mu}$ as an iterated extension of coefficient systems $\tau_{\geq c} W_{\nu\xi}$, where 
	\[\lambda = \lambda^{(0)} \searrow \lambda^{(1)} \searrow \dots \searrow \lambda^{(a)} = \nu \qquad \text{and} \qquad \mu = \mu^{(0)} \searrow \mu^{(1)} \searrow \dots \searrow \mu^{(a)} = \xi\]
	and $c= \max_i \{ l(\lambda^{(i)}) + l(\mu^{(i)})-a+i\}$. It suffices to show that all such coefficient systems $\tau_{\geq c}W_{\nu\xi}$ satisfy 
	\begin{equation}
		\label{eq:vanishing 2}H_k(Q_{n},Q_{n-1}; \tau_{\geq c}W_{\nu\xi}(n),\tau_{\geq c}W_{\nu\xi}(n-1))=0.
	\end{equation}     
	If $(\nu,\xi) \neq (0,0)$, then $H_k(Q_{n};W_{\nu\xi}(n))=0$ for $k<n-1$ by \cref{thm:borel-gl}, since we have arranged that $W_{\nu\xi}(n)$ is not an even tensor power of the determinant. Similarly $H_k(Q_{n-1};W_{\nu\xi}(n-1))=0$ for $k<n-2$. Thus \eqref{eq:vanishing 2} holds in this case, regardless of the value of $c$. If $(\nu,\xi)=(0,0)$, then as before we note that each element of the sequence
	\[ l(\lambda^{(0)})+l(\mu^{(0)})-a, \, l(\lambda^{(1)})+l(\mu^{(1)})-a+1, \, l(\lambda^{(2)})+l(\mu^{(2)})-a+2,\, \dots \, , \, l(\lambda^{(a)})+l(\mu^{(a)})\]
	differs from the preceding element by at most $1$, so that  $c \leq \tfrac 1 2 (l(\lambda)+l(\mu))$. But then we may disregard the truncation, and \eqref{eq:vanishing 2} holds by \cref{thm:borel-gl}.
\end{proof}

This is the last ingredient needed to apply \cref{thm:GenThm} to deduce \cref{thmB} and \cref{thmH}, uniform homological stability for automorphism groups of free groups and mapping class groups of handlebodies. 

\subsubsection{Uniform stability in the general linear group cases}\label{uniform GL}
Let us finally explain in what sense \cref{thmB} and \cref{thmH} are theorems of ``uniform homological stability''. At this point, it will be convenient to use a different  convention for how to parametrise representations of $\GL_n$, in terms of  tuples $(w_1,\dots,w_n) \in \Z^n$ such that $w_1 \geq w_2 \geq \dots \geq w_n$.  To such a tuple $w$ we associate a pair of partitions $(\lambda,\mu)$, where $\lambda$ consists of the positive entries of the tuple $w$, and $\mu$ consists of minus the negative entries of $w$, written in reverse order. This assignment sets up a bijection between  $(w_1,\dots,w_n) \in \Z^n$ such that $w_1 \geq w_2 \geq \dots \geq w_n$, and pairs of partitions $(\lambda,\mu)$ with $l(\lambda)+l(\mu) \leq n$. We call $(w_1,\dots,w_n)$ the \emph{highest weight} of a representation. If the irreducible representation $V$ has highest weight $(w_1,\dots,w_n)$, then  the determinant twist $V \otimes \mathrm{det}^{\otimes k}$ has highest weight $(w_1+k,\dots,w_n+k)$, for any $k \in \Z$. 

Let $V$ be an irreducible representation of $\GL_n$	of highest weight $(w_1,\dots,w_n)$. Assume first that $n$ is odd. Choose $k \in \Z$ such that $w_{(n+1)/2} +2k \in \{0,1\}$. Let $\lambda$ and $\mu$ be the partitions such that $W_{\lambda\mu}(n) \cong V \otimes \det^{\otimes 2k}$ as representations of $\GL_n$. Then necessarily $\max(\lambda_2',\mu_2') \leq (n-1)/2$. An analogous argument when $n$ is even gives a determinant twist for which $\max(\lambda_2',\mu_2') \leq n/2$. Applying \cref{thmB} iteratively, we find that the stabilisation homomorphism
$$ H_d(\mathrm{Aut}(F_n);V) \cong H_d(\mathrm{Aut}(F_n);W_{\lambda\mu}(n))\lra H_d(\mathrm{Aut}(F_\infty);W_{\lambda\mu})$$
is surjective for $d= \lfloor \tfrac{n-1}{4} \rfloor$ and an isomorphism below this degree.  Hence the homology $H_\ast(\mathrm{Aut}(F_n);V)$ is completely determined, as $V$ ranges over all algebraic representation of $\GL_n$, in terms of the stable homology (as computed by Lindell \cite{lindell}), in a range growing linearly with $n$. The same argument produces the same sort of uniformity for the twisted homology of the handlebody mapping class groups, using \cref{thmH}.

\bibliographystyle{alpha}
\bibliography{database}

\end{document}